\documentclass{birkjour}
\PassOptionsToPackage{linktocpage}{hyperref}

\usepackage{schur-com}
\usepackage{amssymb}
\usepackage{float}
\usepackage{mathscinet}
\usepackage{hyperref}
\usepackage{pmat}
\usepackage{tikz}
\usetikzlibrary{decorations.markings,arrows}
\newtheorem{thm}{Theorem}[section]
\newtheorem{cor}[thm]{Corollary}

\theoremstyle{definition}
\newtheorem{defn}[thm]{Definition}
\theoremstyle{remark}

\numberwithin{equation}{section}

\begin{document}

%
%
%
%
%
%
%
%
%

\title[Characteristic Function, Schur Problem and Darlington Synthesis]{Characteristic Function, Schur Interpolation Problem and Darlington Synthesis}

\author[Boiko]{Sergey S. Boiko}

\address{Department of Mathematics and Computer Science\\
Karazin National University\\
Kharkov\\
Ukraine}

\email{ss\_boiko@ukr.net.ch}

\author[Dubovoy]{Vladimir K. Dubovoy}
\address{Department of Mathematics and Computer Science\\
    Karazin National University\\
    Kharkov\\
    Ukraine\\
    and\\
    Max Planck Institute for Human Cognitive and Brain Sciences\\
    Stephanstrasse~1A\\
    04103~Leipzig\\
    Germany\\
    and\\
    Max Planck Institute for Mathematics in the Sciences\\
    Inselstrasse~22\\
    04103~Leipzig\\
    Germany}
\email{dubovoy.v.k@gmail.com}

\author[Fritzsche]{Bernd Fritzsche}
\address{Universit\"at Leipzig\\
    Fakult\"at f\"ur Mathematik und Informatik\\
    PF~10~09~20\\
    D-04009~Leipzig\\
    Germany}
\email{fritzsche@math.uni-leipzig.de}

\author[Kirstein]{Bernd Kirstein}
\address{Universit\"at Leipzig\\
    Fakult\"at f\"ur Mathematik und Informatik\\
    PF~10~09~20\\
    D-04009~Leipzig\\
    Germany}
\email{kirstein@math.uni-leipzig.de}

\author[M\"adler]{Conrad M\"adler}
\address{Universit\"at Leipzig\\
    Fakult\"at f\"ur Mathematik und Informatik\\
    PF~10~09~20\\
    D-04009~Leipzig\\
    Germany}
\email{maedler@math.uni-leipzig.de}

\author[M\"uller]{Karsten M\"uller}
\address{Max Planck Institute for Human Cognitive and Brain Sciences\\
    Stephanstrasse~1A\\
    04103~Leipzig\\
    Germany}
\email{karstenm@cbs.mpg.de}



\date{January 1, 2004}
\dedicatory{Dedicated to M.~S.~Liv\v{s}ic and V.~P.~Potapov}

\begin{abstract}
In this paper we would like to show the interrelation between the different mathematical theories concerning the Schur interpolation problem, contractions in Hilbert spaces, pseudocontinuation and Darlington synthesis.
The main objects of this article are contractive functions holomorphic in the unit disc (Schur functions).
Here they are considered, on the one hand, as characteristic functions of contractions in Hilbert spaces and, on the other hand, as transfer functions of open systems.
\end{abstract}

\maketitle
\tableofcontents
\section{Introduction}

\subsection{Basic notations}
Let $\cF$ and $\cG$ be Hilbert spaces (all Hilbert spaces considered in this paper  are assumed to be complex and separable).
By $[\cF,\cG]$ we denote the Banach space of bounded linear operators defined on $\cF$ and taking values in $\cG$.
If $\cG =\cF$, we use the notation $[\cF] \defeq  [\cF, \cF]$.
If $\cH$ is a Hilbert space and $\cL$ is a subspace of $\cH$, then by $P_\cL$ we denote the orthogonal projection in $\cH$ onto $\cL$.

Let $\D  \defeq  \setaca{ \zeta\in \C }{\lvert\zeta\rvert <1}$ and  $\T\defeq\setaca{t\in \C}{\lvert t\rvert=1}$.
Denote by $\bv\lambda$ the Lebesgue--Borel measure on $\T$ normalized by $\bv\lambda\rk{\T}=2\pi$.

We recall the definitions of some spaces of vector and operator-valued functions that will be used in what follows.
A detailed treatment of this subject can be found, \teg{}, in \cite{RR,SN}.

By $\LinfT{\cF}{\cG}$ we denote the Banach space of measurable (not important in what sense, weak or strong, in view of the
separability of the spaces $\cF$ and $\cG$) $[\cF, \cG]$\nobreakdash-valued functions $\vartheta$ defined on $\T$ and such that
\[
\lVert\vartheta\rVert_{\LinfT{\cF}{\cG}}
\defeq\bv\lambda-\esssup_{t\in\T}\lVert\vartheta(t)\rVert_{[\cF, \cG]}
<\infty.
\]
Functions belonging to the closed unit ball
\[
CM(\T;[\cF,\cG])
\defeq\setacab{\vartheta\in\LinfT{\cF}{\cG}}{\lVert\vartheta\rVert_{\LinfT{\cF}{\cG}}\le 1}
\]
of the space $\LinfT{\cF}{\cG}$ are called \emph{$[\cF, \cG]$\nobreakdash-valued contractive measurable functions on $\T$}.

If $\HinfD{\cF}{\cG})$ is the Hardy space of $[\cF, \cG]$\nobreakdash-valued bounded holomorphic functions $\theta$ defined on $\D$ and such that
\[
\lVert\theta\rVert_{\HinfD{\cF}{\cG}}
\defeq\sup_{\zeta\in\D}\lVert\theta(\zeta)\rVert_{[\cF, \cG]}
<\infty,
\]
then by $\LinfpT{\cF}{\cG}$ we denote the subspace of $\LinfT{\cF}{\cG}$ consisting of $\bv\lambda$\nobreakdash-a.\,e.\ strong boundary value functions $\bv\theta$ for $\theta\in\HinfD{\cF}{\cG}$.
Moreover, the equality
\[
\lVert\bv\theta\rVert_{\LinfpT{\cF}{\cG}}
=\lVert\theta\rVert_{\HinfD{\cF}{\cG}}
\]
makes it possible to identify the spaces $\HinfD{\cF}{\cG}$ and $\LinfpT{\cF}{\cG}$  up to the obvious isomorphism.
Functions belonging to the closed unit ball
\[
\SFD{\cF}{\cG}
\defeq\setacab{\theta\in\HinfD{\cF}{\cG}}{\lVert\theta\rVert_{\HinfD{\cF}{\cG}}\le1}
\]
of the space $\HinfD{\cF}{\cG}$ are usually called \emph{$[\cF, \cG]$\nobreakdash-valued Schur functions on $\D$}.

Let $\cN $ be a Hilbert space.
By $\LzwoT{\cN}$ we will denote the Hilbert space of measurable (no matter in the weak or strong sense) $\cN$\nobreakdash-valued functions $h$ defined on $\T$ and such that
\[
\lVert h\rVert^2_{\LzwoT{\cN}}
\defeq\frac{1}{2\pi}\int_\T\lVert h(t)\Vert^2_\cN\bv\lambda\rk{\dif t}
<\infty,
\]
while the inner product in $\LzwoT{\cN}$ is defined by
\[
\langle h_1,h_2\rangle_{\LzwoT{\cN}}
\defeq\frac{1}{2\pi}\int_\T\langle h_1(t),h_2(t)\rangle_\cN\bv\lambda\rk{\dif t}.
\]
It is known that a function $h$ belongs to $\LzwoT{\cN}$ if and only if it admits the representation
\begin{align*}
    h(t)&=\sum_{k=-\infty}^\infty t^kh_k,&
    h_k&\in\cN\quad
    (k=0,\pm1,\pm2,\dotsc),\quad
    t\in\T,
\end{align*}
where the series convergence is understood as the convergence in the space $\LzwoT{\cN}$.
Moreover, the condition
\[
\sum_{k=-\infty}^\infty\lVert h_k \rVert^2_\cN
<\infty
\]
is satisfied.

By $\LzwopT{\cN}$ we denote the important subspace of $\LzwoT{\cN}$ that consists of functions $h\in\LzwoT{\cN}$ admitting the representation
\begin{align*}
    h(t)&=\sum_{k=0}^{\infty}t^kh_k,&
    h_k&\in\cN\quad(k=0,1,2,\dotsc),\quad
    t\in\T.
\end{align*}
Such functions can be considered as $\bv\lambda$\nobreakdash-a.\,e.\ strong boundary value functions of functions from the Hardy class $\HzwoD{\cN}$ on the unit disk
$\D$.
The space $\HzwoD{\cN}$ is formed by $\cN$\nobreakdash-valued holomorphic functions $h$ on $\D$ with Taylor series representation
\begin{align*}
    h(\zeta)&=\sum_{k=0}^\infty\zeta^kh_k,&
    h_k&\in\cN\quad
    (k=0,1,2,\dotsc),\quad
    \zeta\in\D,
\end{align*}
and such that
\[
\lVert h\rVert^2_{\HzwoD{\cN}}
\defeq\sup_{r<1}\left\{\frac{1}{2\pi}\int_\T\lVert h(rt)\rVert^2_\cN\bv\lambda\rk{\dif t}\right\}
<\infty.
\]
Furthermore, $\lVert h\rVert_{\HzwoD{\cN}}=\lVert\bv h\Vert_{\LzwopT{\cN}}$.

Any function $\vartheta$ from the space $\LinfT{\cF}{\cG}$ admits the following representation
\begin{align*}
    \vartheta\rk{t}&=\sum_{k=-\infty}^\infty t^k\vartheta_k,&
    \vartheta_k\in[\cF,\cG]\quad
    (k=0,\pm1,\pm2,\dotsc),\quad
    t\in\T,
\end{align*}
where for any $f\in\cF$ the series $\sum_{k=-\infty}^\infty t^k\vartheta_k f$ converges in the norm of the space $\LzwoT{\cG}$ and it is the Fourier expansion of the function $\vartheta f$.
Accordingly, any function $\bv\theta$ from the space $\LinfpT{\cF}{\cG}$ admits the representation
\begin{align*}
    \bv\theta\rk{t}&=\sum_{k=0}^\infty t^k\theta_k,\quad
    \theta_k\in[\cF,\cG]\quad
    (k=0,1,2,\dotsc),\quad
    t\in\T.
\end{align*}
Moreover, the corresponding function $\theta\in\HinfD{\cF}{\cG}$ admits the representation
\begin{align*}
    \theta(\zeta)&=\sum_{k=0}^\infty\zeta^k\theta_k,&
    \zeta&\in\D.
\end{align*}

In what follows (except for Subsection 4.8), we consider spaces of measurable vector- and operator-valued functions only on $\T$ and spaces of such holomorphic functions only on $\D$.
Throughout this paper (except for Subsection 4.8), for simplicity, we omit the symbols $\T$ and $\D$ in the notations of these spaces.
For example, further the spaces $L^\infty\rk{\T;[\cF,\cG]}$ and $S(\D;[\cF,\cG])$ will be denoted by $L^\infty[\cF,\cG]$ and $S[\cF,\cG]$, respectively.
In the case $\cF=\cG=\C$, we will use the denotations $\HinfDs\defeq H^\infty\rk{\D;\ek{\C}}$, $\SFDs\defeq S\rk{\D;\ek{\C}}$, $\LinfTs\defeq L^\infty\rk{\T;\ek{\C}}$, $\CMTs\defeq CM\rk{\T;\ek{\C}}$.


\subsection{Origins of problems and background}
When considering operator-valued Schur functions on $\D$, on the one hand, it is important that each such function is \emph{the characteristic function of some contraction} in a Hilbert space and, hence, it is \emph{the transfer function of some open scattering system} (\cite{SN,B,MR2013544,BD4}).
\rsec{sec2-0927} is devoted to the approach to considering Schur functions from this point of view.

On the other hand, the name of this class of functions comes from the well-known interpolation problem formulated and solved by I.~Schur in 1917 for the scalar case \cite{Schur} (see \rsubsec{subsec3.1-0927}).
An important role in solving the Schur problem is played by \emph{the step-by-step algorithm} introduced by I.~Schur, later named after him.

In 1974 V.~P.~Potapov proposed a new approach to solving interpolation problems of analysis by \emph{the method of $J$\nobreakdash-expanding matrix functions} \cite{MR409825}.
When stepwise solving interpolation problems like the Schur one, a sequence of nested operator balls arises.
V.~P.~Potapov proposed calling them \emph{Weyl balls} by analogy with disks in the well-known papers of H.~Weyl \cite{MR1511560,MR1503221} (the so-called cases of \emph{the limit disks} and \emph{the limit point}).
The method proposed by V.~P.~Potapov and I.~V.~Kovalishina makes it possible to control the asymptotic behavior of these balls and obtain the limit ball with predetermined characteristics \cite{MR647177}.

Consider the Schur problem generated by the first $n+1$ coefficients $\set{c_k}_{k=0}^n$ of the Taylor expansion
\[
\theta\rk{\zeta}=c_0+c_1\zeta+\dotsb+c_n\zeta^n+\dotsb,\qquad\zeta\in\D,
\]
for a function $\theta\in\SFFG$, $\dim\cF<\infty$, $\dim\cG<\infty$ (see \rsubsec{subsec3.1-0927}).
The set $S\ek{\cF,\cG;\set{c_k}_{k=0}^n}$ of solutions of this problem consists of all Schur functions whose first $n+1$ Taylor coefficients coincide with $\set{c_k}_{k=0}^n$.
At each point $\zeta_0\in\D$ the set
\[
K_n\rk{\zeta_0}\defeq\setacab{\omega\rk{\zeta_0}}{\omega\in S\ekb{\cF,\cG;\set{c_k}_{k=0}^n}}
\]
is an operator ball, called the Weyl ball, of the form
\beql{E1.3-0927}
\setacab{M_n\rk{\zeta_0}+\rho_{l,n}^{1/2}\rk{\zeta_0}u\rho_{r,n}^{1/2}\rk{\zeta_0}}{u\in\FG,u^\ad u\leq\IF}
\eeq
with \emph{the center} $M_n\rk{\zeta_0}$, \emph{the left semi-radius} $\rho_{l,n}\rk{\zeta_0}\in\ek{\cG}$, $\rho_{l,n}\rk{\zeta_0}\geq0$, and \emph{the right semi-radius} $\rho_{r,n}\rk{\zeta_0}\in\ek{\cF}$, $\rho_{r,n}\rk{\zeta_0}\geq0$ (see \rsubsec{subsec3.5-0927}; \rthm{T3.13}).

As noted above, the balls are nested, \tie,
\[
K_0\rk{\zeta_0}
\supset K_1\rk{\zeta_0}
\supset\dotsb
\supset K_n\rk{\zeta_0}
\supset K_{n+1}\rk{\zeta_0}
\supset\dotsb
\]
while
\[
\rho_{l,n}\rk{\zeta_0}
\geq\rho_{l,n+1}\rk{\zeta_0},
\quad
\rho_{r,n}\rk{\zeta_0}
\geq\rho_{r,n+1}\rk{\zeta_0}
\qquad(n=0,1,2,\dotsc).
\]
In the interpolation Schur problem, always
\beql{E1.4-0927}
\lim_{n\to\infty}\rho_{l,n}\rk{\zeta_0}=0
\eeq
holds true.
Letting $n\to\infty$, we obtain \emph{the limit Weyl operator ball}
\[
K_\infty\rk{\zeta_0}
\defeq\bigcap_{n\geq0}K_n\rk{\zeta_0},
\]
admitting the representation
\[
K_\infty\rk{\zeta_0}
=\setacab{M_\infty\rk{\zeta_0}+\rho_{l,\infty}^{1/2}\rk{\zeta_0}u\rho_{r,\infty}^{1/2}\rk{\zeta_0}}{u\in\FG,u^\ad u\leq\IF},
\]
where
\begin{align*}
    M_\infty\rk{\zeta_0}&\defeq\lim_{n\to\infty}M_n\rk{\zeta_0},\\
    \rho_{l,\infty}\rk{\zeta_0}&\defeq\lim_{n\to\infty}\rho_{l,n}\rk{\zeta_0},&
    \rho_{r,\infty}\rk{\zeta_0}&\defeq\lim_{n\to\infty}\rho_{r,n}\rk{\zeta_0}.
\end{align*}

The limit Weyl ball $K_\infty\rk{\zeta_0}$ corresponds to the ``infinite'' interpolation Schur problem generated by the sequence $\set{c_k}_{k=0}^\infty$ of all coefficients of the function $\theta$.
From \eqref{E1.4-0927} it follows that $\rho_{l,\infty}\rk{\zeta_0}=0$.
Thus, the sequence of balls $\set{K_n\rk{\zeta_0}}_{k=0}^\infty$ converges to the singleton $\set{M_\infty\rk{\zeta_0}}$, and, hence, the solution of this problem is unique.
Obviously, this unique solution is the function $\theta$, \tie,
\[
K_\infty\rk{\zeta_0}=\setb{M_\infty\rk{\zeta_0}}=\setb{\theta\rk{\zeta_0}},\qquad\zeta_0\in\D.
\]

Note that it follows from the fundamental Orlov theorem \cite{MR0425671} that $\rank\rho_{r,\infty}\rk{\zeta}$ does not depend on $\zeta\in\D$.

Let $T\in\ek{\cH}$ be a contraction for which the function $\theta\in\SFFG$ is characteristic.
In 1978 V.~P.~Potapov posed the question:
\begin{quote}
    ``What properties of the contraction $T$ are characterized by the constant $\rank\rho_{r,\infty}\rk{\zeta}$, $\zeta\in\D$?''
\end{quote}
The series of papers \cite{D82} is devoted to the answer to this question.
In this survey, this range of issues is expounded in \rsec{sec3-0927}, while solving the Schur problem by the method of $J$\nobreakdash-expanding matrix functions is presented in \rsubsecsts{subsec3.1-0927}{subsec3.5-0927}.

To answer V.~P.~Potapov's question, consider the \emph{normalized} left semi-radius (\rdefn{D3.14})
\[
\tilde{\rho}_{l,n}\rk{\zeta}\defeq\frac{1}{\abs{\zeta}^{2n+2}}\rho_{l,n}\rk{\zeta},\qquad\zeta\in\D\setminus\set{0}.
\]
Since the considered Schur problem is generated by the sequence $\set{c_k}_{k=0}^\infty$ of the Taylor coefficients of the function $\theta\in\SFFG$, we will add the symbol $\theta$ to the notations of the semi-radii of the Weyl balls.
The introduction of the normalized left semi-radius is motivated by the following equality (\rthm{T3.15})
\beql{E1.7-0927}
\tilde{\rho}_{l,n}\rk{\zeta,\theta}=\rho_{r,n}\rk{\zeta,\theta^\sch},\qquad\zeta\in\D\setminus\set{0},
\eeq
where $\theta^\sch\rk{\zeta}\defeq\te^\ad\rk{\ko\ze}$, $\ze\in\D$.
It makes possible to define $\tilde{\rho}_{l,n}\rk{\zeta,\theta}$ also at the point $\ze=0$ by setting $\tilde{\rho}_{l,n}\rk{0,\theta}\defeq\rho_{r,n}\rk{0,\theta^\sch}$.
From \eqref{E1.7-0927} the existence of the limit
\[
    \tilde{\rho}_{l,\infty}\rk{\zeta,\theta}
    \defeq\lim_{n\to\infty}\tilde{\rho}_{l,n}\rk{\zeta,\theta}
    =\rho_{r,\infty}\rk{\ze,\theta^\sch},\qquad\ze\in\D,
\]
follows.
From here and the Orlov theorem we obtain that $\rank\tilde{\rho}_{l,\infty}\rk{\zeta,\theta}$ does not depend on $\ze\in\D$.
Now, in these refined notations, the Weyl ball \eqref{E1.3-0927} can be written in the form
\[
K_n\rk{\ze_0}
=\setacab{M_n\rk{\zeta_0}+\abs{\ze_0}^{2n+2}\tilde{\rho}_{l,n}^{1/2}\rk{\zeta_0,\te}u\rho_{r,n}^{1/2}\rk{\zeta_0,\te}}{u\in\FG,u^\ad u\leq\IF}.
\]
Here it is seen that the sequence of Weyl balls shrinks to a singleton as $n\to\infty$ on account of the factor $\abs{\ze_0}^{2n+2}$.
Along with this the normalized left semi-radius and the right one tend to $\tilde{\rho}_{l,\infty}\rk{\zeta_0,\theta}$ and $\rho_{r,\infty}\rk{\zeta_0,\te}$, respectively.

The answer to V.~P.~Potapov's question is given by \rthm{T3.24}:

\textit{The values of $\rank\tilde{\rho}_{l,\infty}\rk{\ze,\theta}$ and $\rank\rho_{r,\infty}\rk{\ze,\te}$ coincide respectively with the multiplicities of the largest unilateral shift and coshift contained in the contraction $T$.}

More profound is the assertion about the factorizations of the limit semi-radii (\rthm{T3.22}):

\textit{The limit semi-radii admit the factorizations:
\begin{align*}
    \rho_{r,\infty}\rk{\ze,\te}&=\te_r^\ad\rk{\ze}\te_r\rk{\ze},\qquad\ze\in\D;\\
    \tilde{\rho}_{l,\infty}\rk{\ze,\te}&=\te_l\rk{\ze}\te_l^\ad\rk{\ze},\qquad\ze\in\D,
\end{align*}
where the Schur functions $\te_r\in\SF{\cF}{\cL_0}$ and $\te_l\in\SF{\tcL_0}{\cG}$ have forms \eqref{E3.76} and \eqref{E3.77}, respectively.}

The functions $\te_r$ and $\te_l$ (\rdefn{D3.23}) are called \emph{the right and left defect functions} of the Schur function $\te$, respectively.

Note that representations \eqref{E3.76} and \eqref{E3.77} allow us to extend the concept of the defect functions of a Schur function $\te\in\SFFG$ to the case of infinite-dimensional spaces $\cF$ and $\cG$.

\rsec{sec4-1003} outlines the role of the defect functions $\ter$ and $\tel$ of a function $\te\in\SFFG$ in solving the problem of loss decrease (of energy) in an open scattering system for which $\te$ is the transfer function. 
Following to the ideas of D.~Z.~Arov (\cite{A-1,MR2013544}), we proceed from the fact that this physical problem is directly related to the study of \emph{up-leftward Schur extensions} of the form
\beql{E1.8-1003}
\Xi\defeq\Mat{\te_{11}&\te_{12}\\\te_{21}&\te_{22}}\in\SF{\cF'\oplus\cF}{\cG'\oplus\cG},\qquad\te_{22}\defeq\te,
\eeq
for the function $\te$.
Their special cases
\beql{E1.9-1003}
\Omega\defeq\Mat{\te_{12}\\\te}\in\SF{\cF}{\cG'\oplus\cG},\qquad\Upsilon\defeq\mat{\te_{21},\te}\in\SF{\cF'\oplus\cF}{\cG},
\eeq
where respectively $\cF'=\set{0}$ and $\cG'=\set{0}$, are called \emph{upward} and \emph{leftward Schur extensions of $\te$}, respectively.

The main motivations are the following definitions being refinements of those in \cite{MR2013544}.
An open scattering system with the transfer function $\te\in\SFFG$ is called \emph{lossless at the input} (resp.\ \emph{at the output}) if $\te$ is a $*$\nobreakdash-inner (resp.\ inner) function.
It is termed \emph{lossless} if $\te$ is a two-sided inner function (see \rsubsec{subsec4.7-1003} for details).

The problem of describing unidirectional Schur extensions of form \eqref{E1.9-1003} is actually solved
in \rsubsec{subsec4.1} (see \rthm{T4.6}, \rdefn{D1454}, \rthm{T4.8}, and \cite[Notes to \csec{V.4}]{SN}).
\rthm{T4.8} characterizes the function $\ter$ (\tresp\ $\tel$) as the largest minorant (\tresp\ $*$-minorant) for the function
\begin{align*}
    \Pi\rk{t}&\defeq\rk{\IF-\bv\te^\ad\rk{t}\bv\te\rk{t}}^{1/2}\in CM\ek{\cF}\\
    \text{(\tresp\ }\Sigma\rk{t}&\defeq\rk{\IG-\bv\te\rk{t}\bv\te^\ad\rk{t}}^{1/2}\in CM\ek{\cG}\text{)},&t&\in\T.
\end{align*}
Consequently (\rcor{C4.9}), the right (\tresp\ left) defect function $\ter$ (\tresp\ $\tel$) is outer (\tresp\ $*$-outer).
In terms of up-leftward Schur extensions, \rthm{T4.8} can be reformulated in the following way:

\textit{``A function $\Omega\defeq\col\rk{\te_{12},\te}\in H^\infty\ek{\cF,\cG'\oplus\cG}$ 
    \textnormal{(}resp.\ $\Upsilon\defeq\mat{\te_{21},\te}\in H^\infty\ek{\cF'\oplus\cF,\cG}$\textnormal{)} 
is an upward \textnormal{(}resp.\ leftward\textnormal{)} Schur extension of a function $\te\in\SFFG$ if and only if the function $\te_{12}$ \textnormal{(}resp.\ $\te_{21}$\textnormal{)} admits the representation
\[
\te_{12}=\mu\ter\qquad\rk{\text{resp.\ }\te_{21}=\tel\nu},
\]
where $\mu\in\SF{\cL_0}{\cG'}$ \textnormal{(}resp.\ $\nu\in\SF{\cF'}{\tcL_0}$\textnormal{)}.''}

Let $\te\in\SFFG$ and $\vphi\defeq\ter\in\SF{\cF}{\cG_0}$, $\cG_0\defeq\cL_0$ (resp.\ $\psi\defeq\tel\in\SF{\cF_0}{\cG}$, $\cF_0\defeq\tcL_0$).
Note that, as is easy to see, on the Schur extension
\beql{E1.10-1003}
\Omega_0\defeq\Mat{\vphi\\\te}\in\SF{\cF}{\cG_0\oplus\cG}
\qquad\rk{\text{resp.\ }\Upsilon_0\defeq\mat{\psi,\te}\in\SF{\cF_0\oplus\cF}{\cG}}
\eeq
the least deviation from the inner (resp.\  $*$\nobreakdash-inner) case among all upward (resp.\ leftward) Schur extensions of form \eqref{E1.9-1003} is attained (as well as in the case when  $\te_{12}=\mu\vphi$ (resp.\ $\te_{21}=\psi\nu$), where $\mu$ (resp.\ $\nu$) is an inner (resp.\ $*$\nobreakdash-inner) function).
Thus, we can say that in the open system with the transfer function $\Omega_0$ (resp.\ $\Upsilon_0$) the losses at the output (resp.\ at the input) are the smallest among all such systems with the transfer functions $\Omega$ (resp.\ $\Upsilon$) of form \eqref{E1.9-1003} (including the original system with the transfer function $\te$).
A detailed scheme of restructuring from the original open system to a system with the transfer function $\Omega$ (resp.\ $\Upsilon$) of form \eqref{E1.9-1003} is discussed in \rsubsec{subsec4.7-1003}.

With bilateral Schur extensions of form \eqref{E1.8-1003}, the problem of their description is more complicated because for a given pair of upward and leftward Schur extensions $\Omega$ and $\Upsilon$ of form \eqref{E1.9-1003} may not exist a Schur function $\te_{11}\in\SF{\cF'}{\cG'}$ which completes an up-leftward Schur extension $\Xi$ of form \eqref{E1.8-1003} (\cite{MR2013544,BD20}).
At the same time, passing to the boundary value functions $\bv{\te_{22}}$ ($=\bv\te$), $\bv{\te_{12}}$, $\bv{\te_{21}}$ of the functions $\te_{22}$, $\te_{12}$, $\te_{21}$ we can now consider the problem of completion to a block matrix already in the class of contractive measurable operator-valued functions on $\T$.

By analogy with \eqref{E1.8-1003}--\eqref{E1.9-1003} we consider up-leftward extensions of the form
\beql{E1.11-1003}
\Xi\defeq\Mat{\vte_{11}&\vte_{12}\\\vte_{21}&\vte_{22}}\in  CM\ek{\cF'\oplus\cF,\cG'\oplus\cG}, \qquad \vte_{22}\defeq\vte,
\eeq
for a function $\vte\in  CM\ek{\cF,\cG}$, as well as their special cases, upward and leftward ones of the form
\beql{E1.12-1003}
\Omega\defeq\Mat{\vte_{12}\\\vte}\in  CM\ek{\cF,\cG'\oplus\cG}, \qquad \Upsilon\defeq\mat{\vte_{21},\vte}\in  CM\ek{\cF'\oplus\cF,\cG},
\eeq
respectively.
The need to study them leads us to use a more general \emph{theory of unitary couplings} (see \rsubsec{subsec4.2-1003}).
It originates from the paper \cite{A-A} of V.~M~Adamyan ad D.~Z.~Arov and is a generalization of the theory of unitary colligation (see \rsubsecss{Sec1.1}{subsec4.5-1003}) which was effectively used in the study of operator-valued Schur functions.

A unitary coupling (or simply coupling) $\sg$ (\rdefn{D4.9}) can be viewed as a more general mathematical model of an open discrete time scattering system with an input space $\cG$, an output space $\cF$, and a connecting space $\cH$ on which the connecting unitary operator $U\in\ek{\cH}$ acts.
An exhaustive characteristic of a coupling $\sg$ is a contractive measurable operator-valued function $\vte_\sg\in  CM\ek{\cF,\cG}$ called \emph{the scattering suboperator of the coupling} $\sg$ (\rdefn{D4.16}).
Note that every $\vte\in  CM\ek{\cF,\cG}$ is the scattering suboperator of some coupling 
$\sg$, i.\,e., $\vte_\sg=\vte$.
Moreover, under the natural condition of minimality (\rdefn{D4.10}), the coupling 
$\sg$ is determined up to unitary equivalence (\rthm{T4.18}).
The definition of unitary colligations $\Dl$ (\rdefn{de1.1}) has direct connections with the special subclass of 
\emph{orthogonal} couplings $\sg$ (\rdefn{D4.24}) whose suboperators $\vte_\sg$ belong to $L_+^\infty\ek{\cF,\cG}$.
Moreover, $\vte_\sg=\bv{\te_\Dl^\sch}$, where $\te_\Dl\in\SFGF$ is the characteristic function of corresponding colligation $\Dl$ (\rthm{T4.26}).

The key tool for describing extensions of form \eqref{E1.11-1003} in terms of couplings is the notion of \emph{unilateral input} and \emph{output scattering channels for a coupling} $\sg$ \zitaa{BD-1}{\cpart{II}} which are directly related to unilateral coshifts and shifts contained in $U$ (\rdefn{D4.19}).
Herewith the unilateral channels generated by $\cG$ and $\cF$ are called \emph{the principal input} and \emph{output channels} of $\sg$, respectively.
This enables each up-leftward extension $\Xi$ of form \eqref{E1.11-1003} to be associated with some rearrangement of a coupling $\sg$ ($\vte_\sg=\vte$) by extending its principal unilateral channels in such a way that the suboperator of the obtained coupling is $\Xi$ (see details in \rsubsec{subsec4.4-1003}).

Distinguishing among all extensions of form \eqref{E1.11-1003} the so-called \emph{regular} ones, when the extension of the principal channels mentioned above is possible without extending the connecting space $\cH$, we come to an important property of such $\Xi$.
In this case, for any pair of regular extensions $\Omega$ and $\Upsilon$ of form \eqref{E1.12-1003} there exists a unique regular extension $\Xi$ of form \eqref{E1.11-1003}.
In an obvious way, the concept of regularity can be spread to Schur extensions of form \eqref{E1.8-1003}.

For an orthogonal minimal coupling $\sg$, one can naturally define the orthogonal decomposition $\cH=\cH_{\mathrm{int}}\oplus\cH_{\mathrm{ext}}$ of the space $\cH$ into \emph{the internal} and \emph{external subspaces} (see \rsubsec{subsec4.5-1003}).
In this regard, those unilateral channels of $\sg$ which are realized on subspaces of 
$\cH_{\mathrm{int}}$ or $\cH_{\mathrm{ext}}$ are termed \emph{internal} or \emph{external}, respectively.
Note that the principal unilateral channels of $\sg$ are external.
Exactly the use of internal unilateral channels of $\sg$ for extending its principal (external) unilateral ones allows to describe all regular Schur extensions of $\te\in\SFFG$ ($\vte_\sg=\bv{\te}^\sch$).

Proceeding from the regularity of the Schur extensions $\Omega_0$ and $\Upsilon_0$ of form \eqref{E1.10-1003} (see \rthm{T4.31}) and the uniqueness of the completion to a regular up-leftward extension of form \eqref{E1.11-1003} when $\vte_{22}\defeq\bv{\te}$, $\vte_{12}\defeq\bv{\vphi}$, $\vte_{21}\defeq\bv{\psi}$, we can introduce the important function \cite{MR1491502}:

\textit{``The unique function $\chi\in  CM\ek{\cF_0,\cG_0}$ such that the function
\[
\Xi_0
\defeq\Mat{\chi&\bv{\vphi}\\\bv{\psi}&\bv{\te}}
\in CM\ek{\cF_0\oplus\cF,\cG_0\oplus\cG}
\]
is a regular up-leftward extension of $\bv{\te}$, where $\te\in\SFFG$, is called the suboperator of internal scattering of an orthogonal minimal coupling $\sg$ such that $\vte_\sg=\bv{\te}^\sch$.''}

Note that earlier the function $\chi$ appeared for a particular case in \cite{MR2013544}.
Now a description of all regular Schur extensions of $\te$ in terms $\vphi$, $\psi$,
$\chi$ can be given in the following way (\rthmp{th4.36}{T4.36.a}):

\textit{``A function
\[
\Xi
\defeq\Mat{\te_{11}&\te_{12}\\\te_{21}&\te_{22}}
\in  H^\infty\ek{\cF'\oplus\cF,\cG'\oplus\cG},
\qquad\te_{22}\defeq\te,
\]
is a regular up-leftward Schur extension of a function $\te\in\SFFG$ if and only if the functions $\te_{12}$, $\te_{21}$, $\bv{\te_{11}}$ admit the representations
\[
\te_{12}=\omega\vphi,
\qquad\te_{21}=\psi\upsilon,
\qquad\bv{\te_{11}}=\bv{\om}\chi\bv{\upsilon},
\]
where $\om\in\SF{\cG_0}{\cG'}$ and $\upsilon\in\SF{\cF'}{\cF_0}$ are $*$\nobreakdash-inner and inner functions, respectively.''}

A special case of the problem of describing Schur extensions of form \eqref{E1.8-1003} is the problem of \emph{the Darlington synthesis} arising in the electrical network theory \cite{MR0001658}.
It needs to find all two-sided inner extensions of form \eqref{E1.8-1003} for a given Schur function $\te\in\SFFG$ if they exist.
In terms of this article, the answer for the regular case is given in \rthmp{T4.30}{T4.30.b} and \rthmp{th4.36}{T4.36.b}.

Another approach to solving the Darlington problem which is due to D.~Z.~Arov \cite{MR0428098} and, independently, due to R.~G.~Douglas and J.~W.~Helton \cite{MR0322538} is presented in \rsubsec{subsec4.8-1003}.
Let $\De\defeq\setaca{\ze\in\D}{\abs{\ze}>1}\cup\set{\infty}$.
It turned out that the property of bounded type pseudocontinuability across the unit circle $\T$ to $\De$ of a Schur function $\te\in\SFFG$ (\rdefn{D4.42}) is a sufficient condition for the existence of a two-sided inner extension of $\te$ (\rthm{T4.46}).
However,  in the general case this property is not necessary, being such if $\dim\cF<\infty$, $\dim\cG<\infty$ (\rthm{T4.48}).
D.~Z.~Arov proposed in \cite{MR2013544} a more general sufficient condition (\rthm{T4.50}), about which it seems to be unknown whether it is necessary.

\section{Unitary colligations, open systems and characteristic operator functions}\label{sec2-0927}

\subsection{Unitary colligations, shifts and coshifts contained in contraction}\label{Sec1.1}
Let $T$ be a contraction acting in some Hilbert space $\cH$, \tie{},\  $\norm{T}\le 1$.
The operators 
\begin{align*}
    D_T&\defeq \sqrt{I_\cH-T^*T}&
    &\text{and}&
    D_{T^*}&\defeq \sqrt{I_\cH-TT^*}
\end{align*}
are called \emph{the defect operators} of $T$.
The closures of their ranges
\begin{align*}
    \cD_T&\defeq \ol{D_T(\cH)}&
    &\text{and}&
    \cD_{T^*}&\defeq \ol{D_{T^*}(\cH)}
\end{align*}
are said to be \emph{the defect spaces} of $T$.
The dimensions of these spaces
\begin{align*}
    \dl_T&\defeq \dim{\cD_T}&
    &\text{and}&
    \dl_{T^*}&\defeq \dim{\cD_{T^*}}
\end{align*}
are called \emph{the defect numbers} of the contraction $T$.
In this way, the condition $\dl_T=0$ (\tresp{}\ $\dl_{T^*}=0$) characterizes isometric (\tresp{}\ coisometric) operators, whereas the conditions $\dl_T=\dl_{T^*}=0$ characterize unitary operators.
Recall that an operator is called coisometric if its adjoint is isometric.

Clearly, $TD_T^2=D_{T^*}^2T $.
From here (see \cite[Chapter~I]{SN}) it follows that $TD_T=D_{T^*}T$.
Forming the adjoint operators, we obtain
\beql{1.1}
T^*D_{T^*}
=D_TT^*.
\eeq
Starting from the contraction $T$, we can always find Hilbert spaces $\cF$ and $\cG$ and operators $F\colon\cF\to\cH$, $G\colon\cH\to\cG$, and $S\colon\cF\to\cG$ such that the operator matrix
\beql{1.2}
Y
=\Mat{T&F\\ G&S} \colon\cH\oplus\cF\to\cH\oplus\cG
\eeq
is unitary, \tie{},\ the conditions $Y^*Y=I_{\cH\oplus\cF}$, $YY^*=I_{\cH\oplus\cG}$ are satisfied.
Obviously, these equalities can be rewritten in the form
\beql{1.3}
\begin{aligned}
    G^*G+T^*T&=I_{\cH},&S^*S+F^*F&=I_{\cF},&G^*S+T^*F&=0\\
    TT^*+FF^*&=I_{\cH},&GG^*+SS^*&=I_{\cG},&TG^*+FS^*&=0.
\end{aligned}
\eeq
As an example of such a construction, one can consider the spaces

\beql{1.100}
\cF\defeq \cD_{T^*},\ \cG\defeq \cD_T
\eeq
and the operators
\begin{align}\label{1.101}
    F&\defeq \operatorname{Rstr}_{\cF}D_{T^*} \in [\cF,\cH]&
    G&\defeq D_T \in [\cH,\cG],&
    S&\defeq \operatorname{Rstr}_{\cF}(-T^*) \in [\cF,\cG].
\end{align}
Using \eqref{1.1}, it is easily checked that conditions \eqref{1.3} are fulfilled in this case.
Note that in the general situation from \eqref{1.3} it follows $G^*G=D_T^2$, $FF^*=D_{T^*}^2$.
Hence,
\begin{align}\label{1.4}
    \ol{F(\cF)}&=\cD_{T^*},&
    \ol{G^*(\cG)}&=\cD_T.
\end{align}

\begin{defn}\label{de1.1}
    The 7-tuple
    \beql{E2.5}
    \Dl=(\cH, \cF, \cG; T, F, G, S)
    \eeq
    consisting of three Hilbert spaces $\cH,\cF,\cG$ and four operators $T,F,G,S$ where
    \begin{align*}
        T&\colon\cH\to\cH,&
        F&\colon\cF\to\cH,&
        G&\colon\cH\to\cG,&
        S&\colon\cF\to\cG
    \end{align*}
    is called a \emph{unitary colligation} (or more short \emph{colligation}) if the operator matrix $Y$ given via \eqref{1.2} is unitary. 
\end{defn}

The operator $T$ is called \emph{the fundamental operator} of the colligation $\Dl$.
Clearly, the fundamental operator of a colligation is a contraction.
The operation of representing a contraction $T$ as fundamental operator of a unitary colligation is called \emph{embedding} $T$ in a colligation.
The space $\cH$ of the colligation $\Dl$ is said to be \emph{internal} and the spaces $\cF$ and $\cG$ are called \emph{external}, namely, \emph{input external} space and 
\emph{output external} space, respectively.
This embedding permits to use the spectral theory of unitary operators for the study of contractions (see \cite{SN}).

We will use the notation $\bigvee_{\alpha\in A}\cL_\alpha$, where $\cL_\alpha\subset\cH$, $\alpha\in A$, for denoting the smallest subspace of the space $\cH$ which contains all $\cL_\alpha$, $\alpha\in A$.
The spaces
\begin{align*}
    \cH_{\cF}&\defeq \bigvee_{n=0}^{\infty}T^nF(\cF),&
    \cH_{\cG}&\defeq \bigvee_{n=0}^{\infty}T^{*n}G^*(\cG)
\end{align*}
and their orthogonal complements $\cH_{\cF}^\bot\defeq \cH\ominus\cH_{\cF}$, $\cH_{\cG}^\bot\defeq \cH\ominus\cH_{\cG}$ play an important role in the theory of colligations. 
Therefore, we will use the decompositions
\begin{align}\label{1.7}
    \cH&=\cH_{\cF}^\bot\oplus\cH_{\cF},&
    \cH&=\cH_{\cG}\oplus\cH_{\cG}^\bot.
\end{align}
in what follows.
From \eqref{1.4} it follows that the spaces $\cH_\cF$ and $\cH_\cG$ can also be defined on an alternate way, namely
\begin{align}\label{1.8}
    \cH_{\cF}&\defeq \bigvee_{n=0}^{\infty}T^n\cD_{T^*},&
    \cH_{\cG}&\defeq \bigvee_{n=0}^{\infty}T^{*n}\cD_T.
\end{align}
Consequently, the spaces $\cH_\cF$ and $\cH_\cG$ do not depend on the concrete way of embedding $T$ in a colligation.
Note that $\cH_\cF$ is invariant with respect to $T$ whereas $\cH_\cG$ is invariant with respect to $T^*$.
This means that $\cH_\cF^\bot$ and $\cH_\cG^\bot$ are invariant with respect to $T^*$ and $T$, respectively.
Passing on to the kernels of the adjoint operators in equalities \eqref{1.8}, we obtain
\begin{align}
    \cH_\cF^\bot&=\bigcap_{n=0}^\infty\ker(D_{T^*}T^{*n}),\label{E2.10-20230623}\\
    \cH_\cG^\bot&=\bigcap_{n=0}^\infty\ker(D_TT^n).\label{E2.11-20230623}
\end{align}

\begin{thm}[{\cite[Part~VI]{D82}}]
    The equalities
    \[
    \cH_\cF^\bot
    =\{h\in\cH : \norm{T^{*n}h}=\norm{h}, \ n=1,2,3,\dotsc \}
    \]
    and
    \[
    \cH_\cG^\bot
    =\{h\in\cH : \norm{T^nh}=\norm{h}, \ n=1,2,3,\dotsc \}
    \]
    hold true.
\end{thm}
\begin{proof}
    For $n=1,2,3,\dotsc$, clearly
    \[
    \norm{T^nh}^2
    =\ip{T^{*n}T^nh}{h}
    =\ip{T^{*n-1}T^{n-1}h}{h}-\ip{T^{*n-1}D_T^2T^{n-1}h}{h}.
    \]
    Now the second assertion follows from \eqref{E2.11-20230623} and the equality
    \begin{align*}
        \norm{T^{n-1}h}^2-\norm{T^nh}^2
        &=\norm{D_TT^{n-1}h}^2,&n&=1,2,3,\dotsc.
    \end{align*}
    Analogously, the first assertion can be proved, using \eqref{E2.10-20230623}.
\end{proof}

\begin{cor}\label{cor1.1} 
    The space $\cH_\cF^\bot\;(\text{resp.\ }\cH_\cG^\bot)$ is characterized by the following properties:
    \benui
    \il{cor1.1.a} $\cH_\cF^\bot\;(\text{resp.\ }\cH_\cG^\bot)$ is invariant with respect to $T^*$ $(\text{resp.\ }T)$.
    \il{cor1.1.b} $ \rstr{T^*}{\cH_\cF^\bot} \ (\tresp{}\ \rstr{T}{\cH_\cG^\bot})$
    is an isometric operator.
    \il{cor1.1.c} $\cH_\cF^\bot\;(\text{resp.\ }\cH_\cG^\bot)$ is the largest subspace of $\cH$ having the properties \eqref{cor1.1.a}, \eqref{cor1.1.b}.
    \eenui
\end{cor}

From the foregoing consideration we immediately obtain the following result.

\begin{thm}[{\cite[\cch{I}, \cthm{3.2}]{SN}}]\label{th1.2}
    The equality
    \[
    \cH_\cF^\bot\cap\cH_\cG^\bot
    =\{h\in\cH : \norm{T^{*n}h}=\norm{h}=\norm{T^nh}, \ n=1,2,3,\dotsc \}
    \]
    holds true.
\end{thm}

\begin{cor}\label{cor1.2}
    The subspace $\cH_\cF^\bot\cap\cH_\cG^\bot$ is the largest among all subspaces $\cH'$ of $\cH$ having the following properties:
    \begin{enuit}
        \il{cor1.2.a} $\cH'$ reduces $T$.
        \il{cor1.2.b} $ \rstr{T}{\cH'} $ is a unitary operator.
    \end{enuit}
\end{cor}

A contraction $T$ on $\cH$ is called \emph{completely nonunitary} if there is no nontrivial reducing subspace $\cL$ of $\cH$ for which the operator $ \rstr{T}{\cL}$ is unitary.
Consequently, a contraction is completely nonunitary if and only if $\cH_\cF^\bot\cap\cH_\cG^\bot=\set{0}$.
The colligation $\Dl$ given in \eqref{E2.5} is called \emph{simple} if $\cH=\cH_\cF\vee\cH_\cG$ 
and \emph{abundant} otherwise.
Hence, the colligation $\Dl$ is simple if and only if its fundamental operator $T$ is a completely nonunitary contraction.

Taking into account the Wold decomposition for isometric operators  (see, e.g., \cite[Ch.~I]{SN}), from \rcor{cor1.1} we infer the following result:

\begin{thm}[{\cite[Part~VI]{D82}}]\label{th1.3}
    Let $T \in [\cH]$ be a completely nonunitary contraction. 
    Then the subspace $\cH_\cF^\bot\;(\text{resp.\ }\cH_\cG^\bot)$ is characterized by the following properties:
    \benui
    \il{th1.3.a} The subspace $\cH_\cF^\bot\;(\text{resp.\ }\cH_\cG^\bot)$ is invariant with respect to $T^*$ $(\text{resp.\ }T)$.
    \il{th1.3.b} The operator $ \rstr{T^*}{\cH_\cF^\bot}\;(\text{resp.\ }\rstr{T}{\cH_\cG^\bot})$   is a unilateral shift.
    \il{th1.3.c} $\cH_\cF^\bot\;(\text{resp.\ }\cH_\cG^\bot)$ is the largest subspace of $\cH$ having the properties \eqref{th1.3.a}, \eqref{th1.3.b}.
    \eenui
\end{thm}

We say that a unilateral shift $V\colon\cL\to\cL$ \emph{is contained} in the contraction $T$ if $\cL$ is a subspace of $\cH$ which is invariant with respect to $T$ and $\rstr{T}{\cL}=V$ is satisfied.

\begin{defn}
    Let $T \in [\cH]$ be a completely nonunitary contraction. 
    Then the shift $V_T\defeq\ \rstr{T}{\cH_\cG^\bot}$   is called \emph{the largest shift} contained in $T$.
\end{defn}

By a \emph{coshift} we mean an operator the adjoint of which is a unilateral shift.
We say that a coshift $\wt V\colon\tcL\to\wt \cL$ \emph{is contained} in $T$ if the unilateral shift $\wt V^*$ is contained in $T^*$.
Then from \rthm{th1.3} it follows that the operator $ V_{T^*}= \rstr{T^*}{\cH_\cF^\bot}  $ is the largest shift contained in $T^*$.
If $\cH_\cG^\bot=\set{0}$ (\tresp{}\ $\cH_\cF^\bot=\set{0}$) we will say that the shift $V_T$ (\tresp{}\ $V_{T^*}$) has multiplicity zero.

\begin{defn}
    Let $T \in [\cH]$ be a completely nonunitary contraction. 
    Then the coshift $\wt V_T\defeq (V_{T^*})^*$ is called \emph{the largest coshift} contained in $T$.
\end{defn}

\begin{thm}[{\cite[Theorem~1.9]{D06}}]\label{th1.4}
    Let $T \in [\cH]$ be a completely nonunitary contraction. 
    Then the multiplicities of the largest shifts $V_T$ and $V_{T^*}$ are not greater than $\dl_{T^*}$ and $\dl_T$,  respectively.
\end{thm}

\begin{cor}[{\cite[Corollary~1.10]{D06}}]\label{cor1.9}
    Let $\Dl$ be a simple unitary colligation of type \eqref{E2.5}.
    Denote $\cL_0$ and $\tcL_0$ the generating wandering subspaces for the largest shifts $V_T$ and $V_{T^*}$,  respectively.
    Then $\ol{P_{\cL_0}F(\cF)}=\cL_0$, $\ol{P_{\tcL_0}G^*(\cG)}=\tcL_0$, where $P_{\cL_0}$ and $P_{\tcL_0}$ are the orthogonal projections from $\cH$ onto $\cL_0$ and $\tcL_0$,  respectively.
\end{cor}


Let $T$ be a completely nonunitary contraction in $\cH$.
Then
\begin{align}
    \cH_\cF^\bot&=\bigoplus_{n=0}^\infty\wt{V}_T^{*n}\wt{\cL}_0=\bigoplus_{n=0}^\infty T^{*n}\wt{\cL}_0=
    \dotsb\oplus T^{*n}\wt{\cL}_0\oplus\dotsb\oplus T^*\wt{\cL}_0\oplus\wt{\cL}_0,\notag\\
    \cH_\cG^\bot&=\bigoplus_{n=0}^\infty V_T^n\cL_0=\bigoplus_{n=0}^\infty T^n\cL_0=\cL_0\oplus T\cL_0\oplus 
    \dotsb\oplus T^n\cL_0\oplus \dotsb\label{E2.1002}.
\end{align}
Therefore, equalities \eqref{1.7} can be rewritten in the form
\begin{align*} 
    \cH&=\dotsb\oplus T^{*n}\wt{\cL}_0\oplus\dotsb\oplus T^*\wt{\cL}_0\oplus\wt{\cL}_0\oplus\cH_{\cF},\\
    \cH&=\cH_{\cG}\oplus\cL_0\oplus T\cL_0\oplus \dotsb\oplus T^n\cL_0\oplus \dotsb.
\end{align*}

\subsection{Characteristic function as the transfer function of the open system}

Let $\Delta$ be a unitary colligation of type \eqref{E2.5}. To the colligation $\Delta$ one can assign the open system $\cO_\Delta$ with discrete time $n = 0,1,2,3,\dotsc$ which is
defined by the relations
\begin{equation}\label{Nr.2.12}
    \left\{\begin{array}{cl}
        h(n+1) =  \ T h(n) + F f(n), \\
        \ \ g(n) \ \ \ =  \ G h(n) + S f(n).
    \end{array}\right.
\end{equation}
Here $\{f(n)\}_{n=0}^{\infty}, \ \{g(n)\}_{n=0}^{\infty}$ \ and $\{h(n)\}_{n=0}^{\infty}$ \ are sequences in $\cF, \ \cG, \ \cH$, respectively, which are interpreted as input data, output data and internal states of the system, respectively. In relations (\ref{Nr.2.12}) we assume that the input data $\{f(n)\}_{n=0}^{\infty}$ and the initial internal state $h(0)$ are known. It is required to determine the output data  $\{g(n)\}_{n=0}^{\infty}$ and the internal states $h(n)$ for all subsequent instants $n = 1, 2, 3, ... \ $. This problem can be solved in the following way. Knowing $f(0)$ and $h(0)$, we determine $h(1)$ and $g(0)$ via (\ref{Nr.2.12}). After that we determine $h(2)$ and $g(1)$ by $f(1)$ and $h(1)$ and so on. If $f(n) = 0, \ n = 0, 1, 2, ... \ $, then the contraction $T$ describes the evolution of the internal state
$h(n) = T^n h(0), \ n = 0, 1, 2, ... \ $.

Let us feed to the input of the open systems a sequence of the form $f(n) = \zeta^{-n}f,\ f \in \cF, \ n= 0, 1, 2, ... \ $. We shall seek the internal states of the same form, namely, $h(n) = \zeta^{-n}h, \ h \in \cH, \ n= 0, 1, 2, ... \ $. Then the output data have also this form, i. e., $g(n) = \zeta^{-n}g, \ g \in \cG, \ n= 0, 1, 2, ...$, and the relations (\ref{Nr.2.12}) take the form
\begin{equation}\label{Nr.2.13}
    \left\{\begin{array}{cl}
        (I-\zeta T)h \ = \ \ \zeta F f, \\
        \ \ g \ =  \ G h + S f.
    \end{array}\right.
\end{equation}

For $|\zeta|<1$ the operator $(I-\zeta T)$ is invertible. In this case,
\beql{E2.14}
g = \theta(\zeta)f, \ \ h = \zeta(I-\zeta T)^{-1}Ff,  
\eeq
where
\beql{E2.15}
\theta\rk{\zeta}\defeq  S + \zeta G(I-\zeta T)^{-1}F, \ \ \zeta\in\D. 
\eeq

\begin{defn}\label{d2.11}
    Let $\Delta$ be the unitary colligation of form  \eqref{E2.5}.
    The operator function  \eqref{E2.15} 
    is called \emph{the characteristic operator function (c.o.f) of the colligation $\Delta$}
    or \emph{the transfer function of the system} $\cO_\Delta $.
\end{defn}

Taking into account the described construction, the open system $\cO_\Delta$ can be schematically illustrated in the following way:

\begin{figure}[H]
    \centering
    \begin{tikzpicture}[decoration={
            markings,
            mark=at position 0.5 with {\arrowreversed{latex}}}
        ]
        
        \draw[postaction={decorate}] (0,0) --node[anchor=south] {$g$} (1,0);
        \draw (2,0) circle (1cm);
        \node[label=center:{$h$}] at (2,0) {};
        \node[label=below:{$\cO_\Delta$}] at (2,-1) {};
        \draw[postaction={decorate}] (3,0) --node[anchor=south] {$f$} (4,0);
    \end{tikzpicture}
    \caption{}\label{F0}
\end{figure}

Rewrite  \eqref{Nr.2.13}  as  

\[
\Mat{\zeta^{-1}h\\g} = \Mat{T&F\\ G&S} \Mat{h\\f}  
\]

Then the unitarity of the operator  $Y$ (see \eqref{1.2}) implies

\[
\  |\zeta|^{-2}|h\|^2 + \|g\|^2 = \|h\|^2 + \|f\|^2,
\]

Hence, taking into account \eqref{E2.14} , we find
\[
\|f\|^2 - \|\theta\rk{\zeta}f\|^2 = (1 - |\zeta|^2 ) \|(I-\zeta T)^{-1}Ff \|^2, \ f \in \cH, \ \zeta \in\D,  
\]
or 
\beql{E2.17}
I_{\cF} - \theta^*\rk{\zeta} \theta\rk{\zeta} = (1 - |\zeta|^2 )F^*(I-\ko\zeta T^*)^{-1}(I-\zeta T)^{-1}F, \ \ \zeta \in\D   
\eeq

Hence, the characteristic operator function  $\theta$ belongs to the Schur class
$ S[\cF,\cG]$.  In this connection, it turns out to be important that the converse statement is also true (see \cite{B}):

\begin{thm}[{\cite{B}}]\label{t1.14}
    The characteristic operator function $\theta_\Delta$ of the unitary colligation $\Delta$ belongs to the class  $S[\cF,\cG]$.
    Conversely, suppose that $\theta$ is an operator function belonging to the class   $ S[\cF,\cG]$.
    Then there exists a simple unitary colligation $\Delta$ of form \eqref{E2.5} such that $\theta$ is the characteristic operator function of $\Delta.$
\end{thm}


\begin{defn}\label{d2.13}
    Let $\Delta_k=(\cH_k,\cF, \cG;T_k, F_k,G_k,S_k)$, $k=1,2$, be unitary colligations.
    Then $\Delta_1$ and $\Delta_2$ are called \emph{unitarily equivalent} if $S_1=S_2$ and there exists a unitary operator $Z\colon\cH_1\to\cH_2$ which satisfies $ZT_1=T_2Z$, $ZF_1=F_2$, $G_2Z=G_1$.
\end{defn}

It can be easily seen that the characteristic operator functions of unitarily equivalent colligations coincide.  It turns out that the converse is also true,
namely, the following statement holds true (see \cite{B}). 

\begin{thm}[{\cite{B}}]\label{t1.16}
    If the characteristic operator functions of two simple colligations coincide, then the colligations are unitarily equivalent.
\end{thm}

If $\Delta$ is a unitary colligation of form \eqref{E2.5}, then
\begin{equation}
    \Delta^\sim
    \defeq\rk{\cH,\cG,\cF;T^\ad,G^\ad,F^\ad,S^\ad}
\end{equation}
is a unitary colligation as well.
It is called the colligation \emph{associated} with $\Delta$.
Note that the function $\theta^\sim \in S[\cG,\cF] $ given by $\theta^\sim(\zeta)\defeq\theta^\ad\rk{\overline{\zeta}}$ is called \emph{associated} with the function $\theta\in S[\cF,\cG]$.
Let $\theta_\Delta$ be the characteristic operator function of a unitary colligation $\Delta$.
It is readily checked that the associated function $\theta^{\sim}_\Delta$ is the characteristic function of the associated colligation $\Delta^\sim$, \tie{}, $ \theta_{\Delta^\sim}=
\theta^{\sim}_\Delta$.

\subsection{Product of unitary colligations and factorization of characteristic operator functions}
Let
\beql{E2.18}
\Dl_i=(\cH_i, \cF_i, \cG_i; T_i, F_i, G_i, S_i), \ \ i = 1,2, 
\eeq
be unitary colligations and $\theta_i = \theta_{\Dl_i}, \ i = 1,2.$ Let the input space $\cF_1$ coincide with the output space $\cG_2 : \cF_1=\cG_2$. 
Consider relations (see \eqref{Nr.2.13})

\begin{align*}
    \left\{\begin{array}{cl}
        (I_{\cH_1}-\zeta T_1)h_1 \ = \ \ \zeta F_1 f_1, \\
        \ \ g_1 \ =  \ G_1 h_1 + S_1 f_1,
    \end{array}\right.
\end{align*}

and

\begin{align*}
    \left\{\begin{array}{cl}
        (I_{\cH_2}-\zeta T_2)h_2 \ = \ \ \zeta F_2 f_2, \\
        \ \ g_2 \ =  \ G_2 h_2 + S_2 f_2,
    \end{array}\right.
\end{align*}

which define open systems $\cO_{\Delta_1}$ and $\cO_{\Delta_2}$, respectively. Here

\beql{E2.19}
\ f_i \in \cF_i, \ \ h_i \in \cH_i, \ \ g_i \in \cG_i, \ \ g_i = \theta(\zeta)f_i, \ \ i = 1,2.    
\eeq

If we put 
\beql{E2.20}
\ f_1 = g_2,     
\eeq
then we get a new open system $\cO$ with the input $f_2,$ the internal state $h:= h_1 + h_2,$ and the output $g_1.$ Substituting the expression for $g_2$ in the first system 
instead of $f_1$, we obtain the relations that determine the system  $\cO$ :

\begin{align}\label{Nr.2.21}
    \left\{\begin{array}{cl}
        (I_{\cH_1}-\zeta T_1)h_1 - \zeta F_1G_2h_2  \ = \ \  \zeta F_1S_2f_2, \\
        \ \ (I_{\cH_2}-\zeta T_2)h_2               \ = \ \ \zeta F_2f_2, \\ 
        \ \ g_1 \ =  \ G_1 h_1 + S_1G_2h_2 + S_1S_2 f_2,
    \end{array}\right.
\end{align}
or
\begin{equation}\label{Nr.2.22}
    \left\{\begin{array}{cl}
        (I-\zeta T)h \ = \ \ \zeta F f_2, \\
        \ \ g_1 \ =  \ G h + S f_2,
    \end{array}\right.
\end{equation}
where
\[
T: = \Mat{T_1&F_1G_2\\ 0&T_2} \in [\cH_1 \oplus \cH_2], \ \ 
F: =  \Mat{F_1S_2\\F_2} \in [\cF_2, \cH_1 \oplus \cH_2],         
\]

\[ 
G: = \Mat{G_1,&S_1G_2} \in [\cH_1 \oplus \cH_2, \cG_1], \ \
S: = S_1S_2 \in [\cF_2, \cG_1 ] .        
\]

Denote by $P_1$ and $P_2$ the orthogonal projections of
$\cH_1\oplus\cH_2$ onto $\cH_1$ and $\cH_2$, respectively. Then
\beql{E2.23}
T := T_1P_1 + T_2P_2 + F_1G_2P_2, \ \ \ F := F_2 + F_1S_2,    
\eeq
\beql{E2.24}
G := G_1P_1 + S_1G_2P_2, \ \ \ \ \ \ S := S_1S_2.    
\eeq

The operator
\[
\Mat{T&F\\ G&S} \in [\cH \oplus \cF_2, \cH\oplus \cG_2], \ \ \cH = \cH_1 \oplus  \cH_2,    
\]
is unitary, since
\[
\Mat{T&F\\ G&S} = \Mat{T_1&F_1G_2&F_1S_2\\ 0&T_2&F_2\\ G_1&S_1G_2&S_1S_2} =
\Mat{T_1&0&F_1\\ 0&I_{\cH_2}&0\\ G_1&0&S_1} \Mat{I_{\cH_1}&0&0\\ 0&T_2&F_2\\ 0&G_2&S_2}.      
\]
Thus the 7-tuple
\beql{E2.25}
\Dl:=(\cH, \cF, \cG; T, F, G, S),
\eeq
where $T, F$ and $G, S$ have forms \eqref{E2.23} and \eqref{E2.24}, respectively,
is a unitary colligation and $\cO = \cO_{\Delta}.$
\begin{defn}\label{d2.17} 
    Colligation \eqref{E2.25} constructed as described above 
    is called  \emph{the product of the colligations} $\Delta_1$  \emph{and} $\Delta_2$
    and it will be denoted by the symbol $\Delta_1 \Delta_2$.
    The corresponding open system  $\cO_{\Delta}$ (see \eqref{Nr.2.21})
    is called  \emph{the coupling of the open systems} $\cO_{\Delta_1}$  \emph{and} $\cO_{\Delta_2}$ 
    and it will be denoted by the symbol $\cO_{\Delta_1} \curlyvee \cO_{\Delta_2}$.  
\end{defn}
Taking this into account, the system $\cO_{\Delta} = \cO_{\Delta_1} \curlyvee \cO_{\Delta_2}$ can be schematically illustrated in the following way: 

\begin{figure}[H]
    \centering
    \begin{tikzpicture}[scale=0.8, decoration={
            markings,
            mark=at position 0.5 with {\arrowreversed{latex}}}
        ]
        
        \draw[postaction={decorate}] (0,0) --node[anchor=south] {$g_1$} (1.5,0);
        \draw (3,0) circle (1.5cm);
        \node[label=center:{$h=h_1+h_2$}] at (3,0) {};
        \node[label=below:{$\cO_{\Delta}$}] at (3,-1.5) {};
        \draw[postaction={decorate}] (4.5,0) --node[anchor=south] {$f_2$} (6,0);
        
        \node[label=center:{$=$}] at (6.5,0) {};
        
        \draw[postaction={decorate}] (7,0) --node[anchor=south] {$g_1$} (8.5,0);
        \draw (9,0) circle (0.5cm);
        \node[label=center:{$h_1$}] at (9,0) {};
        \node[label=below:{$\cO_{\Delta_1}$}] at (9,-0.5) {};
        
        \draw[postaction={decorate}] (9.5,0) --node[anchor=south] {$f_1=g_2$} (11.5,0);
        
        \draw (12,0) circle (0.5cm);
        \node[label=center:{$h_2$}] at (12,0) {};
        \node[label=below:{$\cO_{\Delta_2}$}] at (12,-0.5) {};
        \draw[postaction={decorate}] (12.5,0) --node[anchor=south] {$f_2$} (14,0);
    \end{tikzpicture}
    \caption{}\label{F2.1}
\end{figure}

By \emph{a factorization of a colligation} we mean any of its representations
as a product of two colligations. From \eqref{E2.23} it can be
seen that if $\Delta=\Delta_1\Delta_2$, then the subspace $\cH_1$ of
$\cH_1\oplus\cH_2$ is invariant with respect to $T$ and the equality
$T_1 = \rstr{T}{\cH_1}$ holds true.

\bthmnl{{\cite{B}}}{T2.18}
A colligation $\Dl=(\cH, \cF, \cG; T, F, G, S)$ admits a
factorization $\Delta = \Delta_1\Delta_2$,
where $\Delta_i,\ i=1,2,$ have form
\eqref{E2.18} if and only if
the subspace $\cH_1$ of $\cH = \cH_1\oplus\cH_2$ is invariant with respect
to the fundamental operator $T$ of the colligation $\Delta$.
\ethm
By \emph{a factorization of a function} $\theta(\zeta)$, belonging to
$S[\cF,\cG]$, we mean an arbitrary product representation
\beql{E2.26}
\theta(\zeta) = \theta_1(\zeta)\theta_2(\zeta), \ \  \zeta\in\D, 
\eeq
where $\theta_1 \in S[\cE,\cG]$ and $\theta_2 \in S[\cF,\cE]$ with some Hilbert space $\cE$.
\bthmnl{{\cite{B}}}{T2.19}
If a unitary colligation $\Delta$ admits a factorization 
$\Delta =  \Delta_1\Delta_2$, 
then the factorization 
\beql{E2.27}
\theta_\Delta(\zeta) =
\theta_{\Delta_1}(\zeta)\theta_{\Delta_2}(\zeta), \ \zeta\in\D,      
\eeq  
holds true.
\ethm
\bproof On the one hand, it follows from \eqref{Nr.2.22} that $g_1 = \theta_\Delta(\zeta)f_2, \ 
\zeta\in\D$. On the other hand, from \eqref{E2.19} and \eqref{E2.20} we have 
\begin{align*}
    g_1 = \theta_{\Delta_1}(\zeta)f_1 = \theta_{\Delta_1}(\zeta)g_2 = \theta_{\Delta_1}(\zeta)
    \theta_{\Delta_2}(\zeta)f_2, \ \ \zeta\in\D. 
\end{align*}
\eproof  
\begin{defn} \label{d2.20}
    Factorization $($\ref{E2.26}$)$ is called  \emph{regular} if the
    product $\Delta_1\Delta_2$ of two simple colligations such that
    $\theta_{\Delta_1} = \theta_1$ and
    $\theta_{\Delta_2} = \theta_2$ is a simple colligation.
\end{defn}
Note that every factorization  $\Delta =
\Delta_1\Delta_2$ of a simple colligation
$\Delta$ generates a regular factorization of form \eqref{E2.27}.

\bthmnl{{\cite{B}}}{T2.21}
Let the c.o.f. $\theta_{\Delta}(\zeta)$ of a simple colligation
$\Delta$ admit a regular factorization $\theta_\Delta(\zeta) =
\theta_1(\zeta)\theta_2(\zeta)$. Then there exists the unique factorization $\Delta =
\Delta_1\Delta_2$ such that $\theta_{\Delta_1}(\zeta) =
\theta_1(\zeta)$ and $\theta_{\Delta_2}(\zeta)=\theta_2(\zeta)$.
\ethm

\section{The Schur interpolation problem.
    V.~P.~Potapov's approach}\label{sec3-0927}

In this section, we assume that $\cF$ and $\cG$ are finite-dimensional Hilbert spaces,
and let 
\begin{align*}
    \dim\cF&=:q<+\infty,&
    \dim\cG&=:p<+\infty.
\end{align*}

\subsection{Formulation of the Schur problem}\label{subsec3.1-0927}

The well-known Schur problem \cite{Schur} for functions belonging to the class $S[\cF, \cG]$  is formulated as follows:

Given a sequence
\begin{align*}
    c_k&\in\FG ,&k&=0,1,\dotsc,n,
\end{align*}
of $(n+1)$ operators:
\benui
    \item Find a necessary and sufficient condition which ensures that $c_0,c_1,\dotsc,c_n$ are the initial coefficients in the Taylor series representation of some function $\theta\in S[\cF,\cG]$, \tie{},\ 
    \begin{align*}
        \theta\rk{\zeta}&=c_0+c_1\zeta+\dotsb+c_n\zeta^n+\dotsb,&\zeta&\in\D;
    \end{align*}
    \item If such functions exist, give a description of the set of them all.
\eenui
The approach to the study of the Schur problem considered here is based on a generalization of the well-known Schwarz--Pick inequality (see \cite{P}) for a function $\theta\in S[\cF,\cG]$.

\bthmnl{{\cite[Ch.~II]{EP}}}{T3.1}
Let $\zeta_k\in\D$, $k=1,\dotsc,n$.
Then a function $\theta$ which is holomorphic on $\D$ belongs to the class $S[\cF, \cG]$  if and only if the operator inequality
\beql{E3.1}
\Mat{
    \frac{\IH-\theta\rk{\zeta_1}\theta^\ad\rk{\zeta_1}}{1-\zeta_1\ko{\zeta_1}}&\hdots&\frac{\IH-\theta\rk{\zeta_1}\theta^\ad\rk{\zeta_n}}{1-\zeta_1\ko{\zeta_n}}&\frac{\IH-\theta\rk{\zeta_1}\theta^\ad\rk{\zeta}}{1-\zeta_1\ko{\zeta}}\\
    \vdots&&\vdots&\vdots\\
    \frac{\IH-\theta\rk{\zeta_n}\theta^\ad\rk{\zeta_1}}{1-\zeta_n\ko{\zeta_1}}&\hdots&\frac{\IH-\theta\rk{\zeta_n}\theta^\ad\rk{\zeta_n}}{1-\zeta_n\ko{\zeta_n}}&\frac{\IH-\theta\rk{\zeta_n}\theta^\ad\rk{\zeta}}{1-\zeta_n\ko{\zeta}}\\
    \frac{\IH-\theta\rk{\zeta}\theta^\ad\rk{\zeta_1}}{1-\zeta\ko{\zeta_1}}&\hdots&\frac{\IH-\theta\rk{\zeta}\theta^\ad\rk{\zeta_n}}{1-\zeta\ko{\zeta_n}}&\frac{\IH-\theta\rk{\zeta}\theta^\ad\rk{\zeta}}{1-\zeta\ko{\zeta}}
}\geq0
\eeq
is satisfied for each $\zeta\in\D$.
\ethm


\bthmnl{{\cite[Ch.~II]{EP}}}{T3.1'}
Let $\zeta_k\in\D$, $k=1,\dotsc,n$.
Then a function $\theta$ which is holomorphic on $\D$ belongs to the class $S[\cF, \cG]$  if and only if the operator inequality
\beql{E3.1'}
\Mat{
    \frac{\IH-\theta\rk{\zeta_1}\theta^\ad\rk{\zeta_1}}{1-\zeta_1\ko{\zeta_1}}&\hdots&\frac{\IH-\theta\rk{\zeta_1}\theta^\ad\rk{\zeta_n}}{1-\zeta_1\ko{\zeta_n}}&\frac{\theta\rk{\zeta}-\theta\rk{\zeta_1}}{\zeta-\zeta_1}\\
    \vdots&&\vdots&\vdots\\
    \frac{\IH-\theta\rk{\zeta_n}\theta^\ad\rk{\zeta_1}}{1-\zeta_n\ko{\zeta_1}}&\hdots&\frac{\IH-\theta\rk{\zeta_n}\theta^\ad\rk{\zeta_n}}{1-\zeta_n\ko{\zeta_n}}&\frac{\theta\rk{\zeta}-\theta\rk{\zeta_n}}{\zeta-\zeta_n}\\
    \frac{\theta^\ad\rk{\zeta}-\theta^\ad\rk{\zeta_1}}{\ko{\zeta}-\ko{\zeta_1}}&\hdots&\frac{\theta^\ad\rk{\zeta}-\theta^\ad\rk{\zeta_n}}{\ko{\zeta}-\ko{\zeta_n}}&\frac{\IF-\theta^\ad\rk{\zeta}\theta\rk{\zeta}}{1-\ko{\zeta}\zeta}
}\geq0
\eeq
is satisfied for each $\zeta\in\D\setminus\set{\zeta_1,\dotsc,\zeta_n}$.
\ethm

Note that the form of the block operator in the inequality \eqref{E3.1} indicates that this operator acts in the space $ \cG^{n+1}$, where
\[
\cG^1:=  \cG, \ \ \cG^{n+1}
\defeq \cG^n\oplus\cG.
\]
Analogously, the block operator in the inequality \eqref{E3.1'} acts in the space $\cG^n\oplus\cF$.
Inequality \eqref{E3.1'} is usually called the dual of \eqref{E3.1}.


The following result characterizes a Schur function in terms of their Taylor coefficients in the Taylor series representation around the origin.
For scalar Schur functions this result was obtained by I.~Schur {\cite[\csatz{VIII}]{Schur}}. 

\bthmnl{see, \teg{}, {\cite[\cthm{3.1.1}]{DFK92}}}{T3.3}
Let $c_k\in\FG $, $k=0,1,2,\dotsc$
Then a necessary and sufficient condition that the function $\theta\colon\D\to\ek{\cG,\cF}$ admitting the representation
\[
\theta\rk{\zeta}
\defeq c_0+c_1\zeta+c_2\zeta^2+\dotsb+c_n\zeta^n+\dotsb
\]
on $\D$ belongs to $S[\cF, \cG]$  is that the block operator
\beql{E3.2}
C_n
\defeq\Mat{
    c_0&0&\hdots&0\\
    c_1&c_0&\hdots&0\\
    \vdots&\vdots&\ddots&\vdots\\
    c_n&c_{n-1}&\hdots&c_0
}\in\ek{\cF^{n+1},\cG^{n+1}}
\eeq
is contractive for each  $n=0,1,2,\dotsc$, \tie{},
\[
I_{\cG^{n+1}}-C_nC_n^\ad
\geq0.
\]
\ethm
\bproof
Necessity:
Let  $\theta\in S[\cF,\cG]$.
We fix $\zeta\in\D$.
Set
\beql{E3.5}
\La_\cG\rk{\zeta}
\defeq\rk{I_\cG,\zeta I_\cG,\zeta^2I_\cG,\dotsc,\zeta^n I_\cG,\dotsc}.
\eeq
Then
\[
\La_\cG\rk{\zeta}\La_\cG^\ad\rk{\zeta}
=\frac{1}{1-\zeta\ko\zeta}I_\cG.
\]
From this we obtain
\beql{T3.3-2}
\frac{I_\cG-\theta\rk{\zeta}\theta^\ad\rk{\zeta}}{1-\zeta\ko\zeta}
=\La_\cG\rk{\zeta}\La_\cG^\ad\rk{\zeta}-\theta\rk{\zeta}\La_\cF\rk{\zeta}\La_\cF^\ad\rk{\zeta}\theta^\ad\rk{\zeta}.
\eeq
Because of
\[
\theta\rk{\zeta}\La_\cF\rk{\zeta}
=\La_\cG\rk{\zeta}C_\infty,
\]
where
\[
C_\infty
\defeq
\Mat{
    c_0&0&0&\hdots&0&0&\hdots\\
    c_1&c_0&0&\hdots&0&0&\hdots\\
    c_2&c_1&c_0&\hdots&0&0&\hdots\\
    \vdots&\vdots&\vdots& &\vdots&\vdots\\
    c_n&c_{n-1}&c_{n-2}&\hdots&c_0&0&\hdots\\
    c_{n+1}&c_n&c_{n-1}&\hdots&c_1&c_0&\hdots\\
    \vdots&\vdots&\vdots& &\vdots&\vdots
},
\]
identity \eqref{T3.3-2} can be rewritten as
\beql{T3.3-4}
\frac{I_\cG-\theta\rk{\zeta}\theta^\ad\rk{\zeta}}{1-\zeta\ko\zeta}
=\La_\cG\rk{\zeta}\rk{I_{\cG^{\infty}}-C_\infty C_\infty^\ad}\La_\cG^\ad\rk{\zeta},
\eeq
where

\[
\cG^{\infty} :=  
\cG\oplus\cG...\oplus\cG\oplus...
\]

Now we have to take into account that $\zeta$ varies over $\D$.

Sufficiency:
Since the operators $C_n$, $n=0,1,2,\dots$, are contractive, we obtain 
\[
c_0c_0^\ad+c_1c_1^\ad+\dotsb+c_nc_n^\ad
\leq I_\cG, \ \ n=0,1,2,\dotsc 
\]
This means that the function
\begin{align*}
    \theta\rk{\zeta}&\defeq c_0+c_1\zeta+\dotsb+c_n\zeta^n+\dotsb, \ \
    \zeta\in\D,
\end{align*}
is holomorphic on $\D$.
The belonging of $\theta$ to the class $S[\cF, \cG]$  follows now from \eqref{T3.3-4}.
\eproof

This theorem is completed by the following result:

\bthmnl{see, \teg{}, {\cite[\csec{3.5}]{DFK92}}}{T3.4}
Let the sequence  $c_k\in\FG $, $k=0,1,2,\dotsc,n,$ be such that
\begin{align}\label{E3.1000}
    I-C_nC_n^\ad\geq0\ ,
\end{align} 
where $C_n$ is given by \eqref{E3.2}.
Then there exists some function  $\theta\in S[\cF,\cG]$ with the initial Taylor coefficients $c_0,c_1,\dotsc,c_n$, \tie{},\
\begin{align}\label{E3.3}
    \theta\rk{\zeta}&=c_0+c_1\zeta+\dotsb+c_n\zeta^n+\dotsb,&\zeta&\in\D.
\end{align}
\ethm

The message of this theorem is that the operator sequence $\rk{c_k}_{k=0}^n$ given there can always be extended to an infinite sequence $\rk{c_k}_{k=0}^\infty$ such that for each $m\in\N$ the operator $C_m$ given via \eqref{E3.2} is contractive.
Thus, in view of \rthm{T3.3}, the function $\theta$ with Taylor series representation \eqref{E3.3} belongs to
$S[\cF, \cG]$ .

\rthmss{T3.3}{T3.4} provide an answer to the first question that was raised in Schur's problem.
Indeed, the condition \eqref{E3.1000}
is necessary and sufficient that the operators $c_0,c_1,\dotsc,c_n$ are the initial Taylor coefficients of some function  $\theta\in S[\cF,\cG]$.

Let the sequence  $c_k\in\FG $, $k=0,1,2,\dotsc,n,$ be such that the condition
\eqref{E3.1000} is satisfied. We will denote the set of the functions $\theta \in S[\cF, \cG] $ with Taylor series representation \eqref{E3.3} (the set of the solutions of the corresponding Schur problem) by $S[\cF, \cG;\rk{c_k}_{k=0}^n].$  

\subsection{The fundamental matrix inequalities}
The main feature of V.~P.~Potapov's approach to matricial versions of classical interpolation problems is that some matrix inequality  is associated with each such problem (see \cite{MR703593,MR409825}).
This inequality contains full information about the interpolation problem under consideration, whereas the solutions of these matrix inequalities provide a complete description of the solutions of the 
corresponding interpolation problem.

\bthmn{{\cite{G}, \cite[\cpart{I}]{D82},  \cite[Ch.~5]{DFK92}}}  
Let $n\in\NO$ and let $\seqa{c}{n}$ be a sequence from $\FG $.
Let $\theta$ be a holomorphic function on $\D$ with values in $\FG $.
Then $\theta\in S[\cF, \cG;\rk{c_k}_{k=0}^n]$ 
if and only if the matrix inequality
\beql{E3.12}
\begin{pmat}({|})
    \Iu{\cG^{(n+1)}}-C_nC_n^\ad&{\begin{gathered}
            \Iu{\cG}-c_0\theta^\ad\rk{\zeta}\cr
            \ko\zeta\rk{\Iu{\cG}-\rk{c_0+\frac{c_1}{\ko\zeta}}\theta^\ad\rk{\zeta}}\cr
            \vdots\cr
            \ko\zeta^{n+1}\rk{\Iu{\cG}-\rk{c_0+\frac{c_1}{\ko\zeta}+\dotsb+\frac{c_n}{\ko\zeta^n}}\theta^\ad\rk{\zeta}}\cr
        \end{gathered}}\cr\-
    \ast&\frac{\Iu{\cG}-\theta\rk{\zeta}\theta^\ad\rk{\zeta}}{1-\zeta\ko\zeta}\cr
\end{pmat}\geq0
\eeq
is satisfied for $\zeta\in\DO $.
\ethm

\bthmn{{\cite{G}, \cite[\cpart{I}]{D82}, \cite[Ch.~5]{DFK92}}}
Let $n\in\NO$ and let $\seqa{c}{n}$ be a sequence from $\FG $.
Let $\theta$ be a holomorphic function on $\D$ with values in $\FG $.
Then $\theta\in S[\cF, \cG;\rk{c_k}_{k=0}^n]$ 
if  and only if the 
\beql{E3.13}
\begin{pmat}({|})
    \Iu{\cG^{(n+1)}}-C_nC_n^\ad&{\begin{gathered}
            \zeta^\inv\rk{\theta\rk{\zeta}-c_0}\cr
            \zeta^{-2}\rk{\theta\rk{\zeta}-c_0-c_1\zeta}\cr
            \vdots\cr
            \zeta^{-\rk{n+1}}\rk{\theta\rk{\zeta}-c_0-c_1\zeta-\dotsb-c_n\zeta^n}\cr
        \end{gathered}}\cr\-
    \ast&\frac{\Iu{\cF}-\theta^\ad\rk{\zeta}\theta\rk{\zeta}}{1-\ko\zeta\zeta}\cr
\end{pmat}\geq0
\eeq
is satisfied for $\zeta\in\DO $.
\ethm

Inequalities  \eqref{E3.12} and  \eqref{E3.13} are called \emph{the fundamental matrix inequalities} (FMI's) of the Schur problem.
These inequalities can be obtained from inequalities 
\eqref{E3.1} and  \eqref{E3.1'}, respectively.
FMI \eqref{E3.13} is called  \emph{dual} of \eqref{E3.12}. 

We note that all results of this and the following section were obtained in \cite{G} for the case $\dim\cF=\dim\cG$.
In \zitaa{D82}{\cpart{I}} these results were generalized to the case of arbitrary finite-dimensional spaces $\cF$ and $\cG$.
A detailed treatment can be found in \zitaa{DFK92}{\cch{4 and~5}}.

The lemma of the non-negative Hermitian block matrix (see \zitaa{EP}{\cch{II}}, \zitaa{DFK92}{\csec{1.1}}) plays an important role in the course of solving of FMI's.

\bleml{L3.7}
Let $\cL$ and $\cN$ be finite-dimensional Hilbert spaces, let $A\in\ek{\cL}$, $B\in\ek{\cN,\cL}$ and $D\in\ek{\cN}$.
Then the block operator
\[
E
=
\Mat{
    A&B\\
    B^\ad&D
}\in\ek{\cL\oplus\cN}
\]
is non-negative if and only if
\begin{enuin}
    \il{L3.7-1} $A\geq0$;
    \il{L3.7-2} the equation $AX=B$ has at least one solution;
    \il{L3.7-3} for each solution $X$ of $AX=B$ the inequality
    \[
    D-X^\ad AX
    \geq0
    \]
    holds true, where $X^\ad AX$ is independent of the choice of $X$.
\end{enuin}
\elem

We note that in FMI's \eqref{E3.12} and \eqref{E3.13} the block $A$ is equal to $\Iu{\cG^{n+1}}-C_nC_n^\ad$ and contains all information about the data of the considered Schur problem because \rthm{T3.3} shows that the inequality $A\geq0$ is necessary and sufficient for the solvability of the corresponding problem.
For this reason, the block $\Iu{\cG^{n+1}}-C_nC_n^\ad$ is said to be \emph{the information block} of FMI's \eqref{E3.12} and \eqref{E3.13}.

\bdefnl{D3.8}
The Schur interpolation problem associated with the given sequence $\seqaj{c}{k}{n}$ of operators belonging to $\FG $ is said to be \emph{non-degenerate} if
\begin{align}\label{E3.16}
    \Iu{\cG^{n+1}}-C_nC_n^\ad&>0,&
\end{align}
where $C_n$ is given via \eqref{E3.2}.
It is called \emph{degenerate} if
\begin{align*}
    \Iu{\cG^{n+1}}-C_nC_n^\ad&\geq0&
    &\text{and}&
    \ker{\rk{\Iu{\cG^{n+1}}-C_nC_n^\ad}}&\neq\set{0}.
\end{align*}
\edefn

In this paper, we consider only non-degenerate Schur interpolation problems.
Degenerate Schur interpolation problems are considered analogously, however require rather complicated methods (see, \teg{},\ \zitaa{D82}{\cpart{IV}}).

In the case of the non-degenerate Schur problem, conditions~\eqref{L3.7-1} and~\eqref{L3.7-2} are automatically satisfied, and the inequality $E\geq0$ is equivalent to the inequality
\beql{E3.17}
D-B^\ad A^\inv B
\geq0.
\eeq
In accordance with \eqref{E3.5}, let
\[
\La_{\cG,n}\rk{\zeta}
\defeq\rk{I_\cG,\zeta I_\cG,\zeta^2I_\cG,\dotsc,\zeta^n I_\cG}, \ n = 0,1,2,\dotsc \ .
\]
Then in the case of the FMI \eqref{E3.13} we have
\begin{align}
    B&:=
    \Mat{
        \zeta^\inv\rk{\theta\rk{\zeta}-c_0}\\
        \zeta^{-2}\rk{\theta\rk{\zeta}-c_0-c_1\zeta}\\
        \vdots\\
        \zeta^{-n-1}\rk{\theta\rk{\zeta}-c_0-c_1\zeta-\dotsb-c_n\zeta^n}
    },\label{E3.19}\\
    D&:=\frac{\Iu{\cG}-\theta^\ad\rk{\zeta}\theta\rk{\zeta}}{1-\ko\zeta\zeta}
    =\frac{1}{1-\ko\zeta\zeta}\matb{\theta^\ad\rk{\zeta},\Iu{\cF}}\tilde{j}\Mat{\theta\rk{\zeta}\\\Iu{\cF}}, \label{E3.20}
\end{align}
where
\beql{E3.21}
\tilde{j}
\defeq
\Mat{
    -\Iu{\cG}&0\\
    0&\Iu{\cF}
}.
\eeq
For this reason, in the non-degenerate case FMI \eqref{E3.13} is equivalent to the inequality \eqref{E3.17} which in this case has the form
\begin{multline}\label{E3.22}
\matb{\theta^\ad\rk{\zeta},\Iu{\cF}}\\
\times\rkb{\frac{1}{1-\ko\zeta\zeta}\tilde{j}-\tilde{j}\Mat{\La_{\cG,n}\rk{\frac{1}{\ko\zeta}}&0\\0&\La_{\cF,n}\rk{\frac{1}{\ko\zeta}}}\tilde{H}_n\Mat{\La^\ad_{\cG,n}\rk{\frac{1}{\ko\zeta}}&0\\0&\La^\ad_{\cF,n}\rk{\frac{1}{\ko\zeta}}}\tilde{j}}\\
\times\Mat{\theta\rk{\zeta}\\\Iu{\cF}}
\geq0
\end{multline}
for $\zeta\in\D$, where
\[
\tilde{H}_n
\defeq
\Mat{\Iu{\cG^{n+1}}\\C_n^\ad}\rk{\Iu{\cG^{n+1}}-C_nC_n^\ad}^\inv\mat{\Iu{\cG^{n+1}},C_n}.
\]
Analogously, we obtain that in the non-degenerate case FMI \eqref{E3.12} is equivalent to the inequality
\begin{multline}\label{E3.24}
\matb{\theta\rk{\zeta},\Iu{\cG}}\\
\times\rkb{\frac{1}{1-\zeta\ko\zeta}j-j\Mat{\La_{\cF,n}\rk{\zeta}&0\\0&\La_{\cG,n}\rk{\zeta}}H_n\Mat{\La^\ad_{\cF,n}\rk{\zeta}&0\\0&\La^\ad_{\cG,n}\rk{\zeta}}j}\\
\times\Mat{\theta^\ad\rk{\zeta}\\\Iu{\cG}}
\geq0
\end{multline}
for $\zeta\in\D$, where
\beql{E3.25}
j
\defeq
\Mat{
    -\Iu{\cF}&0\\
    0&\Iu{\cG}
}
\eeq
and
\beql{E3.26}
H_n
\defeq\Mat{C_n^\ad\\\Iu{\cG^{n+1}}}\rk{\Iu{\cG^{n+1}}-C_nC_n^\ad}^\inv\mat{C_n,\Iu{\cG^{n+1}}}.
\eeq
The inequalities \eqref{E3.22} and \eqref{E3.24} are equivalent to FMI's \eqref{E3.12} and \eqref{E3.13}, respectively.
They are solved with the aid of the so-called \tJ{elementary} multiple factor.
These questions will be handled in the next Section.

Now we turn our attention to the structure of the blocks $A$, $B$ and $D$ of FMI's \eqref{E3.12} and \eqref{E3.13}.
As we noted already earlier, the block $A=\Iu{\cG^{n+1}}-C_nC_n^\ad$ contains all the information on the data of the corresponding Schur problem.
The block $D$ has the form
\begin{align*}
    &\frac{\Iu{\cG}-\theta\rk{\zeta}\theta^\ad\rk{\zeta}}{1-\zeta\ko\zeta}&
    &\text{and}&
    &\frac{\Iu{\cF}-\theta^\ad\rk{\zeta}\theta\rk{\zeta}}{1-\ko\zeta\zeta}
\end{align*}
for FMI \eqref{E3.12} and \eqref{E3.13}, respectively.
The condition $D\geq0$ for $\zeta\in\D$ ensures that the function $\theta$ belongs to the Schur class $\SFD{\cF}{\cG}$.
Finally, the block $B$ is responsible for the fact that the function $\theta$ has the desired initial $n+1$ Taylor coefficients.
For example, what concerns FMI \eqref{E3.13}, this follows immediately from inequality \eqref{E3.17}.
In order to see this, it is sufficient to pass to the limit in this inequality as $\zeta\to0$ and to take into account that the blocks $A$ and $D$ of shape \eqref{E3.20} are bounded.
Hence, on tending $\zeta\to0$, the block $B$ of shape \eqref{E3.19} has to be bounded.

\subsection{$J$-contractive operator functions in the Schur problem.
    Solution of the fundamental matrix inequalities}


Let $\cE$ be a finite-dimensional Hilbert space of dimension $N$ and $J$ a signature operator on $\cE$,  \tie{},\
\[
J
=J^\ad
=J^\inv.
\]
We consider an analytic operator function $\sB$ with values in $\ek{\cE}$ and with exactly one pole of arbitrary multiplicity in the extended complex plane which is located depending on the situation at the point $\zeta=0$ or $\zeta=\infty$.
Furthermore, we assume that the function $\sB$ is $J$\nobreakdash-expansive in the unit disk and $J$\nobreakdash-unitary on its boundary, namely,
\begin{align}
    \sB^\ad\rk{\zeta}J\sB\rk{\zeta}-J&\geq0,&\abs{\zeta}&<1,\label{E3.27}\\
    \sB^\ad\rk{\zeta}J\sB\rk{\zeta}-J&=0,&\abs{\zeta}&=1,\label{E3.25a}
\end{align}
or, equivalently (see, \teg{}, \cite[Lemma~1.3.6.]{DFK92}),
\begin{align*}
    \sB\rk{\zeta}J\sB^\ad\rk{\zeta}-J&\geq0,&\abs{\zeta}&<1,\\
    \sB\rk{\zeta}J\sB^\ad\rk{\zeta}-J&=0,&\abs{\zeta}&=1.
\end{align*}
Such an operator function is called a \emph{$J$\nobreakdash-elementary multiple factor, $J$\nobreakdash-expansive on the unit disk}.
In the sequel, $J$\nobreakdash-elementary multiple factors $\sB$ will be normalized as the identity operator at the point $\zeta=1$, \tie{},\
\[
\sB\rk{1}
=\Iu{\cE}.
\]

We associate with inequality \eqref{E3.13} the operator $J:=\tilde{j} \in [\cE],$ 
$\cE := \cG \oplus \cF$, of form \eqref{E3.21}
which naturally appears by the transformation of FMI \eqref{E3.13} 
to equivalent inequality \eqref{E3.22}. 
Applying the same argumentation, we associate with inequality \eqref{E3.12} the operator 
$J:=j  \in [\cE], \ \cE := \cF \oplus \cG $, of form \eqref{E3.25}.

Let $J:=j$, of form \eqref{E3.25} and let $\sB_n$ be a $j$\nobreakdash-elementary multiple factor with a single pole at the point $\zeta=0$, \tie{},\
\beql{E1213}
\sB_n\rk{\zeta}
:=\frac{d_{n+1}}{\zeta^{n+1}}+\dotsb+\frac{d_1}{\zeta}+d_0.
\eeq
Then from \eqref{E3.27} and shape \eqref{E3.25} of the operator $j$ it follows that $\dim\rk{\ran\rk{d_{n+1}}}\leq p$ ($=\dim\cG$).
In the case when this dimension is equal to $p$, we call $\sB_n$ a $J$\nobreakdash-\emph{elementary multiple factor of full rank}.

The connection of $j$\nobreakdash-elementary multiple factors with the Schur problem will be clear from the following result:

\bthmnl{{on a parametrization) ( \cite{G}, \cite[\cpart{I}]{D82}}}{T3.9}
Let $\sB_n$ be a $j$\nobreakdash-elementary multiple factor of form \eqref{E1213} of full rank.
Then there is a uniquely determined sequence $\seqa{c}{n}$ of operators belonging to $\FG $ such that condition \eqref{E3.16} is satisfied
and such that $\sB_n$ admits the representation
\begin{multline}\label{E3.33}
    \sB_n\rk{\zeta}
    =\Mat{\Iu{\cF}&0\\0&\Iu{\cG}}\\
    +\frac{1-\zeta}{\zeta}j\Mat{\La_{\cF,n}\rk{1}&0\\0&\La_{\cG,n}\rk{1}}H_n\Mat{\La^\ad_{\cF,n}\rk{\frac{1}{\ko\zeta}}&0\\0&\La^\ad_{\cG,n}\rk{\frac{1}{\ko\zeta}}},\\
    n = 0,1,2,\dotsc,
\end{multline}
where  $ H_n$ is the operator of form \eqref{E3.26} 

Conversely, every function $\sB_n$ of this form is a $j$\nobreakdash-elementary multiple factor of full rank and its $j$\nobreakdash-form is given by
\begin{multline}\label{E3.28}
    \frac{1}{1-\zeta\ko\zeta}\rkb{\sB_n^\ad\rk{\zeta}j\sB_n\rk{\zeta}-j}\\
    =\frac{1}{\zeta\ko\zeta}\Mat{\La_{\cF,n}\rk{\frac{1}{\ko\zeta}}&0\\0&\La_{\cG,n}\rk{\frac{1}{\ko\zeta}}}H_n\Mat{\La^\ad_{\cF,n}\rk{\frac{1}{\ko\zeta}}&0\\0&\La^\ad_{\cG,n}\rk{\frac{1}{\ko\zeta}}}.
\end{multline}
\ethm

For solving FMI \eqref{E3.13} we need a $\tilde j$\nobreakdash-elementary multiple factor of full rank with a pole at $\zeta=\infty$, namely
\begin{align}\label{E3.29}
    \tilde \sB_n\rk{\zeta}
    &:=\tilde d_{n+1}\zeta^{n+1}+\dotsb+\tilde d_1\zeta+\tilde d_0, \ \ \ \
    \dim\rk{\ran\rk{\tilde d_{n+1}}}= p.
\end{align}
Evidently, $\sB^\ad\rk{\frac{1}{\ko\zeta}}$ is a $\rk{-\tilde j}$\nobreakdash-elementary multiple factor of full rank with a pole at $\zeta=0$.
Therefore, it follows from \rthm{T3.9} that the $c_k\in\FG $, $k=0,1,\dotsc,n,$ are uniquely determined so that
\begin{multline*}
    \theta^\ad\rkb{\frac{1}{\ko\zeta}}
    =\Mat{\Iu{\cF}&0\\0&\Iu{\cG}}\\
    -\frac{1-\zeta}{\zeta}\tilde j\Mat{\La_{\cF,n}\rk{1}&0\\0&\La_{\cG,n}\rk{1}}\tilde H_n\Mat{\La_{\cF,n}\rk{\frac{1}{\ko\zeta}}&0\\0&\La_{\cG,n}\rk{\frac{1}{\ko\zeta}}},
\end{multline*}
where
\[
\tilde{H}_n
=
\Mat{\Iu{\cG^{n+1}}\\C_n^\ad}\rk{\Iu{\cG^{n+1}}-C_nC_n^\ad}^\inv\mat{\Iu{\cG^{n+1}},C_n}.
\]
Hence, we obtain
\begin{multline}\label{E1214}
    \tilde \sB_n\rk{\zeta}
    =\Mat{\Iu{\cG}&0\\0&\Iu{\cF}}\\
    +\rk{1-\zeta}\Mat{\La_{\cG,n}\rk{\zeta}&0\\0&\La_{\cF,n}\rk{\zeta}}\tilde H_n\Mat{\La_{\cG,n}^\ad\rk{\zeta}&0\\0&\La_{\cF,n}^\ad\rk{\zeta}},\\
    n = 0,1,2,\dotsc\ .
\end{multline}
Consequently, we have the following conclusion.

\bthmnl{{on a parametrization) ( \cite{G}, \cite[\cpart{I}]{D82}}}{T3.10}
Let $\tilde \sB_n$ be a $j$\nobreakdash-elementary multiple factor of form \eqref{E3.29} of full rank.
Than there is a uniquely determined sequence $\seqa{c}{n}$ of operators belonging to $\FG $ such that condition \eqref{E3.16} is satisfied and such that $\tilde \sB_n$ admits representation \eqref{E1214}.

Conversely, every function $\tilde \sB_n$ of this form is a $\tilde j$\nobreakdash-elementary multiple factor of full rank and its $\tilde j$\nobreakdash-form is given by
\beql{E3.32}
\frac{1}{1-\zeta\ko\zeta}\rkb{\tilde \sB_n\rk{\zeta}\tilde j\tilde \sB_n^\ad\rk{\zeta}-\tilde j}
=\Mat{\La_{\cG,n}\rk{\zeta}&0\\0&\La_{\cF,n}\rk{\zeta}}\tilde H_n\Mat{\La_{\cG,n}^\ad\rk{\zeta}&0\\0&\La_{\cF,n}^\ad\rk{\zeta}}.
\eeq
\ethm

A full description of the set of solutions of FMI's \eqref{E3.12} and \eqref{E3.13} in the non-degenerate case is given in the following two theorems.

\bthmnl{{\cite{G}, \cite[\cpart{I}]{D82}}}{T3.11}
Let a sequence $\seq{c_k}{k}{0}{n}$ of operators belonging to $\FG $ satisfy condition \eqref{E3.16} and let $\sB_n$ be the $j$\nobreakdash-elementary multiple factor of form \eqref{E3.33}.
In accordance with \eqref{E3.33}, we consider the block partition
\beql{E3.33a}
\sB_n
=
\Mat{
    a_n&b_n\\
    c_n&d_n
}.
\eeq
Then the general solution $\theta$ of FMI \eqref{E3.12} admits a description as the linear fractional transformation
\begin{align}\label{E3.34}
    \theta\rk{\zeta}&=\rkb{w\rk{\zeta}b_n\rk{\zeta}+d_n\rk{\zeta}}^\inv\rkb{w\rk{\zeta}a_n\rk{\zeta}+c_n\rk{\zeta}},&\zeta&\in\DO ,
\end{align}
where a parameter function $w$ runs over $S[\cF,\cG]$.
\ethm

\bthmnl{{\cite{G}, \cite[\cpart{I}]{D82}}}{T3.12}
Let a sequence $\seq{c_k}{k}{0}{n}$ of operators belonging to $\FG $ satisfy condition \eqref{E3.16} and let $\tsB_n$ be the $\tilde j$\nobreakdash-elementary multiple factor
of form \eqref{E1214}.
In accordance with \eqref{E1214}, we consider the block partition
\beql{E3.35}
\tsB_n
=
\Mat{
    \tilde a_n&\tilde b_n\\
    \tilde c_n&\tilde d_n
}.
\eeq
Then the general solution $\theta$ of FMI \eqref{E3.13} admits a description as linear fractional transformation
\begin{align}\label{E3.36}
    \theta\rk{\zeta}&=\rkb{\tilde a_n\rk{\zeta}w\rk{\zeta}+\tilde b_n\rk{\zeta}}\rkb{\tilde c_n\rk{\zeta}w\rk{\zeta}+\tilde d_n\rk{\zeta}}^\inv,&\zeta&\in\D,
\end{align}
where a parameter function $w$ runs over $ S[\cF,\cG]$.
\ethm

For the proofs of \rthmss{T3.11}{T3.12}, we note that from \eqref{E3.25a} it follows
\begin{align*}
    \sB_n\rk{\zeta}&=J\sB_n^{\ad\inv}\rkb{\frac{1}{\ko\zeta}}J,&\abs{\zeta}&=1.
\end{align*}
Since on both sides of the equality the functions are analytic,  we obtain
\begin{align*}
    \sB_n\rk{\zeta}&=J\sB_n^{\ad\inv}\rkb{\frac{1}{\ko\zeta}}J,&\zeta&\in\C\setminus\set{0}.
\end{align*}
From this and \eqref{E3.28} we get
\begin{multline}\label{E3.37}
    \frac{\sB_n^\inv\rk{\zeta}j\sB_n^{\ad\inv}\rk{\zeta}}{1-\zeta\ko\zeta}\\
    =\frac{1}{1-\zeta\ko\zeta}j-j\Mat{\La_{\cF,n}\rk{\zeta}&0\\0&\La_{\cG,n}\rk{\zeta}}H_n\Mat{\La^\ad_{\cF,n}\rk{\zeta}&0\\0&\La^\ad_{\cG,n}\rk{\zeta}}j.
\end{multline}
Analogously, from \eqref{E3.32} it follows that
\begin{multline}\label{E3.38}
    \frac{\tsB_n^{\ad\inv}\rk{\zeta}\tilde j\tilde \sB_n^\inv\rk{\zeta}}{1-\ko\zeta\zeta}\\
    =\frac{1}{1-\ko\zeta\zeta}\tilde j-\tilde j\Mat{\La_{\cG,n}\rk{\frac{1}{\ko\zeta}}&0\\0&\La_{\cF,n}\rk{\frac{1}{\ko\zeta}}}\tilde H_n\Mat{\La_{\cG,n}^\ad\rk{\frac{1}{\ko\zeta}}&0\\0&\La_{\cF,n}^\ad\rk{\frac{1}{\ko\zeta}}}\tilde j.
\end{multline}
For this reason, inequalities \eqref{E3.24}, \eqref{E3.22} and, hence, also FMI's \eqref{E3.12}, \eqref{E3.13} are equivalent to the inequalities
\begin{align}
    \matb{\theta\rk{\zeta},\Iu{\cG}}\frac{\sB_n^\inv\rk{\zeta}j\sB_n^{\ad\inv}\rk{\zeta}}{1-\zeta\ko\zeta}\Mat{\theta^\ad\rk{\zeta}\\\Iu{\cG}}&\geq0,&\zeta\in\D,\label{E3.39}\\
    \matb{\theta^\ad\rk{\zeta},\Iu{\cF}}\frac{\tsB_n^{\ad\inv}\rk{\zeta}\tilde j\tilde\sB_n^\inv\rk{\zeta}}{1-\ko\zeta\zeta}\Mat{\theta\rk{\zeta}\\\Iu{\cF}}&\geq0,&\zeta\in\DO ,\label{E3.40}
\end{align}
respectively.
We describe the set of solutions of inequality \eqref{E3.40}.

Let $\theta$ be a holomorphic function on $\D$ with values in $\FG $ which satisfies this inequality.
We define a pair of holomorphic operator functions 
$p$ and $q$ on $\D$ with values in $[\cF,\cG]$ and $[\cF]$, respectively, 
by setting
\begin{align}\label{E3.41}
    \Mat{p\rk{\zeta}\\q\rk{\zeta}}&\defeq\tilde\sB_n^\inv\rk{\zeta}\Mat{\theta\rk{\zeta}\\\Iu{\cF}},&\zeta&\in\D.
\end{align}
Now inequality \eqref{E3.40} can be rewritten as
\begin{align}\label{E3.42}
    p^\ad\rk{\zeta}p\rk{\zeta}&\leq q^\ad\rk{\zeta}q\rk{\zeta},&\zeta&\in\D.
\end{align}
From this and \eqref{E3.41} the invertibility of $q\rk{\zeta}$ follows for $\zeta\in\D$.
However,  then it follows from \eqref{E3.42} that the function $w$ given by
\begin{align*}
    w\rk{\zeta}&\defeq p\rk{\zeta}q^\inv\rk{\zeta},&\zeta&\in\D,
\end{align*}
belongs to $S[\cF,\cG]$. 
Taking into account block partition \eqref{E3.35}, we rewrite \eqref{E3.41} as
\begin{align*}
    \theta\rk{\zeta}&=\tilde a_n\rk{\zeta}p\rk{\zeta}+\tilde b_n\rk{\ze}q\rk{\ze},\\
    \Iu{\cF}&=\tilde c_n\rk{\zeta}p\rk{\zeta}+\tilde d_n\rk{\ze}q\rk{\ze}.
\end{align*}
Thus,
\[\begin{split}
    \theta\rk{\ze}
    &=\rkb{\tilde a_n\rk{\ze}p\rk{\ze}+\tilde b_n\rk{\ze}q\rk{\ze}}\rkb{\tilde c_n\rk{\ze}p\rk{\ze}+\tilde d_n\rk{\ze}q\rk{\ze}}^\inv\\
    &=\rkb{\tilde a_n\rk{\ze}w\rk{\ze}+\tilde b_n\rk{\ze}}\rkb{\tilde c_n\rk{\ze}w\rk{\ze}+\tilde d_n\rk{\ze}}^\inv
\end{split}\]
and formula \eqref{E3.36} is proved.

Conversely, let $w\in S[\cF,\cG]$.
We consider the pair of holomorphic operator functions $r$ and $s$ on $\D$ with values in $[\cF,\cG]$ and $[\cF]$, respectively, which is defined by
\begin{align}\label{E3.43}
    \Mat{r\rk{\zeta}\\s\rk{\zeta}}&\defeq\tilde\sB_n\rk{\zeta}\Mat{w\rk{\zeta}\\\Iu{\cF}},&\zeta&\in\D.
\end{align}
For $\zeta\in\D$ we conclude from the $\tilde j$\nobreakdash-expanding property of $\sB_n$ (see \eqref{E3.27}) that
\beql{E3.44}\begin{split}
    \matb{r^\ad\rk{\ze},s^\ad\rk{\ze}}\tilde j\Mat{r\rk{\ze}\\s\rk{\ze}}
    &=\matb{w^\ad\rk{\ze},\Iu{\cF}}\tsB_n^\ad\rk{\ze}\tilde j\tsB_n\rk{\ze}\Mat{w\rk{\ze}\\\Iu{\cF}}\\
    &\geq\matb{w^\ad\rk{\ze},\Iu{\cF}}\tilde j\Mat{w\rk{\ze}\\\Iu{\cF}}
    =\Iu{\cF}-w^\ad\rk{\ze}w\rk{\ze}
    \geq0.
\end{split}\eeq
This means
\begin{align*}
    s^\ad\rk{\zeta}s\rk{\zeta}&\geq r^\ad\rk{\zeta}r\rk{\zeta},&\zeta&\in\D.
\end{align*}
As above, we obtain now that $s\rk{\ze}$ is invertible for $\ze\in\D$.
The function $\theta\defeq rs^\inv$ is holomorphic on $\D$ and admits the representation \eqref{E3.36}.
As it follows from \eqref{E3.43} and \eqref{E3.44}, the function $\theta$ satisfies inequality \eqref{E3.40} and, consequently, inequality \eqref{E3.13} as well.
\rthm{T3.12} is completely proved.
Analogously, \rthm{T3.11} can be proved using inequality \eqref{E3.39}.
We note that in \cite{Dub83} a parametrization of a $j$\nobreakdash-elementary (\tresp{}\ $\tilde j$\nobreakdash-elementary) multiple factor of a non-full rank was given.
This parametrization enabled us to state a description of the set of solutions of the degenerate Schur problem in a form analogous to \eqref{E3.34} (\tresp{}\ \eqref{E3.36}), where the parameter $w\in S[\cF,\cG]$ has to satisfy a special condition (see \cite[\cch{IV}]{D98}).

Let
\[
\tsB_n\gk{w}
\defeq\rk{\tilde a_nw+\tilde b_n}\rk{\tilde c_nw+\tilde d_n}^\inv.
\]
Then \eqref{E3.36} can be rewritten in the form
\[
\theta
=\tsB_n\gk{w},
\]
where, as above, any function $w\in S[\cF,\cG]$ can be taken as a parameter.

\subsection{Stepwise solution of Schur problem.
    Schur parameters}\label{S3.4}

In \cite[\cpart{I}]{D82}, adapting the remarkable idea of I.~Schur \cite{Schur}, the Schur problem is solved recursively assuming that condition \eqref{E3.16} is satisfied.

\paragraph{1.}
We first treat the Schur problem in the simplest case when only the constant term $c_0$ is given.
It follows from \rthm{T3.12} that the operator functions $\theta\in S[\cF,\cG]$ having the operator $c_0$ as constant term in the power series representation
\[
\theta\rk{\ze}
=c_0+\dotsb
\]
are determined by the formula
\beql{E3.46}
\theta
:=\tilde b_0\gk{w},
\eeq
where $\tilde b_0$ has form \eqref{E1214} and $w\in S[\cF,\cG]$.

\paragraph{2.}
Now we look amongst so obtained solutions $\theta$ for those that have the second given coefficient $c_1$ in the power series representation
\[
\theta\rk{\ze}
=c_0+c_1\ze+\dotsb
\]
Each such function $\theta$ admits a representation of form \eqref{E3.46}.
Let us describe the set of parameters $w$ which is obtained in this way.
It is convenient to set $\theta_1\defeq w$ and let the power series of $\theta_1$ be given by
\[
\theta_1\rk{\ze}
=c_0^{\rk{1}}+\dotsb.
\]
Then if we rewrite \eqref{E3.46} in the form
\beql{E3.47}
\theta
=\tilde b_0\gk{\theta_1},
\eeq
it follows that
\[
\rkb{\Iu{\cG}+c_1c_0^\ad\rk{\Iu{\cG}-c_0c_0^\ad}^\inv}c_0^{\rk{1}}
=\rkb{\Iu{\cG}+c_1c_0^\ad\rk{\Iu{\cG}-c_0c_0^\ad}^\inv}c_0+c_1.
\]
We find from \eqref{E3.16} that the operator
\[
\Iu{\cG}+c_1c_0^\ad\rk{\Iu{\cG}-c_0c_0^\ad}^\inv
\]
has an inverse.
Therefore,
\beql{E3.48}
c_0^{\rk{1}}
:=c_0+\rkb{\Iu{\cG}+c_1c_0^\ad\rk{\Iu{\cG}-c_0c_0^\ad}^\inv}^\inv c_1.
\eeq
Hence, we obtain
\beql{E3.49}
\Iu{\cG}-c_0^{\rk{1}}\rk{c_0^{\rk{1}}}^\ad
>0.
\eeq
Thus, the solution of the Schur problem with two given coefficients in the power series representation
\[
\theta\rk{\ze}
=c_0+c_1\ze+\dotsb
\]
and under the assumption
\[
\Iu{\cG^2}-C_1C_1^\ad
>0
\]
can be described by \eqref{E3.47}, where $\theta_1\rk{\ze}=c_0^{\rk{1}}+\dotsb,$ and $c_0^{\rk{1}}$ is determined by \eqref{E3.48} and satisfies condition \eqref{E3.49}.

If we describe, as above, all functions $\theta_1\in S[\cF,\cG]$ of the form $\theta_1\rk{\ze}=c_0^{\rk{1}}+\dotsb$ with condition \eqref{E3.49}, we obtain a representation of $\theta_1$ in the form of a linear fractional transformation of the form
\[
\theta_1
=\tilde b_1\gk{\theta_2},
\]
where $\tilde b_1$ has form \eqref{E1214} with $c_0$ replaced by $c_0^{\rk{1}}$.

The superposition
\begin{align*}
    \theta&=\tilde b_0\gkb{\tilde b_1\gk{\theta_2}},&\theta_2&\in S[\cF,\cG]
\end{align*}
with coefficient block operator $\tilde b_0\tilde b_1$ yields the general form of an operator function $\theta\in S[\cF,\cG]$ with two given coefficients in the power series representation
\[
\theta\rk{\ze}
=c_0+c_1\ze+\dotsb
\]

\paragraph{3.}
Evidently, we can continue this procedure.
After $n+1$ steps we obtain the following result:
The general form of an operator function $\theta\in S[\cF,\cG]$ with prescribed $n+1$ initial coefficients in the power series representation
\[
\theta\rk{\ze}
=c_0+c_1\ze+\dotsb+c_n\ze^n+\dotsb
\]
with condition \eqref{E3.16} is the superposition
\[
\theta
=\tilde b_0\gkb{\tilde b_1\gkb{\dotso \tilde b_n\gk{\theta_{n+1}}\dotso }}
\]
with an arbitrary $\theta_{n+1}\in S[\cF,\cG]$ as a parameter, where
\beql{E3.50}
\tilde b_k\rk{\ze}
:=\Iu{\cG}+\rk{1-\ze}\Mat{\Iu{\cG}\\\rk{c_0^{\rk{k}}}^\ad}\rkb{\Iu{\cG}-c_0^{\rk{k}}\rk{c_0^{\rk{k}}}^\ad}^\inv\mat{\Iu{\cG},c_0^{\rk{k}}}
\eeq
and $c_0^{\rk{0}}:=c_0$.
The coefficient block operator of the resulting linear fractional transformation is
\beql{E3.51}
\prodr_{k=0}^{n}\tilde b_k.
\eeq
Consequently, to the Schur problem with condition \eqref{E3.16} there corresponds the product \eqref{E3.51} of binomial factors \eqref{E3.50} of full rank.

The converse statement is also valid.
To a given product \eqref{E3.51} of binomial factors \eqref{E3.50} of full rank there corresponds a non-degenerate Schur problem.

Thus, a Schur problem with strictly positive information block \eqref{E3.16} is equivalent to finding finite product \eqref{E3.51} of binomial $\tilde j$\nobreakdash-elementary factors of full rank with poles at $\ze=\infty$.

We note that the operators $c_0^{\rk{0}},c_0^{\rk{1}},\dotsc,c_0^{\rk{n}}$ are called \emph{the Schur parameters} of the sequence $c_0,c_1,\dotsc,c_n$.

\paragraph{4.} 
We turn our attention to the connection between product \eqref{E3.51} and the multiple factors $\tsB_n$ corresponding to the Schur problem.
Product \eqref{E3.51} is a $\tilde j$\nobreakdash-elementary multiple factor of full rank.
In view of parametrization \rthm{T3.10}, such an operator function can be represented in form \eqref{E1214}, namely:
\begin{multline*}
    \prodr_{k=0}^{n}\tilde b_k\rk{\ze}
    =\Mat{\IH&0\\0&\IF}\\
    +\rk{1-\ze}\Mat{\La_{\cG,n}\rk{\ze}&0\\0&\La_{\cF,n}\rk{\ze}}\tilde H_n\Mat{\La_{\cG,n}^\ad\rk{\ze}&0\\0&\La_{\cF,n}^\ad\rk{\ze}}\tilde j,
\end{multline*}
where the operator $\tilde H_n$ is determined by the operators $c_0,c_1,\dotsc,c_n$ which coincide with the data of the problem.
Therefore,
\[
\prodr_{k=0}^{n}\tilde b_k\rk{\ze}
=\tsB_n\rk{\ze}.
\]
In particular, this connection enables us, by means of formula~\eqref{E3.50}, to express the binomial factors $\tilde b_k$, $k=0,1,\dotsc,n$, constructed in terms of the Schur parameters directly in terms of the data of the problem.

\paragraph{5.}
We have been starting from FMI~\eqref{E3.13}.
If we start from FMI~\eqref{E3.12}, we obtain a similar result in terms of $j$\nobreakdash-elementary factors of the form
\beql{E3.52}
b_k\rk{\ze}
=\IH+\frac{1-\ze}{\ze}j\Mat{\rk{c_0^{\rk{k}}}^\ad\\\IH}\rkb{\Iu{\cG}-c_0^{\rk{k}}\rk{c_0^{\rk{k}}}^\ad}\mat{c_0^{\rk{k}},\IH},
\eeq
$k=0,1,\dotsc,n$, where $c_0^{\rk{0}},c_0^{\rk{1}},\dotsc,c_0^{\rk{n}}$ are the same Schur parameters, as above.
In particular, the non-degenerate Schur problem (see condition \eqref{E3.16}
of non-degeneracy) is equivalent to finding a finite product
\[
\prodl_{k=0}^{n}b_k.
\]
Moreover,
\beql{E3.53}
\sB_n
=\prodl_{k=0}^{n}b_k.
\eeq

We note that an essential feature of the approach created by V.~P.~Potapov is the possibility of solving the interpolation problem not only by a stepwise procedure but also in one step with the aid of a multiple $J$\nobreakdash-elementary factor (in Schur problem $J=j$ or $J=\tilde j$) dependent on the concrete form of the FMI.

The stepwise procedure in the degenerate Schur problem was studied in the paper \cite{MR889722}.
In view of the obtained results it is possible to establish a bijective correspondence between operator functions $\theta\in S[\cF,\cG]$, where
\[
\theta\rk{\ze}
=c_0+c_1\ze+\dotsb+c_n\ze^n+\dotsb
\]
for which all the information blocks
\begin{align*}
    &\Iu{\cG^{n+1}}-C_nC_n^\ad,&n&=0,1,2,\dotsc,
\end{align*}
are strictly positive and infinite products of binomial factors \eqref{E3.50} (\tresp{}\ \eqref{E3.52}) of full rank
\begin{align*}
    \prodr_{k=0}^{\infty}&\tilde b_k&\text{(\tresp{}\ }\prodl_{k=0}^{\infty}&b_k\text{).}
\end{align*}
Conversely, to every infinite product of factors of the form \eqref{E3.50} (\tresp{}\ \eqref{E3.52}) there corresponds a unique operator function $\theta\in S[\cF,\cG]$.
In the Schur problem the Blaschke--Potapov products are always divergent.
A precise description of their divergence can be expressed in terms of the semi-radii of the operator Weyl balls which are discussed in the next section.

\subsection{Operator Weyl balls associated with the Schur interpolation problem}\label{subsec3.5-0927}
The theory of operator Weyl balls associated with the non-degenerate Schur problem was worked out in \cite[Part~{II}]{D82}.
In this Section, we follow this paper.
A detailed treatment is in \cite[\cch{5}]{DFK92}.
The theory of operator Weyl balls associated with the degenerate Schur problem was treated in \cite{MR1219936}.
As it was shown in Section 3.3, fundamental matrix inequalities \eqref{E3.12} and \eqref{E3.13} are equivalent to 
inequalities \eqref{E3.39} and \eqref{E3.40}, respectively.
Against to this background, we introduce the operator functions
\begin{align}
    W_n&\defeq\sB_n^\inv j\sB_n^{\ad\inv}, \ \ \ \ \ \zeta \in \D,   \label{E3.54}
    \intertext{and}
    \tilde W_n&\defeq\tsB_n^{\ad\inv}\tilde j\tsB_n^\inv, \ \ \ \ \ \zeta \in \DO. \label{E3.55}
\end{align}
Following \cite{MR703593}, we call $W_n$ and $\tilde W_n$ \emph{the Weyl operator functions}.
Similarly, as in \eqref{E3.33a} and \eqref{E3.35}, we choose for $W_n$ and $\tilde W_n$ the block forms
\begin{align}\label{E3.56}
    W_n&:=\Mat{-R_n&S_n^\ad\\S_n&-T_n},&
    \tilde W_n&:=\Mat{-\tilde R_n&\tilde S_n\\\tilde S_n^\ad&-\tilde T_n}.
\end{align}
with respect to the decompositions $\cF\oplus\cG$ and $\cG\oplus\cF$, respectively. 
From \eqref{E3.54} and \eqref{E3.37} it follows that
\begin{align}\label{E3.57}
    R_n\rk{\ze}&=\IF+\rkb{1-\abs{\ze}^2}\La_{\cF,n}\rk{\ze}C_n^\ad\rk{\Iu{\cG^{n+1}}-C_nC_n^\ad}^\inv C_n\La_{\cF,n}^\ad\rk{\ze},&\ze&\in\D .
\end{align}
Analogously, from \eqref{E3.55} and \eqref{E3.38} we get
\begin{equation}\label{E3.58}
    \tilde R_n\rk{\ze}
    =\IG+\frac{1-\abs{\ze}^2}{\abs{\ze}^2}\La_{\cG,n}\rkb{\frac{1}{\ko\ze}}\rk{\Iu{\cG^{n+1}}-C_nC_n^\ad}^\inv\La_{\cG,n}^\ad\rkb{\frac{1}{\ko\ze}},
    \quad\ze\in\DO .
\end{equation}
From \eqref{E3.57} and \eqref{E3.58} we see that the operator functions $R_n$ and $\tilde R_n$ are strictly positive since
\begin{align}\label{E3.59}
    R_n\rk{\ze}&\geq\IF,&
    \tilde R_n\rk{\ze}&>\IG
\end{align}
for $\ze\in\DO $.
It follows from \eqref{E3.37} and \eqref{E1214} that
\[
\tsB_n
=Q^\inv j\sB_n^\inv jQ,
\]
where
\[
Q
:=\Mat{0&\IF\\\IG&0}\colon\cG\oplus\cF\to\cF\oplus\cG.
\]
Thus, we obtain
\[
\tilde W_n^\inv
= - Q^\inv jW_n jQ.
\]
Passing to the block forms \eqref{E3.56}, we obtain the equality
\[
-\Mat{T_n&S_n\\S_n^\ad&R_n}\Mat{\tilde R_n&-\tilde S_n\\-\tilde S_n^\ad&\tilde T_n}
=\Mat{\IG&0\\0&\IF}.
\]
Thus, using \eqref{E3.59}, we obtain the equalities
\begin{gather}
    R_n^\inv=\tilde S_n^\ad\tilde  R_n^\inv\tilde S_n-\tilde T_n,\label{E3.60}\\
    \tilde R_n^\inv=S_nR_n^\inv S_n^\ad-T_n,\label{E3.61}\\
    \tilde R_n^\inv\tilde S_n=S_nR_n^\inv.\label{E3.62}
\end{gather}
In the present case all the conditions which are necessary for the construction of the theory of the operator Weyl balls \cite[\csec{2}]{MR703593} are satisfied.
In particular:
\begin{enumerate}
    \item Every operator function $W_n$ is representable in the form
    \[
    W_n
    =\sB_n^\inv j\sB_n^{\ad\inv}
    =\Mat{-R_n&S_n^\ad\\S_n&T_n},
    \]
    where $\sB_n$ is a non-singular operator function which is $j$\nobreakdash-expansive on the unit disk. 
    \item The $j$\nobreakdash-forms $j-\sB_n^\inv\rk{\ze}j\sB_n^{\ad\inv}\rk{\ze}$ are monotonically non-decreasing for $\ze\in\D $ as $n$ increases.
    (This follows from the equality $\sB_{n+1}\rk{\ze}=b_{n+1}\rk{\ze}\sB_{n}\rk{\ze}$ (see \eqref{E3.53}).)
    \item The block $R_n\rk{\ze}$ is non-singular on the unit disk.
    
    Therefore, as in \cite{MR703593}, we can prove the following result:
\end{enumerate}

\bthmnl{{\cite[Part~{II}]{D82}}}{T3.13}
Let $n\in\NO$ and let a sequence of operators $c_k\in\FG $, $k=0,\dotsc,n$, satisfy \eqref{E3.16}.
Further, let $\ze_0\in\D$.
Then the set
\[
K_n\rk{\ze_0}
\defeq\setacab{\theta\rk{\ze_0}}{\theta\in S[\cF, \cG; \seqaj{c}{k}{n}]}
\]
fills the operator Weyl ball

\begin{align}\label{E3.63}
    \{M_n\rk{\ze_0}&+\rho_{\ell,n}^\varsqrt\rk{\ze_0}u\rho_{r,n}^\varsqrt\rk{\ze_0}:\ \ u\in\FG, \ \ u^\ad u\leq\IF \}, 
\end{align} 
with the center $M_n\rk{\ze_0}:=S_n\rk{\ze_0}R_n^\inv\rk{\ze_0}$, the left semi-radius
\[
\rho_{\ell,n}\rk{\ze_0}
:=S_n\rk{\ze_0}R_n^\inv\rk{\ze_0}S_n^\ad\rk{\ze_0}-T_n\rk{\ze_0},
\]
and the right semi-radius $\rho_{r,n}\rk{\ze_0}:=R_n^\inv\rk{\ze_0}$.

With the increasing parameter $n$ the balls $K_n\rk{\ze_0}$ are nested whereas the left and right semi-radii $\rho_{\ell,n}\rk{\ze_0}$ and $\rho_{r,n}\rk{\ze_0}$ are monotonically decreasing.
\ethm

Here the operator Weyl balls are written in terms of the blocks of $W_n\rk{\ze_0}$.
Taking into account \eqref{E3.60}--\eqref{E3.62}, an analogous description can be given in terms of blocks of $\tilde W_n\rk{\ze_0}$.
According to \rthm{T3.13}, to each operator functions $\theta\in S[\cF,\cG]$ with
\[
\te\rk{\ze}
=c_0+c_1\ze+\dotsb+c_n\ze^n+\dotsb, \ \ \ze\in\D,
\]
for which condition \eqref{E3.16} is satisfied for $n=0,1,2,\dotsc$ and to each $ \ze\in\D$
there corresponds the unique sequence of nested operator Weyl balls
\[
K_0\rk{\ze}
\supseteq K_1\rk{\ze}
\supseteq\dotsb
\supseteq K_n\rk{\ze}
\supseteq\dotsb.
\]
Letting $n\to\infty$, we obtain \emph{the limit Weyl ball}
\begin{align*}
    K_{\infty}\rk{\ze}&:= \{M_\infty\rk{\ze}+\rho_{\ell,\infty}^\varsqrt\rk{\ze}u\rho_{r,\infty}^\varsqrt\rk{\ze}: \
    u\in\FG,\
    u^\ad u\leq\IF\},
\end{align*}
where
\begin{align*}
    M_\infty\rk{\ze}&:=\lim_{n\to\infty}M_n\rk{\ze},&
    \rho_{\ell,\infty}\rk{\ze}&:=\lim_{n\to\infty}\rho_{\ell,n}\rk{\ze},&
    \rho_{r,\infty}\rk{\ze}&:=\lim_{n\to\infty}\rho_{r,n}\rk{\ze}.
\end{align*}
Note that in the present case
\begin{align}\label{E3.65}
    \rho_{\ell,\infty}\rk{\ze}&=0,&
    \ze&\in\D.
\end{align}
Indeed, from \rthm{T3.13} and equality \eqref{E3.58} and \eqref{E3.61} we obtain
\[\begin{split}
    \rho_{\ell,n}^\inv\rk{\ze}
    &=\tilde R_n\rk{\ze}\\
    &\geq\IG+\frac{1-\abs{\ze}^2}{\abs{\ze}^2}\La_{\cG,n}\rkb{\frac{1}{\ko\ze}}\La_{\cG,n}^\ad\rkb{\frac{1}{\ko\ze}}
    =\frac{1}{\abs{\ze}^{2n+2}}\IG,
    \qquad\ze\in\DO.
\end{split}\]
From this \eqref{E3.65} follows.
This agrees completely with the fact that the ``infinite'' Schur problem has a unique solution.
Taking additionally into account Shmul\cprime yan \cite{MR0273377}, we see that, for each $\ze\in\D$, the limit Weyl ball $ K_{\infty}\rk{\ze}$ shrinks to the single point  $\theta(\zeta)$ since its left semi-radius is the null operator.
However, the specific nature of the operator Schur problem is distinctive in the fact that, although the left limit semi-radius $\rho_{\ell,\infty}=0$, the limit Weyl ball $K_\infty\rk{\ze}$ is characterized by its right semi-radius $\rho_{r,\infty}\rk{\ze}$ which, as it will be shown below, is different from zero.
The subsequent discussion is directed on verifying a connection between the semi-radii of the operator Weyl balls associated with the functions $\te\in\SFFG$ and $\theta^\sim\in S[\cG,\cF]$.
Hence, if necessary, the semi-radii of the operator Weyl ball $ K_n\rk{\ze}, \ze\in\D,$ associated with $\te$ will be denoted by $\rho_{\ell,n}\rk{\ze,\te}$ and $\rho_{r,n}\rk{\ze,\te}$.

\bdefnnl{{\cite[Part~{II}]{D82}}}{D3.14}
The operator function $\hat\rho_{\ell,n}$ defined by
\begin{align*}
    \hat\rho_{\ell,n}\rk{\ze,\te}&\defeq\frac{1}{\abs{\ze}^{2n+2}}\rho_{\ell,n}\rk{\ze,\te},&
    \ze&\in\DO,
\end{align*}
will be called \emph{the normalized left semi-radius}.
\edefn

The importance of this function is based on the following result.

\bthmnl{{\cite[Part~{II}]{D82}}}{T3.15}
For $n=0,1,2,\dotsc$ 
\begin{align}\label{E3.67}
    \rho_{r,n}\rk{\ze, \theta^\sim}&=\hat\rho_{\ell,n}\rk{\ko \ze,\te},&
    \ze&\in\DO,
\end{align}
holds true.
\ethm

Equalities \eqref{E3.67} gives us the possibility to define $\hat\rho_{\ell,n}\rk{\ze,\te}$ at $\ze=0$ by setting $\hat\rho_{\ell,n}\rk{0,\te}\defeq\rho_{r,n}\rk{0,\theta^\sim}$.
Moreover, it follows from \eqref{E3.67} that there exists the limit
\begin{align*}
    \hat\rho_{\ell,\infty}\rk{\ze,\te}&\defeq\lim_{n\to\infty}\hat\rho_{\ell,n}\rk{\ze,\te},&
    \ze&\in\D.
\end{align*}
and we have the equality
\begin{align}\label{E3.68}
    \rho_{r,\infty}\rk{\ze, \theta}&=\hat\rho_{\ell,\infty}\rk{\ko \ze, \theta^\sim},&
    \ze&\in\D.
\end{align}
Following \cite[Part~{II}]{D82}, the operator function $\hat\rho_{\ell,\infty}$ will be called \emph{the normalized left semi-radius} of the limit Weyl ball.

\bthmnl{{\cite[Part~{II}]{D82}}}{T3.16}
For $n=0,1,2,\dotsc$ 
\begin{align*}
    \det\rho_{r,n}\rk{\ze,\te}&=\det\hat\rho_{\ell,n}\rk{\ze,\te},&
    \ze&\in\D,
\end{align*}
holds true.
\ethm

Here and in the sequel, if $A$ is a linear operator in a finite-dimensional linear space, then the notation $\det A$ stands for the product of its eigenvalues with taking into account their multiplicities.

\bcorl{C3.17}
The identities
\begin{align}\label{E3.69}
    \det\rho_{r,\infty}\rk{\ze,\te}&=\det\hat\rho_{\ell,\infty}\rk{\ze,\theta}
    =\det\rho_{r,\infty}\rk{\ko\ze, \theta^\sim},&
    \ze\in\D,
\end{align}
hold true.
\ecor

It follows from the fundamental Orlov theorem \cite{MR0425671} (see also \cite{MR1468778}) that, in particular, $\rank\rho_{r,\infty}\rk{\ze,\te}$ and therefore also $\rank\rho_{\ell,\infty}\rk{\ze,\te^\sim}$ is independent of the choice of the point $\ze\in\D$.
Therefore, $\rank\rho_{r,\infty}\rk{\ze,\te}$ and $\rank\hat\rho_{\ell,\infty}\rk{\ze,\te}$ are constants on $\D$ and in computing them the point $\ze\in\D$ can be chosen arbitrarily.
Consequently, there arises a possibility of classifying ``infinite'' Schur problems and at the same time of classifying functions $\te\in\SFFG$ by means of the values of the ranks of the limit semi-radii $\hat\rho_{\ell,\infty}\rk{\ze,\te}$ and $\rho_{r,\infty}\rk{\ze,\te}$.
In this connection, we mention the following result:

\bthmnl{{\cite[Part~{II}]{D82}}}{T3.18}
Let $\te\in\SFFG$.
Then $\rank\rho_{r,\infty}\rk{\ze,\te}=\dim\cF$ if and only if $\rank\hat\rho_{\ell,\infty}\rk{\ze,\te}=\dim\cG$.
\ethm

In the paper \cite{MR647177} an important algorithm was worked out which enables to choose the Taylor coefficient sequence $\seqaj{c}{k}{\infty}$ of the Schur function $\te\in\SFFG$ in order to construct semi-radii with a priori given fixed ranks.
The application of these methods to the Schur problem leads to the result which complements \rthm{T3.18}.

\bthmnl{{\cite[Part~II]{D82}, \cite{MR1219936}}}{T3.19}
Let $\dim\cF=q$, $\dim\cG=p$, and $\min\set{p,q}\geq1$.
Then for all integers $\alpha,\beta$ satisfying $0\leq\alpha\leq p-1$, $0\leq\beta\leq q-1$ there exists a function $\te\in\SFFG$ for which $\rank\hat\rho_{\ell,\infty}\rk{\ze,\te}=\alpha$ and $\rank\rho_{r,\infty}\rk{\ze,\te}=\beta$ are satisfied.
\ethm

In the paper \cite[Part~II]{D82}, this result was obtained for the case of the non-degenerate ``infinite'' Schur problem (condition \eqref{E3.16} is satisfied for all $n=0,1,2,\dotsc$) whereas this restriction was removed in \cite{MR1219936}.
Furthermore, it was shown in \cite{MR1219936} that in the case of the degenerate ``infinite'' Schur problem the semi-radii $\rank\hat\rho_{\ell,\infty}\rk{\ze,\te}$ and $\rank\rho_{r,\infty}\rk{\ze,\te}$ degenerate.
This means the full ranks of the semi-radii $\hat\rho_{\ell,\infty}\rk{\ze,\te}$ and 
$\rho_{r,\infty}\rk{\ze,\te}$ are only possible in the non-degenerate situation.

\breml{R3.20}
It follows from our discussion that the operator Weyl ball \eqref{E3.63} is conveniently written in the form
\begin{align*}
    \{M_n\rk{\ze}&+\abs{\ze}^{n+1}\hat\rho_{\ell,n}^\varsqrt\rk{\ze}u\rho_{r,n}^\varsqrt\rk{\ze}: \
    u\in\FG, \
    u^\ad u\leq\IF\}.
\end{align*}
\erem 

Here we introduced the factor $\abs{\ze}^{n+1}$ because of which the ball $K_\infty\rk{\ze}$ shrinks to a point.

\bdefnnl{{\cite[Part~II]{D82}}}{D3.21}
We call the numbers
\begin{align}
    \de_\te&\defeq\rank\rho_{r,\infty}\rk{\ze,\te},\label{E3.71}\\
    \de_{\te^\sim}&\defeq\rank\rho_{r,\infty}\rk{\ze,\theta^\sim}\label{E3.72}
\end{align}
\emph{the defect numbers} of the operator function $\te\in\SFFG$.
\edefn

\subsection{Factorization of the semi-radii of the limit Weyl ball, connection with maximal shift and co-shift contained in a contractive operator}\label{s3.6}
Let $\te\in\SFFG$ and
\beql{E3.73}
\Dl
=\rk{\cH,\cF,\cG;T,F,G,S}
\eeq
be a simple unitary colligation such that $\te_\Dl=\te$.
The following statement establishes the connection between the semi-radii of the limit Weyl ball in the Schur problem associated with the function $\te$ and the largest shift and coshift contained in the contraction $T$.   

\bthmnl{{\cite[\cparts{V,VI}]{D82}}, \cite{MR1483807}}{T3.22}
Let $\cL_0$ and $\tcL_0$ be generating wandering subspaces for the largest shifts $V_T$ and $V_{T^\ad}$, respectively.
Then
\begin{align}
    \rho_{r,\infty}\rk{\ze,\te}&=\ter^\ad\rk{\ze}\ter\rk{\ze},& \ze \in \D,\label{E3.74}\\
    \hat\rho_{\ell,\infty}\rk{\ze,\te}&=\tel\rk{\ze}\tel^\ad\rk{\ze},& \ze \in \D,\label{E3.75}
\end{align}
where $\rho_{r,\infty}\rk{\ze,\te}$ and $\hat\rho_{\ell,\infty}\rk{\ze,\te}$ are the right and normalized left semi-radii of the limit Weyl ball, respectively,
\begin{align}
    \ter\rk{\ze}&\defeq P_{\cL_0}\rk{I_{\cH}-\ze T}^\inv F,&\ze&\in\D,\label{E3.76}\\
    \tel\rk{\ze}&\defeq \rstr{G\rk{I_{\cH}-\ze T}^\inv}{\tcL_0},&\ze&\in\D,\label{E3.77} 
\end{align}
are $[\cF, \cL_0]$- and $[\tcL_0, \cG]$-valued Schur functions, respectively,
and $P_{\cL_0}$ 
is the orthoprojection from $\cH$ onto $\cL_0$.
Here $\cL_0$ and $\tcL_0$ are the generating wandering subspaces for the largest shifts $V_T$ and $V_{T^\ad}$ of the contraction $T$, respectively.
\ethm

We note that in \cite[\cparts{V,VI}]{D82} factorizations \eqref{E3.74} and \eqref{E3.75} were obtained for the non-degenerate ``infinite'' Schur problem associated with the function $\te$.
In the degenerate case these factorizations were obtained in \cite{MR1483807}.

\bdefnnl{{\cite[\cpart{VI}]{D82}, \cite{BD4}}}{D3.23}
Let $\te\in\SFFG$ and let $\Dl$ be the simple unitary colligation of form \eqref{E3.73}.
The operator functions $\ter$ and $\tel$ (see \eqref{E3.76} and \eqref{E3.77}) are called \emph{the right and left defect functions of the function $\te$ in the Schur class}, respectively.
\edefn

From \eqref{E3.74} and \eqref{E3.75} we obtain
\begin{align*}
    \rho_{r,\infty}\rk{0,\te}&=F^\ad P_{\cL_0}F,&
    \hat\rho_{\ell,\infty}\rk{0,\te}&=GP_{\tcL_0}G^\ad.
\end{align*}
From this and \rcor{cor1.9} we obtain the following statement.

\bthmnl{{\cite[\cpart{III}]{D82}}}{T3.24}
The identities
\begin{align}\label{E3.78}
    \rank\rho_{r,\infty}\rk{\ze,\te}&=\dim\cL_0,&
    \rank\hat\rho_{\ell,\infty}\rk{\ze,\te}&=\dim\tcL_0, &\ze\in\D, 
\end{align}
hold true.
\ethm

From \eqref{E3.71}, \eqref{E3.72} and \eqref{E3.78} we conclude

\bcorl{C3.25}
$\dl_\te=\dim\cL_0$ and $\dl_{\te^\sim}=\dim\tcL_0$.
\ecor

\bthmnl{{\cite[\cpart{III}]{D82}}}{T3.26}
Let $T$ be a completely non-unitary contraction with finite defect numbers $\dl_{T^\ad}=q$ and $\dl_T=p$ which acts in the Hilbert space $\cH$.
Let $\cL_0$ and $\tcL_0$ be the generating wandering subspaces for the largest shifts $V_T$ and $V_{T^\ad}$, respectively.
Then
\begin{align}\label{E3.79}
    0&\leq\dim\cL_0\leq q,&
    0&\leq\dim\tcL_0\leq p.
\end{align}
Moreover, the following two cases take place:
\begin{enumerate}
    \item[\textnormal{(1)}] $\dim\cL_0=p$ if and only if $\dim\tcL_0=q$;
    \item[\textnormal{(2)}] if $\min\rk{p,q}\geq1$, then all remaining cases in the inequalities
    \begin{align*}  
        0\leq\dim\cL_0\leq q-1, \ \ \ \ 0\leq\dim\tcL_0\leq p-1
    \end{align*}  
    are possible. 
\end{enumerate}

\ethm
\bproof
We embed the contraction $T$ in the unitary colligation of form \eqref{E3.73} such that $\dim\cF=q$ and $\dim\cG=p$.
For example, the choices \eqref{1.100} and \eqref{1.101} can be taken.
Now inequalities \eqref{E3.79} follow from \rcor{cor1.9} whereas the remaining assertions of the theorem are consequences of \rthmss{T3.24}{T3.19}.
\eproof

\breml{R3.27}
Following the paper \cite{MR1483807}, we present an operator interpretation of the fact that the function $\rank\rho_{r,\infty}\rk{\ze,\te}$ (\tresp{}\ $\rank\hat\rho_{\ell,\infty}\rk{\ze,\te}$, see \eqref{E3.68}) is constant on $\D$.
Indeed, if
\begin{align*}
    \te\rk{\ze}&=S+\ze G\rk{\IH-\ze T}^\inv F,&\ze&\in\D,
\end{align*}
is the characteristic function of the simple unitary colligation \eqref{E3.73}, then it is immediately checked that the analogous representation for the function
\begin{align*}
    \te_a\rk{\ze}&=\te\rkb{\frac{\ze+a}{1+\ko{a}\ze}}\in\SFFG,&a&\in\D,
\end{align*}
has the form
\begin{align*}
    \te_a\rk{\ze}&=S_a+\ze G_a\rk{\IH-\ze T_a}^\inv F_a,&\ze&\in\D,
\end{align*}
where the operators
\begin{align*}
    T_a&=\rk{T-\ko{a}\IH}\rk{\IH-aT}^\inv,&
    F_a&=\sqrt{1-\abs{a}^2}\rk{\IH-aT}^\inv F,\\
    G_a&=\sqrt{1-\abs{a}^2}G\rk{\IH-aT}^\inv,&
    S_a&=S-a\rk{\IH-aT}^\inv F
\end{align*}
generate a unitary colligation.
As it is shown in \cite[\cch{1}]{SN}, the linear fractional transformation $T\mapsto T_a$ maps isometric operators to isometric ones and unitary operators to unitary ones.
Taking into account that
\[
\bigvee_{k=0}^\infty T_a^{\ad k}G_a^\ad\rk{\cG}
=\bigvee_{k=0}^\infty T^{\ad k}G^\ad\rk{\cG}
=\cH_\cG,
\]
one can conclude that the largest shift $V_T$ contained in $T$ passes over to the largest shift $V_{T_a}$ contained in $T_a$.
These shifts act in the same subspace $\cH_\cG^\bot=\cH\ominus\cH_\cG$ and their multiplicities coincide.
We note that
\begin{align*}
    T^k\cL_0&\perp\ran F,&k&=1,2,3,\dotsc
\end{align*}
Indeed, we have
\[
FF^\ad T^k\cL_0
=\rk{\IH-TT^\ad}T^k\cL_0
=T^\ad\rk{\IH-T^\ad T}T^{k-1}\cL_0
=T^\ad G^\ad GT^{k-1}\cL_0
=0.
\]
Hence, taking into account the factorization \eqref{E3.74} and the decomposition \eqref{E2.1002}, we obtain
\[
\rho_{r,\infty}\rk{0,\te}
=F^\ad P_{\cL_0} F
=F^\ad\sum_{k=0}^\infty T^kP_{\cL_0}T^{\ad k}F
=F^\ad P_{\cH_\cG^\bot}F.
\]
Analogously, we get
\[\begin{split}
    \rho_{r,\infty}\rk{0,\te_a}
    &=F_a^\ad P_{\cH_\cG^\bot}F_a
    =\rkb{1-\abs{a}^2}F^\ad\rk{\IH-\ko{a}T^\ad}^\inv P_{\cH_\cG^\bot}\rk{\IH-aT}^\inv F\\
    &=\rkb{1-\abs{a}^2}F^\ad\sum_{k=0}^\infty\ko{a}^kT^{\ad k}\sum_{n=0}^\infty T^nP_{\cL_0}T^{\ad n}\sum_{j=0}^\infty a^jT^jF\\
    &=\rkb{1-\abs{a}^2}\sum_{n=0}^\infty\sum_{k=0}^\infty\sum_{j=0}^\infty\ko{a}^ka^jF^\ad T^{\ad k} T^nP_{\cL_0}T^{\ad n} T^jF\\
    &=\rkb{1-\abs{a}^2}\sum_{n=0}^\infty\sum_{k=n}^\infty\sum_{j=n}^\infty\ko{a}^ka^jF^\ad T^\ad T^nP_{\cL_0}T^{\ad n} T^jF\\
    &=\rkb{1-\abs{a}^2}\sum_{n=0}^\infty\sum_{k=n}^\infty\sum_{j=n}^\infty\ko{a}^ka^jF^\ad\rk{T^\ad}^{k-n}P_{\cL_0}T^{j-n}F\\
    &=\rkb{1-\abs{a}^2}\sum_{n=0}^\infty\sum_{k=0}^\infty\sum_{j=0}^\infty\ko{a}^{k+n}a^{j+n}F^\ad T^{\ad k}P_{\cL_0}T^jF\\
    &=F^\ad\rk{\IH-\ko{a}T^\ad}^\inv P_{\cL_0}\rk{\IH-aT}^\inv F
    =\rho_{r,\infty}\rk{a,\te}.
\end{split}\]
Now, using \rthm{T3.24}, we obtain that the constancy of the function $\rank\rho_{r,\infty}\rk{\ze,\te}$ on $\D$ follows from the constancy of the multiplicity of the largest shift $V_{T_a}, a\in \D$, as noted above. 

In addition, on the way we have proved the identity
\begin{align}\label{E3.80}
    \ter^\ad\rk{a}\ter\rk{a}&=\rkb{1-\abs{a}^2}F^\ad\rk{\IH-\ko{a}T^\ad}^\inv P_{\cH_\cG^\bot}\rk{\IH-aT}^\inv F,&
    a&\in\D.
\end{align}
\erem

\section{Defect functions, unitary couplings and the Darlington synthesis}\label{sec4-1003}
\subsection{Defect functions and their properties}\label{subsec4.1}

Let $\te\in\SFFG$ and let
\beql{E4.1}
\Dl
\defeq\rk{\cH,\cF,\cG;T,F,G,S}
\eeq
be a simple unitary colligation such that $\te_\Dl=\te$.
Let $\cL_0$ and $\tcL_0$ be the wandering generating subspaces for the largest shifts $V_T$ and $V_{T^*}$, respectively.
We consider the \tval{$\ek{\cF,\cL_0}$} function $\ter$ defined on $\D$ by
\begin{align}\label{E4.2}
    \ter\rk{\ze}
    &\defeq P_{\cL_0}\rk{\IH-\ze T}^\inv F,&
    \ze&\in\D,
\end{align}
(see \eqref{E3.76}) and the \tval{$\ek{\tcL_0,\cG}$} function $\tel$ defined on $\D$ by
\begin{align}\label{E4.3}
    \tel\rk{\ze}
    &\defeq\rstr{G\rk{\IH-\ze T}^\inv}{\tcL_0},&
    \ze&\in\D,
\end{align}
(see \eqref{E3.77}).
Since
\begin{align*}
    \te^\sch\rk{\ze}
    &\defeq\te^\ad\rk{\ko\ze}
    =S^\ad+\ze F^\ad\rk{\IH-\ze T^\ad}^\inv G^\ad,&
    \ze&\in\D,
\end{align*}
then
\begin{align}\label{E4.4}
    \rk{\te^\sch}_r\rk{\ze}
    &=P_{\tcL_0}\rk{\IH-\ze T^\ad}^\inv G^\ad
    =\te^\ad_\ell\rk{\ko\ze}
    =\rk{\tel}^\sch\rk{\ze},&
    \ze&\in\D.
\end{align}

\bthmnl{\zitaa{D82}{\cpart{VI}}}{T4.1}
Let $\te\in\SFFG$ and let $\Dl$ be a simple unitary colligation of form \eqref{E4.1} such that $\te_\Dl=\te$.
Then the identities
\begin{multline}
    \IF-\te^\ad\rk{\ze}\te\rk{\ze}-\ter^\ad\rk{\ze}\ter\rk{\ze}\\
    =\rkb{1-\abs{\ze}^2}F^\ad\rk{\IH-\ko\ze T^\ad}^\inv P_{\cH_\cG}\rk{\IH-\ze T}^\inv F,\qquad\ze\in\D,\label{E4.5}
\end{multline}
and
\begin{multline}
    \IG-\te\rk{\ze}\te^\ad\rk{\ze}-\tel\rk{\ze}\tel^\ad\rk{\ze}\\
    =\rkb{1-\abs{\ze}^2}G\rk{\IH-\ze T}^\inv P_{\cH_\cF}\rk{\IH-\ko\ze T^\ad}^\inv G^\ad,\qquad\ze\in\D,\label{E4.6} 
\end{multline}
are valid.
\ethm
\bproof
Identity \eqref{E4.5} follows from \eqref{E2.17} and \eqref{E3.80}
since for the case $\dim\cF=\infty$ it can be proved in a similar way. Identity \eqref{E4.6} follows from \eqref{E4.5} for the function $\te^\sch\in\SFGF$ and from \eqref{E4.4}
\eproof

\bcornl{\zitaa{D82}{\cpart{VI}}, \cite{BD4}}{C4.2}
The functions $\ter$ and $\tel$ given by \eqref{E4.2} and \eqref{E4.3} are Schur functions, namely,
\begin{align*}
    \ter&\in\SF{\cF}{\cL_0},&
    \tel&\in\SF{\tcL_0}{\cG}.
\end{align*}
\ecor

The concept of the defect functions in the Schur class which was considered for the matrix case in \rsubsec{s3.6} (see \rdefn{D3.23}) can be extended to the general case.

\bdefnnl{\zitaa{D82}{\cpart{VI}}, \cite{BD4}}{D4.3}
Let $\te\in\SFFG$.
The operator functions $\ter$ and $\tel$ given by \eqref{E4.2} and \eqref{E4.3} are called \emph{the right} and \emph{left defect functions of $\te$ in the Schur class}, respectively.
\edefn

Note that in \cite{BD4} (see also \cite{DM}), outgoing from $\Dl$, unitary colligations $\Dl_r$ and $\Dl_\ell$ are constructed such that $\te_{\Dl_r}=\te_r$ and $\te_{\Dl_\ell}=\tel$ are satisfied.
Starting with the functions $\ter$ and $\tel$, we can continue this procedure and obtain the functions
\begin{align*}
    \terr&\defeq\rk{\ter}_r,&
    \terl&\defeq\rk{\ter}_\ell,&
    \telr&\defeq\rk{\tel}_r,&
    \tell&\defeq\rk{\tel}_\ell.
\end{align*}
After that the process can be continued further in this way.
In \cite{BDFK1} it was shown that $\terr$ (resp.\ $\tell$) is a right (resp.\ left) regular divisor of $\te$.
Furthermore, the following statement is true.

\bthmnl{\zitaa{BDFK1}{\cthm{6}}}{T4.4}
Let $\te\in\SFFG$ and let $\Dl$ be a simple unitary colligation of form \eqref{E4.1} such that $\te_\Dl=\te$.
Then to the right regular divisor $\terr$ of $\te$ there corresponds the smallest subspace of $\cH$ which is invariant with respect to $T^\ad$ and contains the subspace $\cH_\cG^\bot$ \textnormal{(}see \eqref{E2.1002}\textnormal{)}.
Analogously, to the left regular divisor $\tell$ of $\te$ there corresponds the smallest subspace of $\cH$ which is invariant with respect to $T$ and contains the subspace $\cH_\cF^\bot$ \textnormal{(}see \eqref{E2.1002}\textnormal{)}, where $Y$ is given by \eqref{1.2}.
\ethm

\bthmnl{\zitaa{BDFK1}{\crem{1}}}{T4.5}
A function $\mu\in\SF{\cM}{\cN}$ is the right (\tresp{}\ left) defect function of some Schur function if and only if
\beql{E4.7}
\mu=\mu_{rr}\text{ (\tresp{}\ $\mu=\mu_{\ell\ell}$).}
\eeq
\ethm

Important properties of defect functions are contained in the following statement.

\bthmnl{\textnormal{see, \teg{}, \zitaa{SN}{\cch{V}}}}{T4.6}
Let $N\in  CM\ek{\cC}$ be a non-negative function.
Then there exists 
an outer function $\alpha_o\in\SF{\cC}{\cN_o}$ such that
\beql{E4.8}
\bv{\alpha_o}^\ad\rk{t}\bv{\alpha_o}\rk{t}\leq N^2\rk{t}\text{ \tae{}\ on }\T.
\eeq
Furthermore, for any $\alpha\in\SF{\cC}{\cN}$ such that
\[
\bv{\alpha}^\ad\rk{t}\bv{\alpha}\rk{t}\leq N^2\rk{t}\text{ \tae{}\ on }\T
\]
the inequality
\beql{E4.10}
\bv{\alpha}^\ad\rk{t}\bv{\alpha}\rk{t}\leq \bv{\alpha_o}^\ad\rk{t}\bv{\alpha_o}\rk{t}\text{ \tae{}\ on }\T
\eeq
holds true.
These conditions determine the function $\alpha_o$ up to a left constant unitary factor.

In \eqref{E4.8} equality holds true for almost all $t\in\T$ if and only if 
\[
\bigcap_{n=0}^\infty t^n \clo{NL^2_{\mathord{+}}\rk{\cC}}
=\set{0}
\]
is satisfied.
\ethm

\bdefnl{D1454}
Let $N\in CM\ek{\cC}$ be a non-negative function.
The function $\alpha_o\in\SF{\cC}{\cN_o}$ from \rthm{T4.6} is called \emph{the largest minorant for the function $N$} (\emph{in the Schur class}).
A function $\alpha_{\ast o}\in\SF{\cN_{\ast o}}{\cC}$ is said to be \emph{the largest $\ast$\nobreakdash-minorant for the function $N$} if the associated function $\rk{\alpha_{\ast o}}^\sch$ is the largest minorant for the function $N^\sch$ given by $N^\sch\rk{t}\defeq N\rk{\ko t}$, $t\in\T$, \tie,
\[
\bv{\alpha_{\ast o}}\rk{t}\bv{\alpha_{\ast o}}^\ad\rk{t}
\leq N^2\rk{t}\qquad\text{a.\,e.\ on }\T
\]
and for any $\alpha\in\SF{\cN_{\ast}}{\cC}$ such that
\[
\bv{\alpha}\rk{t}\bv{\alpha}^\ad\rk{t}
\leq N^2\rk{t}\qquad\text{a.\,e.\ on }\T
\]
the inequality
\[
\bv{\alpha}\rk{t}\bv{\alpha}^\ad\rk{t}
\leq\bv{\alpha_{\ast o}}\rk{t}\bv{\alpha_{\ast o}}^\ad\rk{t}\qquad\text{a.\,e.\ on }\T
\]
holds true.
\edefn

Note that (see \rthm{T4.6}) the largest minorant (resp.\ $*$\nobreakdash-minorant) is an outer (resp.\ $*$\nobreakdash-outer) function and is determined up to a left (resp.\ right) constant unitary factor.

Ideas for introducing the largest minorant go back to G.~Szeg\H{o} \cite{Sze21} and N.~Wiener \cite{W}.
The concept of largest minorants was essentially developed in the context of multivariate prediction theory by N.~Wiener and P.~R.~Masani (see, \teg{}, \cite{MR121952,MR140930,MR97856,MR97859}).
Further historical information on the introduction of this concept can be found, \teg{}, in \zitaa{SN}{Notes to \csec{V.4}}.

\bthmnl{\zitaa{BD4}{\cch{4}}, \zitaa{DM}{\cthm{2.7}}}{T4.8}
Let $\te\in\SF{\cF}{\cG}$.
\benui
\il{T4.8.a} The right defect function $\ter$ of $\te$ is the largest minorant for the function $\Pi\in CM\ek{\cF}$ given by
\begin{align}\label{E4.12}
    \Pi\rk{t}&\defeq\rk{\IF-\bv\te^\ad\rk{t}\bv\te\rk{t}}^{1/2},&t&\in\T.
\end{align}
\il{T4.8.b} The left defect function $\tel$ of $\te$ is the largest $\ast$\nobreakdash-minorant for the function $\Sigma\in CM\ek{\cG}$ given by
\begin{align}\label{E4.13}
    \Sigma\rk{t}&\defeq\rk{\IG-\bv\te\rk{t}\bv\te^\ad\rk{t}}^{1/2},&t&\in\T.
\end{align}
\eenui
Thus, the inequalities
\begin{align}\label{E4.14}
    \bv\ter^\ad\rk{t}\bv\ter\rk{t}&\leq\IF-\bv\te^\ad\rk{t}\bv\te\rk{t},&
    \bv\tel\rk{t}\bv\tel^\ad\rk{t}&\leq\IG-\bv\te\rk{t}\bv\te^\ad\rk{t}
\end{align}
hold true for almost all $t\in\T$.
\ethm

The functions $\Pi\in  CM\ek{\cF}$ and $\Sigma\in  CM\ek{\cG}$ are called \emph{the defect} and \emph{$\ast$\nobreakdash-defect functions of $\bv\te\in  CM\ek{\cF,\cG}$ in the class of contractive measurable functions}, respectively.

\bcorl{C4.9}
The right \textnormal{(}\tresp{}\ left\textnormal{)} defect function $\ter\in\SF{\cF}{\cL_0}$ \textnormal{(}\tresp{}\ $\tel\in\SF{\tcL_0}{\cG}$\textnormal{)} of a function $\te\in\SFFG$ is outer \textnormal{(}$\ast$\nobreakdash-outer\textnormal{)}.
\ecor

Note that not every outer (\tresp{}\ $\ast$\nobreakdash-outer) Schur function $\mu\in\SF{\cM}{\cN}$ is the right (\tresp{}\ left) defect function of some Schur function.
For example (see \cite{BDFK1}), we consider a scalar function $\mu$ which conformally maps $\D$ onto a domain $\Omega\subseteq\setaca{\ze\in\D}{\re\ze>0}$ such that $\bv\lambda\rk{\partial\Omega\cap\T}>0$.
Then $\re\mu\rk{\ze}>0$ for $\ze\in\D$ and, consequently, $\mu$ is outer (see, \teg{}, \zitaa{RR}{\cch{4}}).
The first inequality \eqref{E4.14} for $\mu$ has the form
\[
\bv{\mu_r}^\ad\rk{t}\bv{\mu_r}\rk{t}
\leq\Iu{\cM}-\bv{\mu}^\ad\rk{t}\bv{\mu}\rk{t}
\]
for almost all $t\in\T$.
Hence, $\mu_r\equiv0$ and the equality $ \mu=\mu_{rr}$ (see \eqref{E4.7}) is not valid.

\subsection{Unitary couplings}\label{subsec4.2-1003}

\bdefnnl{\zitaa{BD-1}{\cpart{I}}}{D4.9}
A 6-tuple
\beql{E4.15}
\sigma
\defeq\rk{\cH,\cF,\cG;U,V_\cF,V_\cG}
\eeq
is called \emph{a unitary coupling} or simply \emph{a coupling} if
\benui
\il{D4.9.a} $\cH,\cF,\cG$ are Hilbert spaces;
\il{D4.9.b} $U$ is a unitary operator on $\cH$;
\il{D4.9.c} $V_\cF\colon\cF\to\cH$ and $V_\cG\colon\cG\to\cH$ are isometric operators, \tie{}, $V_\cF^\ad V_\cF=\IF$ and $V_\cG^\ad V_\cG=\IG$;
\il{D4.9.d} the subspaces $\mathring{\cF}\defeq\ran\rk{V_\cF}$ and $\mathring{\cG}\defeq\ran\rk{V_\cG}$ are wandering with respect to $U$.
\eenui
The spaces $\cF$ and $\cG$ are said to be \emph{the principal input} and \emph{output channeled spaces of the unitary coupling $\sigma$}, respectively.
The operator $U$ is called \emph{connecting} and the operators $V_\cF$ and $V_\cG$ are said to be \emph{the principal embedding isometries of the unitary coupling $\sigma$}.
\edefn

Henceforth, we will denote by $\mathring{\cL}$ the range of any isometry $V_\cL\colon\cL\to\cH$ if it is wandering with respect to $U$.

It should be noted that \rdefn{D4.9} 
differs only in form from the definition of a unitary coupling for semi-unitary operators which is given in \cite{A-A,A-1}.
Indeed, any subspace $\cN$ wandering with respect to $U$ generates the subspaces
\begin{align*}
    \mm{\cN}&\defeq\bigoplus_{k=-\infty}^\infty U^k\cN,&
    \mathfrak{R}_\cN&\defeq\cH\ominus\mm{\cN},\\
    \mmp{\cN}&\defeq\bigoplus_{k=0}^\infty U^k\cN,&
    \mmm{\cN}&\defeq\bigoplus_{k=-\infty}^{-1} U^k\cN.
\end{align*}
Hence, in the notations of \cite{A-A} coupling \eqref{E4.15} is the coupling of the simple semi-unitary operators
\begin{align*}
    V_{\mathord{-}}&\defeq\rstr{U^\ad}{\mmm{\mathring{\cF}}},&
    V_{\mathord{+}}&\defeq\rstr{U}{\mmp{\mathring{\cG}}}.
\end{align*}
Together with any coupling $\sigma$ of form \eqref{E4.15} we will consider the couplings
\begin{align}
    \sigma^\ad&\defeq\rk{\cH,\cG,\cF;U,V_\cG,V_\cF},\label{E4.16}\\
    \sigma^\sch&\defeq\rk{\cH,\cG,\cF;U^\ad,V_\cG,V_\cF}\label{E4.17}
\end{align}
which are called \emph{adjoint} and \emph{associate} with respect to the coupling $\sigma$, respectively.

\bdefnnl{{\cite[Part~I]{BD-1}}}{D4.10}
Let $\sigma$ be a unitary coupling of form \eqref{E4.15}.
By \emph{the principal part of $\sigma$} we mean the coupling
\[
\sigma^{\rk{1}}
\defeq\rk{\cH^{\rk{1}},\cF,\cG;U^{\rk{1}},V_\cF,V_\cG},
\]
where $\cH^{\rk{1}}\defeq\mm{\mathring{\cF}}\vee\mm{\mathring{\cG}}$, \ $U^{\rk{1}}\defeq\rstr{U}{\cH^{\rk{1}}}$.
The coupling $\sigma$ is called \emph{minimal}, if $\sigma=\sigma^{\rk{1}}$ and \emph{abundant} otherwise.
\edefn

Hereinafter, we maintain the notation $V_\cF$ (resp.\ $V_\cG$) when $\cH$ is replaced by any subspace which contains $\mm{\mathring{\cF}}$ (resp.\ $\mm{\mathring{\cG}}$) or by any space that contains $\cH$.
In the following we will need a simple geometric assertion.

\blemnl{{\cite[Part~I]{BD-1}}}{L4.11}
Let $\cL$ and $\cM$ be subspaces of  a Hilbert space $\cH$.
Then
\[
\cL
=\rk{\cL\cap\cM}\oplus\clo{P_\cL\cM^\oc}.
\]
\elem

\bcornl{{\cite[Part~I]{BD-1}}}{C4.12}
The equalities
\begin{align*}
    \cL^\oc&=\clo{P_{\cL^\oc}\cM}&
    &\text{and}&
    \cM^\oc&=\clo{P_{\cM^\oc}\cL}
\end{align*}
are equivalent.
They are valid if and only if
\[
\cH
=\cL\vee\cM.
\]
\ecor

From \rcor{C4.12} we conclude the next result.

\blemnl{{\cite[Part~I]{BD-1}}}{L4.13}
A unitary coupling $\sigma$ of form \eqref{E4.15} is minimal if and only if any of the two following equivalent conditions
\begin{align*}
    \mathfrak{R}_{\mathring{\cF}}&=\clo{P_{\mathfrak{R}_{\mathring{\cF}}}\mm{\mathring{\cG}}}&
    &\text{or}&
    \mathfrak{R}_{\mathring{\cG}}&=\clo{P_{\mathfrak{R}_{\mathring{\cG}}}\mm{\mathring{\cF}}}
\end{align*}
is satisfied.
\elem

\bdefnnl{{\cite[Part~I]{BD-1}}}{D4.14}
Unitary couplings
\begin{align}\label{E4.18}
    \sigma&\defeq\rk{\cH,\cF,\cG;U,V_\cF,V_\cG},&
    \sigma'&\defeq\rk{\cH',\cF,\cG;U',V_\cF',V_\cG'}
\end{align}
are called \emph{unitarily equivalent} if there exists a unitary operator $Z\colon\cH\to\cH'$ such that
\begin{align}\label{E4.19}
    U'Z&=ZU,&
    V_\cF'&=ZV_\cF,&
    V_\cG'&=ZV_\cG.
\end{align}
In this case, we say that $Z$ \emph{establishes the unitary equivalence of $\sigma$ and $\sigma'$}.
\edefn

Note that if the couplings $\sg$ and $\sg'$ are unitary equivalent and minimal, then the operator $Z$ is uniquely determined by conditions \eqref{E4.19} (see, \teg{}, \zitaa{BD-1}{\cpart{I}}).

\subsection{The scattering suboperator of a unitary coupling}

\bdefnnl{\zitaa{SN}{\cch{V}}}{D4.15}
Let $\cH$ and $\cL$ be Hilbert spaces and let $U\colon\cH\to\cH$ be a unitary operator.
Assume that $V_\cL\colon\cL\to\cH$ is an isometric operator such that $\mathring{\cL}$ is a wandering subspace of $\cH$ with respect to $U$.
The operator $\Phi_U^{V_\cL}\colon\cH\to L^2\rk{\cL}$ defined by
\begin{align}\label{E4.20}
    \ek{\Phi_U^{V_\cL}h}\rk{t}&\defeq\sum_{k=-\infty}^\infty t^kV_\cL^\ad U^{-k}h,&
    h&\in\cH,&t\in\T,
\end{align}
is called \emph{the Fourier representation} that corresponds to the operators $U$ and $V_\cL$.
\edefn

Obviously, the operator $\Phi_U^{V_\cL}$ is coisometric with the initial subspace $\mm{\mathring{\cL}}$, \tie{},
\begin{align*}
    \rk{\Phi_U^{V_\cL}}^\ad\Phi_U^{V_\cL}&=P_{\mm{\mathring{\cL}}},&
    \Phi_U^{V_\cL}\rk{\Phi_U^{V_\cL}}^\ad&=\Iu{L^2\rk{\cL}}.
\end{align*}
Moreover, the equalities
\begin{align*}
    \Phi_U^{V_\cL}M_{\mathord{\pm}}\rk{\mathring{\cL}}&=L_{\mathord{\pm}}^2\rk{\cL},&
    \Phi_U^{V_\cL}U&=U_\cL^{\mathord{\times}}\Phi_U^{V_\cL}
\end{align*}
are valid, where $L_{\mathord{-}}^2\rk{\cL}\defeq L^2\rk{\cL}\ominus L_{\mathord{+}}^2\rk{\cL}$ and $U_\cL^{\mathord{\times}}\colon L^2\rk{\cL}\to L^2\rk{\cL}$ is the multiplication operator by $t$.

Let $\sigma$ be a unitary coupling of form \eqref{E4.15}.
Consider the operator $S_\sigma\colon\mm{\mathring{\cG}}\to\mm{\mathring{\cF}}$ defined by
\beql{E4.21}
S_\sigma
\defeq\rstr{P_{\mm{\mathring{\cF}}}}{\mm{\mathring{\cG}}}.
\eeq
Clearly, $S_\sigma$ is a contraction.
Since the subspaces $\mm{\mathring{\cF}}$ and $\mm{\mathring{\cG}}$ reduce the operator $U$, $S_\sigma$ intertwines the bilateral shifts $U_{\mathring{\cF}}\defeq\rstr{U}{\mm{\mathring{\cF}}}$ and $U_{\mathring{\cG}}\defeq\rstr{U}{\mm{\mathring{\cG}}}$, \tie{},
\[
U_{\mathring{\cF}}S_\sigma
=S_\sigma U_{\mathring{\cG}}.
\]

By means of the Fourier representations $\Phi_U^{V_\cF}$ and $\Phi_U^{V_\cG},$ one can introduce the operator $\hat{\vartheta}_\sigma\colon L^2\rk{\cG}\to L^2\rk{\cF}$ by
\[
\hat{\vartheta}_\sigma
\defeq\Phi_U^{V_\cF}S_\sigma\rk{\Phi_U^{V_\cG}}^\ad.
\]
Obviously, the equality $\hat{\vartheta}_\sigma=\Phi_U^{V_\cF}\rk{\Phi_U^{V_\cG}}^\ad$ is also true.
The operator $\hat{\vartheta}_\sigma$ is a contraction intertwining the bilateral shifts $U_\cF^{\mathord{\times}}$ and $U_\cG^{\mathord{\times}}$, \tie{},
\beql{E4.23}
U_\cF^{\mathord{\times}}\hat{\vartheta}_\sigma
=\hat{\vartheta}_\sigma U_\cG^{\mathord{\times}}.
\eeq
The operator $\hat{\vartheta}_\sigma$ is called \emph{the scattering operator of the coupling $\sigma$}.

It is well known (see, \teg{}, \cite{SN}) that if condition \eqref{E4.23} holds true, then there exists a unique (we do not distinguish between operator functions in $L^\infty\ek{\cG,\cF}$ that coincide \tae{}\ on $\T$) operator function $\vartheta_\sigma\in  CM\ek{\cG,\cF}$ such that
\begin{align}\label{E4.24}
    \vartheta_\sigma\rk{t}g\rk{t}&=\rk{\hat{\vartheta}_\sigma g}\rk{t},&
    g&\in L^2\rk{\cG}.
\end{align}
Moreover,
\[
\norm{\vartheta_\sigma}_{L^\infty\ek{\cG,\cF}}
=\norm{\hat{\vartheta}_\sigma}_{\ek{L^2\rk{\cG},L^2\rk{\cF}}}
\leq1.
\]
Following \cite{A-A}, we introduce the following notion.

\bdefnl{D4.16}
The operator function $\vartheta_\sigma\in  CM\ek{\cG,\cF}$ 
is called \emph{the scattering suboperator of the unitary coupling $\sigma$}.
\edefn

It follows from \eqref{E4.21}--\eqref{E4.24} that $\vartheta_{\sigma^{\rk{1}}}=\vartheta_\sigma$, where $\sigma^{\rk{1}}$ is the principal part of the coupling $\sigma$.
Moreover, if $\sigma$ and $\sigma'$ are unitarily equivalent couplings of form \eqref{E4.18}, then $\vartheta_\sigma=\vartheta_{\sigma'}$ (see, \teg{}, \cite{A-A} or \zitaa{BD-1}{\cpart{I}}).

From \eqref{E4.16}--\eqref{E4.17} and \eqref{E4.21}--\eqref{E4.24} it can be easily seen that $\vartheta_{\sigma^\ad}=\vartheta_\sigma^\ad$, $\vartheta_{\sigma^\sch}=\vartheta_\sigma^\sch$, where $\vartheta_\sigma^\ad\rk{t}\defeq\rk{\vartheta_\sigma\rk{t}}^\ad$, $\vartheta_\sigma^\sch\rk{t}\defeq\vartheta_\sigma^\ad\rk{\ko t}$, $t\in\T$.

It is very important that every contractive measurable operator function on $\T$ can be considered as the scattering suboperator of some unitary coupling.

\bthmnl{\cite{A-A}, \zitaa{BD-1}{\cpart{I}}, \cite{MR3208800}}{T4.18}
An arbitrary operator function $\vartheta\in  CM\ek{\cG,\cF}$ is the scattering suboperator of some minimal unitary coupling $\sigma$ of form \eqref{E4.15}, \tie{}, $\vartheta_\sigma=\vartheta$.
Moreover, the minimal coupling $\sigma$ is determined up to unitary equivalence.
\ethm

\subsection{Extensions of contractive measurable operator functions}\label{subsec4.4-1003}

\bdefnnl{\zitaa{BD-1}{\cpart{II}}}{D4.19}
Let $\sigma$ be a unitary coupling of form \eqref{E4.15}.
By \emph{a bilateral} (resp.\ \emph{unilateral input, unilateral output}) \emph{channel of the coupling} $\sigma$ we mean a triple
\beql{E4.25} 
\rk{\cN,\cL;V_\cL}\qquad(\text{resp.\ }\rk{\cN_-,\cL;V_\cL}, \rk{\cN_+,\cL;V_\cL}),
\eeq
where 
\begin{enui}
    \il{D4.19.a} $V_\cL\colon\cL\to\cH$ is an isometric operator;
    \il{D4.19.b} the subspace $\mathring{\cL}\defeq\ran V_\cL$ is wandering with respect to $U$;
    \il{D4.19.c} $\cN=\mm{\mathring{\cL}}$ (resp.\ $\cN_-=\mmm{\mathring{\cL}}$, $\cN_+=\mmp{\mathring{\cL}}$).
\end{enui}
By \emph{the multiplicity} of a bilateral (unilateral) channel of form \eqref{E4.25} we mean $\dim\cL$ ($=\dim\mathring{\cL}$).

The space $\cL$ is said to be \emph{channeled} for the corresponding channel and the isometry $V_\cL$ is called \emph{embedding} for it.
\edefn

Any unilateral input (resp.\ output) channel obviously generates a unique bilateral one.
Conversely, any bilateral channel decomposes into a unique pair of unilateral input and output ones.

Clearly, any pair of unilateral channels
\beql{E4.26}
\setb{\rk{\cN_-,\cL;V_\cL}, \rk{\cN_+,\cM;V_\cM}}
\eeq
for a coupling $\sigma$ of form \eqref{E4.15} generates the coupling
\[
\sigma'
\defeq\rk{\cH,\cL,\cM;U,V_\cL,V_\cM}.
\]
From the point of view of the scattering theory (see, \teg{}, \cite{LPh}), this coupling is determined (up to embedding isometries $V_\cL,V_\cM$) by the pair $\set{\mmm{\mathring{\cL}},\mmp{\mathring{\cM}}}$ of so called \emph{incoming} and \emph{outgoing} subspaces of $\cH$ with respect to $U$.
Therefore, the function $\vartheta_{\sigma'}\in  CM\ek{\cM,\cL}$ can be considered as the suboperator of scattering generated by the pair of unilateral channels \eqref{E4.26}.

\bdefnl{D4.20}
Let $\vartheta\in  CM\ek{\cG,\cF}$.
An operator function
\beql{E4.27}
\Xi
\defeq\Mat{\vartheta_{11}&\vartheta_{12}\\\vartheta_{21}&\vartheta_{22}}\in  CM\ek{\cG'\oplus\cG,\cF'\oplus\cF},\qquad
\vartheta_{22}
\defeq\vartheta,
\eeq
will be called \emph{an up-leftward extension of} $\vartheta$ (\emph{in the class of contractive measurable operator functions}).
In the particular cases, if $\cG'=\set{0}$ or $\cF'=\set{0}$, the function $\Xi$ can be written in the form
\[
\Xi
=\Mat{\vartheta_{12}\\\vartheta}\qquad
\text{or}\qquad
\Xi
=\mat{\vartheta_{21},\vartheta}
\]
and is called \emph{an upward or leftward extension of} $\vartheta$, respectively.
\end{defn}

\begin{defn}
Let $\vartheta\in  CM\ek{\cG,\cF}$ and $\Xi\in  CM\ek{\cG'\oplus\cG,\cF'\oplus\cF}$ be an up-leftward extension of $\vte$ of form \eqref{E4.27}.
Let
\beql{E4.28}
\tau
\defeq\rk{\cH,\cF'\oplus\cF,\cG'\oplus\cG;U,V_{\cF'\oplus\cF},V_{\cG'\oplus\cG}}
\eeq
be a minimal unitary coupling such that $\vte_\tau=\Xi$.
Let also
\[
\tau_{22}
\defeq\rk{\cH,\cF,\cG;U,V_\cF,V_\cG}
\]
be a coupling (not necessarily minimal) generated by the pair
\beql{E4.29}
\setb{\rk{\cN_-,\cF;V_\cF}, \rk{\cN_+,\cG;V_\cG}}
\eeq
of unilateral channels of coupling $\tau$, where
\[
V_\cF\defeq\rstr{V_{\cF'\oplus\cF}}{\cF},\qquad
V_\cG\defeq\rstr{V_{\cG'\oplus\cG}}{\cG},
\]
such that $\vte_{\tau_{22}}=\vte$.
The extension $\Xi$ is called \emph{regular} if the coupling $\tau_{22}$ is minimal.
\edefn

Note that the regularity of the extension  $\Xi$ of $\vte$ is obviously equivalent to the regularity of any of the extensions $\Xi^\ad$ of $\vte^\ad$ and $\Xi^\sch$ of $\vte^\sch$.

Regularity of the extension $\Xi$ means that the equality
\[
\mm{\mathring{\cF}}\vee\mm{\mathring{\cG}}
=\mm{\mathring{\cF}'\oplus\mathring{\cF}}\vee\mm{\mathring{\cG}'\oplus\mathring{\cG}}
\]
holds true.
Taking into account that
\[
\mm{\mathring{\cF}'}\perp\mm{\mathring{\cF}},\qquad
\mm{\mathring{\cG}'}\perp\mm{\mathring{\cG}},
\]
the latter is equivalent to the inclusions
\[
\mm{\mathring{\cF}'}\subset\mathfrak{R}_{\mathring{\cF}},\qquad
\mm{\mathring{\cG}'}\subset\mathfrak{R}_{\mathring{\cG}}.
\]
Obviously, these properties do not depend on the choice of a minimal coupling $\tau$.

\bthmnl{\cite{BD20}}{T4.21}
Let $\vartheta\in  CM\ek{\cG,\cF}$ and let $\sigma$ be a minimal unitary coupling of form \eqref{E4.15} such that $\vartheta_\sigma=\vartheta$.
There exists a bijective correspondence between pairs
\beql{E4.31}
\setb{\rk{\cN_-',\cF';V_{\cF'}}, \rk{\cN_+',\cG';V_{\cG'}}}
\eeq
of unilateral channels of the coupling $\sigma$ such that
\beql{E4.32}
\cN_-'\perp\mmm{\mathring{\cF}},\qquad
\cN_+'\perp\mmp{\mathring{\cG}}
\eeq
and regular extensions $\Xi\in  CM\ek{\cG'\oplus\cG,\cF'\oplus\cF}$ of form \eqref{E4.27}.
This correspondence is established by the equality $\Xi=\vartheta_\tau$, where $\tau$ is a coupling of form \eqref{E4.28} and
\[
V_{\cF'\oplus\cF}
\defeq V_{\cF'}P_{\cF'}+V_\cF P_\cF,\qquad
V_{\cG'\oplus\cG}
\defeq V_{\cG'}P_{\cG'}+V_\cG P_\cG.
\]
\ethm

Note also that for the other blocks of the matrix $\Xi$ that are different from $\vartheta$ the equalities
\[
\vartheta_{12}=\vartheta_{\tau_{12}},\qquad
\vartheta_{21}=\vartheta_{\tau_{21}},\qquad
\vartheta_{11}=\vartheta_{\tau_{11}}
\]
hold true.
These couplings
\beql{E4.33}
\begin{aligned}
    \tau_{12}&\defeq\rk{\cH,\cF',\cG;U,V_{\cF'},V_\cG},&
    \tau_{21}&\defeq\rk{\cH,\cF,\cG';U,V_{\cF},V_{\cG'}},\\
    &&\tau_{11}&\defeq\rk{\cH,\cF',\cG';U,V_{\cF'},V_{\cG'}}
\end{aligned}
\eeq
(not necessarily minimal ones) are generated by the corresponding pairs formed by unilateral input and output channels from \eqref{E4.29} and \eqref{E4.31}--\eqref{E4.32}.

\bcorl{C4.22}
Let $\Xi\in  CM\ek{\cG'\oplus\cG,\cF'\oplus\cF}$ be a regular up-leftward extension of form \eqref{E4.27} for an operator function $\vartheta\in  CM\ek{\cG,\cF}$.
Then the block $\vartheta_{11}\in  CM\ek{\cG',\cF'}$ is uniquely determined by the other three blocks of matrix \eqref{E4.27}.
\ecor

In the case of regularity, any non-trivial (\tie{}, different from $\vartheta$) extension of $\vartheta$ has some extremal property of its norm.

\bthmnl{\cite{BD20}}{T4.23}
Let $\vartheta\in  CM\ek{\cG,\cF}$.
Then for any non-trivial regular extension $\Xi\in  CM\ek{\cG'\oplus\cG,\cF'\oplus\cF}$ of $\vartheta$ the equality
\[
\norm{\Xi}_{L^\infty\ek{\cG'\oplus\cG,\cF'\oplus\cF}}=1
\]
is valid.
\ethm

In the case of finite-dimensional spaces $\cF,\cG,\cF',\cG'$ the functions $\vartheta$ and $\Xi$ may be considered as matrix-valued functions.
If $\dim\cF'=r$ and $\dim\cG'=s$, we may call the matrix function $\Xi$ of form \eqref{E4.27} an up-leftward extension of the matrix function $\vartheta$ \emph{by $r$ rows and $s$ columns}.
This case was investigated in \zitaa{D82}{\cpart{VI}}.

\bthmnl{\zitaa{D82}{\cpart{VI}}, \cite{DM}}{T4.24}
Let $\vartheta\in  CM\ek{\cG,\cF}$, $\dim\cG<\infty$, $\dim\cF<\infty$.
A matrix function $\Xi\in  CM\ek{\cG'\oplus\cG,\cF'\oplus\cF}$, $\dim\cG'=s<\infty$, $\dim\cF'=r<\infty$, of form \eqref{E4.27} is a regular up-leftward extension of the matrix function $\vartheta$ by $r$ rows and $s$ columns if and only if the equalities
\begin{align*}
    \rank\rkb{\IG-\vartheta^\ad\rk{t}\vartheta\rk{t}}&=r+\rank\rkb{\IG-\vartheta_{12}^\ad\rk{t}\vartheta_{12}\rk{t}-\vartheta^\ad\rk{t}\vartheta\rk{t}}
    \intertext{and}
    \rank\rkb{\IF-\vartheta\rk{t}\vartheta^\ad\rk{t}}&=s+\rank\rkb{\IF-\vartheta_{21}\rk{t}\vartheta_{21}^\ad\rk{t}-\vartheta\rk{t}\vartheta^\ad\rk{t}}
\end{align*}
hold true $\underline{\lambda}$\nobreakdash-almost everywhere.
\ethm

Note that in \cite{BD20} a different approach to the study of regular extensions was used.
It is based on the notion of regular factorization of contractive measurable operator functions.
It is easy to check the equivalence of these two approaches.

\subsection{Interrelations between unitary couplings and unitary colligations}\label{subsec4.5-1003}

\bdefnnl{\cite{A-1}}{D4.24}
A unitary coupling $\sigma$ of form \eqref{E4.15} is called \emph{orthogonal} if
\beql{E4.34}
\mmm{\mathring{\cF}}\perp\mmp{\mathring{\cG}}.
\eeq
\edefn

Orthogonal couplings are characterized by the fact that their scattering suboperators are strong boundary value functions of Schur operator functions.

\bthmnl{\cite{A-1}}{T4.25}
A unitary coupling $\sigma$ of form \eqref{E4.15} is orthogonal if and only if $\vartheta_\sigma\in L_+^\infty\ek{\cG,\cF}$.
\ethm
\bproof
Recall that the Fourier-series expansion of $\vartheta_\sigma$ is given by
\[
\vartheta_\sigma\rk{t}
=\sum_{k=-\infty}^\infty t^k\vartheta_k,
\qquad t\in\T,
\quad \vartheta_k\in\ek{\cG,\cF}
\quad (k=0,\pm1,\pm2,\dotsc),
\]
(see, \teg{}, \zitaa{SN}{\cch{V}}) which means that for any $g\in\cG$ we have
\beql{E4.35}
\vartheta_\sigma\rk{t}g
=\sum_{k=-\infty}^\infty t^k\rk{\vartheta_k g}
\eeq
in $L^2\rk{\cF}$.
On the other hand, using representation \eqref{E4.20} for $\cL\defeq\cF$ and $h\defeq\rk{\Phi_U^{V_\cG}}^\ad g(=V_\cG g)$, we obtain
\beql{E4.36}
\vartheta_\sigma\rk{t}g
=\rk{\vartheta_\sigma g}\rk{t}
=\Phi_U^{V_\cF}\rk{\Phi_U^{V_\cG}}^\ad g
=\sum_{k=-\infty}^\infty t^k\rk{V_\cF^\ad U^{-k} V_\cG g},
\quad g\in\cG.
\eeq
From \eqref{E4.35} and \eqref{E4.36} it follows that
\beql{E4.37}
\vartheta_k
=V_\cF^\ad U^{-k} V_\cG
\quad\rk{k=0,\pm1,\pm2,\dotsc}.
\eeq
Clearly, $\vartheta_k=0$ if and only if $\mathring{\cF}\perp U^{-k}\mathring{\cG}$ and, hence, $\vartheta_\sigma\in L_+^\infty\ek{\cG,\cF}$ if and only if $\mathring{\cF}\perp U^{-k}\mathring{\cG}$ for $k\in\set{-1,-2,\dotsc}$.
The latter condition is obviously equivalent to condition \eqref{E4.34}.
\eproof

Following \cite{A-1}, for any orthogonal coupling $\sigma$ of form \eqref{E4.15} one can consider the 7-tuple
\beql{E4.38}
\Dl
\defeq\rk{\cH_\Dl,\cF,\cG;T,F,G,S},
\eeq
where
\begin{gather}
    \cH_\Dl\defeq\cH\ominus\rkb{\mmm{\mathring{\cF}}\oplus\mmp{\mathring{\cG}}},\qquad
    T\defeq\rstr{\rk{P_{\cH_\Dl}U}}{\cH_\Dl},\label{E4.39}\\
    F\defeq P_{\cH_\Dl}V_\cF,\qquad
    G\defeq\rstr{\rk{V_\cG^\ad U}}{\cH_\Dl},\qquad
    S\defeq V_\cG^\ad V_\cF.\nonumber
\end{gather}
Taking into account the equality $U\rk{\cH_\Dl\oplus U^\ad\mathring{\cF}}=\cH_\Dl\ominus\mathring{\cG}$, one can easily verify by direct calculation that the operator
\[
U_\Dl
\defeq\Mat{T&F\\G&S}\colon\cH_\Dl\oplus\cF\to\cH_\Dl\oplus\cG
\]
is unitary.
Therefore, $\Dl$ is a unitary colligation which is said to be \emph{generated by the orthogonal coupling $\sigma$}.
The subspaces $\cH_{\mathrm{int}}\defeq\cH_\Dl$ and $\cH_{\mathrm{ext}}\defeq\cH_{\mathrm{int}}^\bot$ are called \emph{internal} and \emph{external} for the coupling $\sigma$, respectively.
The operator $T$ given in \eqref{E4.39} is called \emph{fundamental} for $\sigma$ (as well as for $\Dl$).
Sometimes we denote it by $T_\sigma$.

The converse construction can also easily be realized (see \cite{A-1}, \zitaa{BD-1}{\cpart{I}}, \cite{MR3208800}).
Let $\Dl$ be a unitary colligation of form \eqref{E4.38}.
Consider two arbitrary unilateral shifts
\[
V_-\colon\mmp{\cF}\to\mmp{\cF},\qquad
V_+\colon\mmp{\cG}\to\mmp{\cG}
\]
with generating wandering subspaces $\cF$ and $\cG$, respectively.
Denote
\begin{gather*}
    \cH\defeq\mmp{\cF}\oplus\cH_\Dl\oplus\mmp{\cG},\;
    U\defeq V_-^\ad P_{\mmp{V_-\cF}}+U_\Dl P_{\cH_\Dl\oplus\cF}+V_+P_{\mmp{\cG}},\\
    V_\cF\defeq\rstr{U}{\cF}(=\rstr{U_\Dl}{\cF}),\qquad
    V_\cG\defeq\rstr{P_\cG}{\cG}.
\end{gather*} 
Then $\sigma\defeq\rk{\cH,\cF,\cG;U,V_\cF,V_\cG}$ is a unitary coupling.
Since $\mathring{\cF}\defeq U\cF$, $\mathring{\cG}\defeq\cG$, we see that $\mmp{\cF}$ for $V_-$ can be rewritten as $\mmm{\mathring{\cF}}$ for $V_\cF$ with respect  to $U$ and that $\mmp{\cG}$ for $V_+$ turns into $\mmp{\mathring{G}}$ for $V_\cG$ and $U$.
Thus, $\sigma$ is orthogonal.
Obviously, the so constructed coupling $\sigma$ generates the original colligation $\Dl$.
Moreover, $\sigma$ is determined by $\Dl$ up to unitary equivalence.
Clearly, $\sigma$ is minimal if and only if $\Dl$ is simple.

\bthmnl{\cite{A-1}}{T4.26}
Assume that $\sigma$ is an orthogonal unitary coupling of form \eqref{E4.15} and $\Dl$ is the unitary colligation generated by $\sigma$.
Then $\vartheta_\sigma=\bv{\theta_{\Dl^\sch}}(=\bv{\theta^\sch_\Dl})$.
\ethm
\bproof
Taking into account \rthm{T4.25} and formulas \eqref{E4.37} and \eqref{E4.39}, we obtain
\[\begin{split}
    \vartheta\rk{t}
    &=\sum_{k=0}^\infty t^kV_\cF^\ad U^{-k} V_\cG\\
    &=V_\cF^\ad V_\cG+\sum_{k=1}^\infty t^k\rk{V_\cF^\ad P_{\cH_\Dl}}\rk{P_{\cH_\Dl}U^{-k+1}P_{\cH_\Dl}}\rk{P_{\cH_\Dl}U^\ad V_\cG}\\
    &=S^\ad+t\sum_{n=0}^\infty t^nF^\ad\rk{T^\ad}^nG^\ad
    =\bv{\theta^\sch_\Dl}\rk{t}
    =\bv{\theta_{\Dl^\sch}}\rk{t}.\qedhere
\end{split}\]
\eproof

The dual relation between $\vartheta_\sigma$ and $\theta_\Dl$ obtained in \rthm{T4.26} follows from the existing duality of approaches in defining the scattering suboperators in the scattering theory and the transfer functions in the theory of open systems (see, \teg{}, \cite{A-A,A-1,B,LPh,SN}).

\subsection{Suboperator of internal scattering of an orthogonal unitary coupling}\label{s4.6}

\breml{R4.27}
Note that a unilateral input (\tresp{}\ output) channel $\rk{\cN_-,\cL;V_\cL}$ (\tresp{}\ $\rk{\cN_+,\cL;V_\cL}$) of a minimal coupling $\sigma$ of form \eqref{E4.15} is restored by the corresponding coshift $\tilde V\colon\cN_-\to\cN_-$ (\tresp{}\ shift $V\colon\cN_+\to\cN_+$) up to an embedding isometry $V_\cL$.
Therefore, in view of \rdefnss{D4.15}{D4.16}, any pair
\[
\set{\tilde V\colon\cN_-'\to\cN_-';\quad V\colon\cN_+'\to\cN_+'}
\]
of unilateral coshift and shift contained in $U$ and satisfying condition \eqref{E4.32} determines the corresponding regular extension $\Xi$ of form \eqref{E4.27} (see \rthm{T4.21}) up to common left and right constant unitary factors for the upper row $\mat{\vartheta_{11},\vartheta_{12}}$ and the left column $\col\rk{\vartheta_{11},\vartheta_{21}}$, respectively.
\erem

\bdefnnl{\cite{BD20}}{D4.28}
Let $\sigma$ be an orthogonal minimal unitary coupling of form \eqref{E4.15}.
A unilateral input (\tresp{}\ output) channel $\rk{\cN_-,\cL;V_\cL}$ (\tresp{}\ $\rk{\cN_+,\cL;V_\cL}$) of the coupling $\sigma$, as well as the corresponding unilateral coshift $\tilde V\colon\cN_-\to\cN_-$ (\tresp{}\ shift $V\colon\cN_+\to\cN_+$) will be called \emph{internal} for the coupling $\sigma$ if $\cN_-\subseteq\cH_\mathrm{int}$ (\tresp{}\ $\cN_+\subseteq\cH_\mathrm{int}$).
Similarly, changing $\cH_\mathrm{int}$ for $\cH_\mathrm{ext}$, we define the notion of an \emph{external} unilateral input (resp.\ output) channel, as well as, an \emph{external} unilateral coshift (\tresp{}\ shift) for the coupling $\sigma$.
\edefn

The condition 
$\cN_-\subseteq\cH_\mathrm{int}$ (\tresp{}\ $\cN_+\subseteq\cH_\mathrm{int}$) means that the coshift $\tilde V\colon\cN_-\to\cN_-$ (\tresp{}\ $V\colon\cN_+\to\cN_+$) is contained in the fundamental contraction $T$ of the coupling $\sigma$ or, what is the same, of the generated colligation $\Dl$.

The unilateral channels
\[
\rkb{\mmm{\mathring{\cF}},\cF;V_\cF},
\qquad\rkb{\mmp{\mathring{\cG}},\cG;V_\cG}
\]
are called \emph{the principal external unilateral input and output channels} of the coupling $\sigma$, respectively.
The corresponding operators
\[
\tilde V_\sigma\defeq\rk{\rstr{U^\ad}{\mmm{\mathring{\cF}}}}^\ad,
\qquad V_\sigma\defeq\rstr{U}{\mmp{\mathring{\cG}}}
\]
are called \emph{the principal external unilateral coshift and shift of $\sigma$}, respectively.
If the coupling $\sigma$ is minimal, then the largest unilateral coshift $\tilde V_T$ and shift $V_T$ contained in $T$ are called \emph{the principal internal unilateral coshift and shift of $\sigma$}, respectively.

\bdefnl{D4.29}
Let $\theta\in\SFFG$.
An operator function
\beql{E4.41}
\Xi\defeq\begin{pmatrix}\theta_{11}&\theta_{12}\\\theta_{21}&\theta_{22}\end{pmatrix}\in\SF{\cF'\oplus\cF}{\cG'\oplus\cG},
\qquad\theta_{22}\defeq\theta,
\eeq
will be called an \emph{up-leftward Schur extension of the Schur function $\theta$}.
In particular cases, if $\cF'=\set{0}$ or $\cG'=\set{0}$, $\Xi$ can be written in the form
\[
\Xi\defeq\begin{pmatrix}\theta_{12}\\\theta\end{pmatrix}\in\SF{\cF}{\cG'\oplus\cG}\qquad\text{or}
\qquad\Xi\defeq\rk{\theta_{21},\theta}\in\SF{\cF'\oplus\cF}{\cG}
\]
and it will be called an \emph{upward or leftward Schur extension of $\theta$}, respectively.
A Schur extension $\Xi$ of form \eqref{E4.41} will be called \emph{regular} if its boundary value function $\bv{\Xi}$ is a regular extension of the boundary value function $\bv{\theta}$.
\edefn

Obviously, the regularity of a Schur extension $\Xi$ of $\te$ is equivalent to the regularity of the Schur extension $\Xi^\sch$ of $\te^\sch$.

Taking into account Theorems~\ref{T4.21},~\ref{T4.25} and \rrem{R4.27}, we obtain the following result.

\bthmnl{\cite{BD4,BD20,MR3208800}}{T4.30}
Let $\theta\in\SFFG$, let $\sigma$ be an orthogonal minimal coupling of form \eqref{E4.15} such that $\vartheta_\sigma=\bv{\theta}^\sch$.
\benui
\il{T4.30.a} There exists a bijective correspondence between pairs
\beql{E4.43}
\set{\tilde V\colon\mmm{\tcL}\to\mmm{\tcL};\quad V\colon\mmp{\cL}\to\mmp{\cL}}
\eeq
of internal unilateral coshifts and shifts of coupling $\sigma$, respectively, which satisfy the condition
\beql{E4.44}
\mmm{\tcL}\perp\mmp{\cL}
\eeq
and regular up-leftward Schur extensions $\Xi$ of form \eqref{E4.41} for the function $\theta$ up to left and right common constant unitary factors for the upper row and the left column of $\Xi$, respectively.
\il{T4.30.b} The function $\theta$ admits a two-sided inner regular extension of form \eqref{E4.41} if and only if there exists a pair
of internal unilateral coshift and shift of form \eqref{E4.43} for the coupling $\sigma$, respectively, which satisfies the conditions \eqref{E4.44} and 
\beql{E4.45}
\mm{\tcL}=\mathfrak{R}_{\mathring{\cF}},
\qquad\mm{\cL}=\mathfrak{R}_{\mathring{\cG}}.
\eeq
\eenui
\ethm

Note that if there exists a pair of form \eqref{E4.43} satisfying \eqref{E4.44} and \eqref{E4.45}, then the pair $\set{\tilde V_T,V_T}$ also satisfies \eqref{E4.45}, but not necessarily \eqref{E4.44}.

Consider the regular upward and leftward Schur extensions
\beql{E4.46}
\Omega_0\defeq\begin{pmatrix}\theta_{12}^{(0)}\\\theta\end{pmatrix}\in\SF{\cF}{\cG_0\oplus\cG},
\qquad\Upsilon_0\defeq\mat{\theta_{21}^{(0)},\theta}\in\SF{\cF_0\oplus\cF}{\cG}
\eeq
of a Schur function $\theta\in\SFFG$ which, by \rthm{T4.30}, correspond to the principal internal unilateral shift $V_T$ and coshift $\tilde V_T$ of the coupling $\sigma$, respectively.
In this case, $\bv{\theta_{12}^{(0)}}=\vartheta_{\tau_{21}^{(0)}}^\sch$, $\bv{\theta_{21}^{(0)}}=\vartheta_{\tau_{12}^{(0)}}^\sch$, where
\beql{E4.47}
\tau_{12}^{(0)}\defeq\rk{\cH,\cF_0,\cG;U,V_{\cF_0},V_{\cG}},
\qquad\tau_{21}^{(0)}\defeq\rk{\cH,\cF,\cG_0;U,V_{\cF},V_{\cG_0}} 
\eeq
are the orthogonal couplings of type \eqref{E4.33} which are not necessarily minimal and generated by
\beql{E4.48}
\tilde V_T\colon\mmm{U\tcL_0}\to\mmm{U\tcL_0},
\qquad V_T\colon\mmp{\cL_0}\to\mmp{\cL_0}
\eeq
up to embedding isometries $V_{\cF_0}$, $V_{\cG_0}$ such that $\mathring{\cF}_0\defeq U\tcL_0$, $\mathring{\cG}_0\defeq\cL_0$, respectively (see \rcor{cor1.9}).

\bthmnl{\zitaa{MR3208800}{\cthm{6.1}}, \zitaa{BD20}{\cthm{7.14}}}{T4.31}
Let $\theta\in\SFFG$, let $\sigma$ be an orthogonal minimal unitary coupling of form \eqref{E4.15} such that $\vartheta_\sigma=\bv{\theta}^\sch$, and let $T\defeq T_\sigma$.
Let $\Omega_0\in\SF{\cF}{\cG_0\oplus\cG}$ \textnormal{(}\tresp{}\ $\Upsilon_0\in\SF{\cF_0\oplus\cF}{\cG}$\textnormal{)} be the regular upward \textnormal{(}\tresp{}\ leftward\textnormal{)} Schur extension of form \eqref{E4.46} for $\theta$ corresponding to the principal internal unilateral shift $V_T$ \textnormal{(}resp.\ coshift $\tilde V_T$\textnormal{)} of the coupling $\sigma$.
Then $\theta_{12}^{(0)}=\theta_r$ \textnormal{(}resp.\ $\theta_{21}^{(0)}=\theta_\ell$\textnormal{)} up to a left \textnormal{(}resp.\ right\textnormal{)} constant unitary factor, where $\theta_r\in\SF{\cF}{\cG_0}$ \textnormal{(}resp.\ $\theta_\ell\in\SF{\cF_0}{\cG}$\textnormal{)} is the right \textnormal{(}resp.\ left\textnormal{)} defect function of $\theta$.
\ethm

In the following, for convenience, we will redesignate $\theta_r$ (resp.\ $\theta_\ell$), namely, $\varphi\defeq\theta_r$ (resp.\ $\psi\defeq\theta_\ell$) and call it \emph{the defect} (resp.\ \emph{$\ast$\nobreakdash-defect}) \emph{function of $\theta$ in the Schur class}.
From now on, taking into account \rthmsss{T4.8}{T4.30}{T4.31}, the defect (resp.\ $\ast$\nobreakdash-defect) function $\varphi\in\SF{\cF}{\cG_0}$ (resp.\ $\psi\in\SF{\cF_0}{\cG}$) of $\theta\in\SFFG$ will be considered only up to a left (resp.\ right) constant unitary factor.
Thus, from \eqref{E4.2} (resp.\ \eqref{E4.3}) it follows that
\[
\varphi\rk{\ze}=V_{\cG_0}^\ad\rk{\IH-\zeta T}^\inv F
\qquad(\text{resp.\ }\psi\rk{\ze}=G\rk{\IH-\zeta T}^\inv V_{\cF_0}),
\qquad\zeta\in\D,
\]
where $V_{\cG_0}$ (resp.\ $V_{\cF_0}$) is an arbitrary embedding isometry such that $\mathring{\cG}_0=\cL_0$ (resp.\ $\mathring{\cF}_0=U\tcL_0$).

From \rthm{T4.31} and \eqref{E4.47} for the defect (resp.\ $*$-defect) function $\varphi\in\SF{\cF}{\cGo}$ (resp.\ $\psi\in\SF{\cFo}{\cG}$) of a Schur function $\te\in\SFFG$ such that $\bv\te=\vte_\sg^\sch$, where $\sg$ is an orthogonal minimal coupling of form \eqref{E4.15}, it follows that $\bv\varphi=\vte_{\sg_\varphi}^\sch$ (resp.\ $\bv\psi=\vte_{\sg_\psi}^\sch$), where $\sg_\varphi\defeq\tau_{21}^\oo$ (resp.\ $\sg_\psi\defeq\tau_{12}^\oo$), \tie,
\[
\sg_\varphi\defeq\rk{\cH,\cF,\cGo;U,V_\cF,V_\cGo}
\qquad\text{(resp.\ }\sg_\psi\defeq\rk{\cH,\cFo,\cG;U,V_\cFo,V_\cG}\text{).}
\]
This orthogonal coupling is not necessarily minimal.
The corresponding unitary colligation $\Dl_\varphi$ (resp.\ $\Dl_\psi$) generated by the coupling $\sg_\varphi$ (resp.\ $\sg_\psi$) such that $\varphi=\te_{\Dl_\varphi}$ (resp.\ $\psi=\te_{\Dl_\psi}$) can be constructed as it was done in \rsubsec{subsec4.5-1003} (see \eqref{E4.38}--\eqref{E4.39}).

\rcor{C4.22} and \rthm{T4.31} enable us to introduce an important concept.

\bdefnnl{\cite{MR1491502}}{D4.32}
Let $\theta\in\SFFG$ and let $\vphi\in\SF{\cF}{\cGo}$, $\psi\in\SF{\cFo}{\cG}$ be its defect and $\ast$\nbd-defect functions, respectively.
Assume that $\sigma$ is an orthogonal minimal unitary coupling of form \eqref{E4.15} such that $\vte_\sg=\bv\te^\sch$.
The unique (up to constant unitary factors on both sides) operator function $\chi\in  CM\ek{\cFo,\cGo}$ such that the operator functions
\[
\Xio\defeq\begin{pmatrix}
    \chi&\bv\vphi\\
    \bv\psi&\bv\te
\end{pmatrix}\in  CM\ek{\cFo\oplus\cF,\cGo\oplus\cG}
\]
is a regular up-leftward extension of the function $\bv\te\in  CM\ek{\cF,\cG}$ is called \emph{the suboperator of internal scattering of the coupling $\sigma$}.
\edefn

The justification of such a name of the function $\chi$ is that $\chi\defeq\vte_{\sg_\chi}^\sch$, where
\[
\sg_\chi
\defeq\rk{\cH,\cFo,\cGo;U,V_\cFo,V_\cGo}
\]
is a coupling (not necessarily orthogonal) of type $\tau_{11}$ from \eqref{E4.33} generated by the pair of the principal internal unilateral coshift $\tilde V_T$ and shift $V_T$ of form \eqref{E4.48} for $\sg$ and the pair of embedding isometries $V_\cFo$ and $V_\cGo$ from \eqref{E4.47}.

By virtue of \rcor{C4.22} and \rthm{T4.23}, the function $\chi\in  CM\ek{\cFo,\cGo}$ is a unique solution of some extremal problem.

\bthmnl{\cite{MR1491502}}{T4.33}
Let $\te\in\SFFG$ and let $\vphi\in\SF{\cF}{\cGo}$, $\psi\in\SF{\cFo}{\cG}$ be its defect and $\ast$\nbd-defect functions, respectively.
Then the suboperator $\chi\in  CM\ek{\cFo,\cGo}$ of internal scattering is the unique function amongst all functions $\xi\in  CM\ek{\cFo,\cGo}$ at which the infimum
\[
\alpha\defeq\inf_\xi\norm{\Xi}_{L^\infty\ek{\cFo\oplus\cF,\cGo\oplus\cG}},
\qquad\Xi\defeq\begin{pmatrix}
    \xi&\bv\vphi\\
    \bv\psi&\bv\te
\end{pmatrix},
\]
is attained.
Moreover, $\alpha=1$.
\ethm

In order to obtain an explicit representation of the function $\chi$ in terms of $\bv\te$, $\bv\vphi$, and $\bv\psi$, we need the following generalization of the Moore--Penrose inverse matrix.
If $A\in\ek{\cF,\cG}$, we denote by $A^\gi$ the densely defined operator (closed, but not necessarily bounded) with the domain
\[
\dom A^\gi
\defeq\ran A\oplus\rk{\ran A}^\bot\subseteq\cG
\]
and the range
\[
\ran A^\gi
\defeq\rk{\Ker A}^\bot\subseteq\cF
\]
setting $A^\gi g\defeq0$ if $g\in\rk{\ran A}^\bot$, and $A^\gi g\defeq f$ if $g\in\ran A$, $Af=g$, $f\in\rk{\Ker A}^\bot$.

\bleml{L4.34}
Let $\te\in\SFFG$ and let $\vphi\in\SF{\cF}{\cGo}$, $\psi\in\SF{\cFo}{\cG}$ be its defect and $\ast$\nbd-defect functions in the Schur class, respectively.
Let also $\Pi\in  CM\ek{\cF}$, $\Sg\in  CM\ek{\cG}$ be the defect and $\ast$\nbd-defect functions of $\bv\te\in  CM\ek{\cF,\cG}$ in the class of contractive measurable functions, respectively, defined by \eqref{E4.12}--\eqref{E4.13}.
Then the inclusions
\beql{E4.49}
\ran\bv\vphi^\ad\rk{t}\subseteq\ran\Pi\rk{t},
\qquad\ran\bv\psi\rk{t}\subseteq\ran\Sg\rk{t}
\eeq
hold true for almost all $t\in\T$.
\elem
\bproof
Inclusions \eqref{E4.49} follow from the non-negativity of the matrices
\[
\begin{pmatrix}
    \Pi^2&\bv\vphi^\ad\\
    \bv\vphi&\Iu{\cGo}
\end{pmatrix}
\qquad\text{and}
\qquad\begin{pmatrix}
    \Sg^2&\bv\psi\\
    \bv\psi^\ad&\Iu{\cFo}
\end{pmatrix}
\]
which in turn follows from \rthm{T4.8} (see, \teg{}, \zitaa{MR3208800}{\clem{3.1}}).
\eproof

\bthmnl{\zitaa{MR3208800}{\cthm{6.8}}}{T4.35}
Let the conditions of \rlem{L4.34} be satisfied.
The function $\chi\in  CM\ek{\cFo,\cGo}$ admits the representation $\chi=-\omega_0\bv\te^\ad\upsilon_0$, where $\omega_0\in  CM\ek{\cF,\cGo}$ and $\upsilon_0\in  CM\ek{\cFo,\cG}$ are defined by the equalities
\[
\omega_0^\ad\rk{t}\defeq\rkb{\Pi\rk{t}}^\gi\bv\vphi^\ad\rk{t},
\qquad\upsilon_0\rk{t}\defeq\rkb{\Sg\rk{t}}^\gi\bv\psi\rk{t},
\qquad t\in\T,
\]
and satisfy the conditions
\[
\omega_0\rk{t}\omega_0^\ad\rk{t}=\Iu{\cGo},
\qquad\upsilon_0^\ad\rk{t}\upsilon_0\rk{t}=\Iu{\cFo}
\qquad\text{\tae{}\ on }\T.
\]
\ethm

Another representation of $\chi$ derived by different methods can be found in \cite{MR1737611} or in \zitaa{MR3208800}{\cthm{6.3}}.
Using the functions $\vphi$, $\psi$, and $\chi$, we can obtain a description of Schur extensions for $\te$ (\tcf{}\ with \rthm{T4.30}).

\begin{thm}\label{th4.36}
    Let $\te\in\SFFG$ and let $\vphi\in\SF{\cF}{\cGo}$, $\psi\in\SF{\cFo}{\cG}$ be its defect and $\ast$\nbd-defect functions, respectively, and let $\chi\in  CM\ek{\cFo,\cGo}$ be the suboperator of internal scattering.
    \benui
    \il{T4.36.a}\textnormal{(\cite{BD4},\zitaa{BD20}{\cthm{8.22(b)}})} 
    There exists a bijective correspondence between pairs
    \beql{E4.50}
    \setb{\omega\in\SF{\cGo}{\cG'},\upsilon\in\SF{\cF'}{\cFo}}
    \eeq
    of $\ast$\nbd-inner and inner operator functions, respectively, such that
    \beql{E4.51}
    \bv\omega\chi\bv\upsilon
    \in L_+^\infty\ek{\cF',\cG'}
    \eeq
    and regular up-leftward Schur extensions $\Xi$ of form \eqref{E4.41} for the function $\te$.
    This correspondence is established by the equalities
    \[
    \te_{12}=\omega\vphi,
    \qquad\te_{21}=\psi\upsilon,
    \qquad\bv{\te_{11}}=\bv\omega\chi\bv\upsilon.
    \]
    \il{T4.36.b}\textnormal{(\zitaa{MR2013544}{\cthm{2}}, \zitaa{BD20}{\cthm{8.22(c)}})} 
    The function $\te$ admits a two-sided inner regular extension of form \eqref{E4.41} if and only if the two factorizations
    \beql{E4.52}
    \bv\vphi^\ad\bv\vphi=\Pi^2,
    \qquad\bv\psi\,\bv\psi^\ad=\Sg^2
    \eeq
    hold true where $\Pi\in CM\ek{\cF}$ and $\Sg\in CM\ek{\cG}$ are defined by \eqref{E4.12}--\eqref{E4.13} and there exists a pair $\set{\omega,\upsilon}$ of two-sided inner operator functions of form \eqref{E4.50} satisfying \eqref{E4.51}.
    \eenui
\end{thm}


The problem solved in \rpart{T4.36.b} of \rthm{th4.36} (but without the assumption of regularity) is known as the problem of \emph{the Darlington synthesis} arising in the electrical network theory (see \cite{MR0001658}).

The function $\chi\in  CM\ek{\cFo,\cGo}$ uniquely determined by the functions $\te\in\SFFG$, $\vphi\in\SF{\cF}{\cGo}$, and $\psi\in\SF{\cFo}{\cG}$ was earlier used by D.~Z.~Arov in the case of the validity of both factorizations \eqref{E4.52} for a description of all two-sided inner regular extensions of $\te$ (see \zitaa{MR2013544}{\cthm{2}}).
In the general case, the uniqueness of the function $\chi$ was announced in \cite{MR1491502}, where $\chi$ was interpreted as the suboperator of internal scattering of a corresponding system.
This made it possible to obtain a number of important results in the theory of passive scattering systems (see, \teg, \cite{MR1752830,MR1957950}).

\subsection{Regular Schur extensions of Schur functions and loss decrease in unitary couplings}\label{subsec4.7-1003}

Unitary couplings can be considered as mathematical models of discrete time scattering system of various physical nature (see, \teg{}, references in \cite{A-1,MR2013544}).

Following the ideas in \cite{A-1}, we can introduce the notion of a coupling without losses.

\bdefnl{D4.37}
\benui
\il{D4.37.a} A unitary coupling $\sigma$ of form \eqref{E4.15} we call \emph{lossless at the input} (resp.\ \emph{output}) if values $\vte_\sg\rk{t}\in\ek{\cG,\cF}$ of its scattering suboperator are isometric (resp.\ coisometric) at almost all $t\in\T$.
\il{D4.37.b} The coupling $\sg$ is termed \emph{lossless} if it is lossless at both the input and output, \tie{}, values $\vte_\sg\rk{t}$ are unitary at almost all $t\in\T$.
Otherwise, it is said to be \emph{a coupling with losses}.
\eenui
\edefn

\blemnl{\cite{A-A}, \zitaa{BD-1}{\cpart{II}}}{L4.38}
Let $\sg$ be a minimal unitary coupling of form \eqref{E4.15}.
\benui
\il{L4.38.a} $\sg$ is lossless at the input \nrk{resp. output} if and only if
\[
\mm{\mathring{\cF}}=\cH
\qquad(\text{resp.\ }\mm{\mathring{\cG}}=\cH).
\]
\il{L4.38.b} $\sg$ is lossless if and only if
\[
\mm{\mathring{\cF}}
=\mm{\mathring{\cG}}
(=\cH).
\]
\eenui
\elem

To illustrate this assertion from a geometrical point of view, assume, for example, that $\mm{\mathring{\cG}}\neq\cH$.
Then there exists an input ``signal'' $h\in\mmm{\rcF}$ such that $h\notin\mm{\rcG}$, \tie{}, $h=h'+h''$, $h'\in\mm{\rcG}$, $h''\in\mathfrak{R}_{\rcG}$, $h''\neq0$.
For all time moments $n=1,2,\dotsc$ the evolution (propagation) $\{U^nh\}_{n=1}^\infty$ of the ``signal'' $h$ satisfies the conditions
\[
U^nh=U^nh'+U^nh'',
\quad U^nh'\in\mm{\rcG},
\quad U^nh''\in\mathfrak{R}_{\rcG},
\quad\norm{U^nh''}=\norm{h''}>0.
\]
Thus, in the course of evolution, only the part $\{U^nh'\}_{n=1}^\infty$ of $\{U^nh\}_{n=1}^\infty$ remains in the subspace $\mm{\rcG}$, while the nonzero part $\{U^nh''\}_{n=1}^\infty$ can be interpreted as scattering losses of the ``signal'' at the output.
In the dual way, we can illustrate losses at the input.

Assume that $\Pi_\sg\in  CM\ek{\cG}$ and $\Sg_\sg\in  CM\ek{\cF}$ be the defect and $\ast$\nbd-defect functions of $\vte_\sg\in  CM\ek{\cG,\cF}$ in the class of contractive measurable functions, \tie,
\[
\Pi_\sg\rk{t}\defeq\rkb{\IG-\vte_\sg^\ad\rk{t} \vte_\sg\rk{t}}^{1/2},\quad
\Sigma_\sg\rk{t}\defeq\rkb{\IF-\vte_\sg\rk{t} \vte_\sg^\ad\rk{t}}^{1/2},\qquad
t\in\T.
\]
The quantity $\norm{\Pi_\sg}$ (resp.\ $\norm{\Sg_\sg}$) is a measure of non-isometricity (resp.\ non-coisometricity) of $\vte_\sg$.
Hence, the quantity $\ell_\sg\defeq\max\set{\norm{\Pi_\sg},\norm{\Sg_\sg}}$ can be considered as a characteristic of the amount of losses in $\sg$.
As it is shown in \zitaa{BD-1}{\cpart{II}, \cssec{7.8}},
\[
\ell_\sg
=\norm{P_\mm{\rcF}-P_\mm{\rcG}},
\]
\tie{}, $\ell_\sg$ is the aperture of the subspaces $\mm{\rcF}$ and $\mm{\rcG}$ (see, \teg, \cite{MR1255973}).

Further in this subsection, for simplicity, we will consider only orthogonal couplings.
From \rdefn{D4.37} and Theorems~\ref{T4.25}--\ref{T4.26} we obtain that such a coupling $\sg$ of form \eqref{E4.15} is lossless at the input (resp.\ output) if and only if $\vte_\sg=\bv\te^\sch$, where $\te\in\SFFG$ is an $\ast$\nbd-inner (resp.\ inner) function.
Moreover, $\sg$ is lossless if and only if $\te$ is two-sided inner.

\blemnl{\zitaa{BD-1}{\cpart{II}}}{L4.39}
Let $\sg$ be an orthogonal unitary coupling of form \eqref{E4.15}.
\benui
\il{L4.39.a} $\sg$ is lossless at the input \nrk{resp. output} if and only if the inclusion
\beql{E4.53}
\mmp{\rcG}\subset\mmp{\rcF}
\qquad(\text{resp.\ }\mmm{\rcF}\subset\mmm{\rcG})
\eeq
is valid.
\il{L4.39.b} $\sg$ is lossless if and only if both inclusions \eqref{E4.53} are valid.
\eenui
\elem

For any orthogonal minimal coupling $\sg$ of form \eqref{E4.15} the ``signal'' propagation in the space $\cH$ can be schematized as the following diagram
\beql{E4.54}
\boxed{\mmm{\rcF}}
\quad\overset{\longrightarrow}{\underset{\longrightarrow}{\oplus}}\quad\boxed{\cH_\mathrm{int}^{\rk{\sg}}}
\quad\overset{\longrightarrow}{\underset{\longrightarrow}{\oplus}}\quad\boxed{\mmp{\rcG}}\;.
\eeq
This means that the ``signal'' through the principal input channel (incoming subspace) enters the internal space (possibly with some losses), from where it then exits (also with some possible losses) through the principal output channel (outgoing subspace).

Let $\te\in\SFFG$ be such that $\bv\te=\vte_\sg^\sch$ and let
\[
\Omega
\defeq\begin{pmatrix}\te_{12}\\\te\end{pmatrix}
\in\SF{\cF}{\cG'\oplus\cG}
\]
be a regular upward Schur extension of $\te$.
By \rthm{T4.30}, there exists a unique internal unilateral output channel $\rk{\cN_+,\cG';V_{\cG'}}$ such that $\bv\Omega=\vte_\tau^\sch$, where the corresponding orthogonal coupling $\tau$ has the form
\[
\tau\defeq\rk{\cH,\cF,\cG'\oplus\cG;U,V_\cF,V_{\cG'\oplus\cG}},
\qquad
V_{\cG'\oplus\cG}:=V_{\cG'}P_{\cG'}+V_{\cG}P_{\cG}.
\]
Denoting $\cH_\mathrm{int}^{\rk{\tau}}\defeq\cH_\mathrm{int}^{\rk{\sg}}\ominus\mmp{\rcG'}$, we obtain for $\tau$ the diagram
\beql{E4.55}
\quad
\boxed{\mmm{\rcF}}\quad\overset{\longrightarrow}{\underset{\longrightarrow}{\oplus}}\quad\boxed{\begin{array}{c}\vphantom{\boxed{\mmm{\rcG'}}}\\\cH_\mathrm{int}^{\rk{\tau}}\\\vphantom{\boxed{\mmm{\rcG}}}\end{array}}\;
\begin{array}{cc}
    \longrightarrow&\boxed{\mmp{\rcG'}}\\
    \oplus&\oplus\\
    \longrightarrow&\boxed{\mmp{\rcG}}
\end{array}.
\eeq
From \eqref{E4.54}--\eqref{E4.55} we see that the transition from $\sg$ to $\tau$ is realized by extending the principal external unilateral output channel of $\sg$ by adding some internal unilateral output one of $\sg$.
Since, obviously, $\Pi_\tau\leq\Pi_\sg$, we obtain that this transition leads to a loss decrease at the output of $\tau$ in comparison with $\sg$.

Let
\beql{E4.56}
\Omega_0
\defeq\begin{pmatrix}\vphi\\\te\end{pmatrix}
\in\SF{\cF}{\cGo\oplus\cG},
\eeq
where $\vphi\in\SF{\cF}{\cGo}$ is the defect function of $\te$ in the Schur class.
By \rthm{T4.31}, $\Omega_0$ is a regular upward Schur extension of $\te$ corresponding to the principal internal unilateral shift $V_T$ for $\sg$.
Hence, the corresponding orthogonal coupling has the form
\[
    \tau_0\defeq\rk{\cH,\cF,\cGo\oplus\cG;U,V_\cF,V_{\cGo\oplus\cG}}
\]
where $V_{\cGo\oplus\cG}\defeq V_\cGo P_\cGo+V_\cG P_\cG$ and, obviously, $\vte_{\tau_0}=\bv{\Omega_0}^\sch$.
Herewith the orthogonal coupling $\tau_0$ has the largest principal external unilateral output channel among all such orthogonal couplings $\tau$ considered above.
Therefore, we can say that $\tau_0$ has the least possible losses at the output in comparison with other such $\tau$.

The problem of the existence of a such orthogonal minimal coupling $\tau_0$ corresponding to $\Omega_0$ of form \eqref{E4.56} and being lossless at the output is reduced to the existence of the factorization $\bv\vphi^\ad\bv\vphi=\IF-\bv\te^\ad\bv\te$.
The latter is equivalent to the validity of the equality
\[
\mm{\rcG_0}
=\mathfrak{R}_\rcG
\]
(see \rthm{T4.30}).

In the dual way, we can describe the connection between regular leftward Schur extensions of $\te\in\SFFG$ and loss decrease at the input of the corresponding orthogonal minimal couplings.

If $\Xi\in\SF{\scF\oplus\cF}{\scG\oplus\cG}$ is an arbitrary regular up-leftward Schur extension of form \eqref{E4.41} for $\te\in\SFFG$, then $\bv\Xi=\vte_\tau^\sch$, where $\tau$ is an orthogonal minimal coupling of the form
\[
\tau
\defeq\rk{\cH,\cF'\oplus\cF,\cG'\oplus\cG;U,V_{\cF'\oplus\cF},V_{\cG'\oplus\cG}}.
\]
By \rthm{T4.30}, $\tau$ is obtained from the coupling $\sg$ of form \eqref{E4.15} such that $\vte_\sg=\bv\te^\sch$ by extending both the principal external unilateral channels of $\sg$ through the addition of the corresponding mutually orthogonal internal unilateral channels of $\sg$.
This can be illustrated by the following diagram
\[
\begin{array}{cc}
    \boxed{\mmm{\rcF'}}&\longrightarrow\\
    \oplus&\oplus\\
    \boxed{\mmm{\rcF}}&\longrightarrow
\end{array}
\quad\boxed{\begin{array}{c}
        \vphantom{\boxed{\mmm{\rcF'}}}\\
        \cH_\mathrm{int}^{\rk{\tau}}\\
        \vphantom{\boxed{\mmm{\rcF}}}
    \end{array}}
\quad\begin{array}{cc}
    \longrightarrow&\boxed{\mmp{\rcG'}}\\
    \oplus&\oplus\\
    \longrightarrow&\boxed{\mmp{\rcG}}
\end{array}\;,
\]
where $\cH_\mathrm{int}^{\rk{\tau}}\defeq\cH_\mathrm{int}^{\rk{\sg}}\ominus\rk{\mmm{\rcF'}\oplus\mmp{\rcG'}}$. 
Amongst all such orthogonal minimal couplings $\tau$ there exist those that are \emph{non-extensible}.
The latter means that neither the input nor the output principal external unilateral channel of $\tau$ can be extended by adding of an internal one without losing the coupling orthogonality (see \zitaa{BD20}{\cssec{8.4}}).
The problem of the existence of a lossless orthogonal coupling $\tau_0$ amongst them is equivalent to the existence of a two-sided inner regular extension of $\te$ (see \rthm{th4.36}\eqref{T4.36.b}).
From a geometrical point of view, this is equivalent to the existence of a pair
\[
\set{\rk{\cN_-,\scF;V_\scF},\;\rk{\cN_+,\scG;V_\scG}}
\]
of internal unilateral channels of $\sg$ such that
\[
\mm{\rcF'}=\mathfrak{R}_\rcF,
\qquad\mm{\rcG'}=\mathfrak{R}_\rcG,
\qquad\mmm{\rcF'}\perp\mmp{\rcG'}
\]
(see \rthm{T4.30}\eqref{T4.30.b}).

Unitary couplings with losses are considered by D.~Z.~Arov (\cite{A-1,MR2013544}).
The presentation of these questions in terms of unitary colligations and open systems see in \cite{BD4}.

\subsection{Pseudocontinuability of Schur functions and the Darlington synthesis}\label{subsec4.8-1003}

In this section we need one important class of functions first considered by R.~Nevanlinna in \cite{MR0417418} for the scalar case (see, \teg, \zitaa{RR1}{\cssec{4.2}}).

\bdefnl{D4.40}
A meromorphic \tval{$\ek{\cF,\cG}$} function $\kappa$ defined on $\D$ is called a \emph{bounded type} one if it can be represented as a quotient of two functions $\mu\in\SFDFG$ and $\eta\in\SFDs$.
This class of bounded type functions is also termed \emph{the Nevanlinna class on $\D$} and is denoted by $\NFDFG$.
\edefn

If $\De\defeq\setaca{\ze\in\C}{\abs{\ze}>1}\cup\set{\infty}$ is the exterior of the unit disk $\D$ in the extended plane $\clo{\C}$, we can introduce the class $\NFDeFG$ in the same way.

As earlier, we denote by $\bv\ka$ the boundary value functions on $\T$ for functions $\ka\in\NFDFG$ as well as for $\ka\in\NFDeFG$.

\blemnl{\cite{MR0322538}}{L4.41}
A function $\bv\ka\in\LinfTFG$ is the boundary value function of some function $\ka\in\NFDeFG$ if and only if there exists a scalar inner function $\ups\in\SFDs$ such that $\bv\ups\,\bv\ka^\ad\in\LinfpTGF$.
\elem
\bproof
Clearly, it suffices to consider the case $\bv\ka\in\CMTFG$.

Let $\bv\ups\,\bv\ka^\ad\in\LinfpTGF$, where $\ups\in\SFDs$ is inner.
If $\rho\in\SFDGF$ is the function such that $\bv\rho=\bv\ups\,\bv\ka^\ad$, then $\ka\defeq\mu/\eta$, $\mu\in\SFDeFG$, $\eta\in\SFDes$, where
\[
\mu\rk{\ze}\defeq\rho^\ad\rk{1/\ko\ze},
\qquad\eta\rk{\ze}\defeq\ko{\ups\rkb{1/\ko\ze}},
\qquad\ze\in\De.
\]

Conversely, let $\ka\defeq\mu/\eta$, $\mu\in\SFDeFG$, $\eta\in\SFDes$.
Denoting $\nu\rk{\ze}\defeq\ko{\eta\rk{1/\ko\ze}}$, $\ze\in\D$, we obtain $\bv\nu\,\bv\ka^\ad\in\LinfpTGF$ since $\rho\in\SFDGF$, where
\[
\rho\rk{\ze}\defeq\mu^\ad\rk{1/\ko\ze}
\rkb{=\nu\rk{\ze}\ka^\ad\rk{1/\ko\ze}},
\qquad\ze\in\D.
\]
If $\nu=\nu_\mri\nu_\mro$ is the inner-outer factorization of $\nu$ and $\ups\defeq\nu_\mri$, then $\bv\ups\,\bv\ka^\ad\in\LinfpTGF$ since $\bv\nu\,\bv\ka^\ad=\rk{\bv\ups\,\bv\ka^\ad}\rk{\bv{\nu_\mro}\IG}$ and $\clo{\bv{\nu_\mro}L_+^2\rk{\cG}}=L_+^2\rk{\cG}$.
\eproof

The concept of pseudocontinuation of each other for two scalar meromorphic functions was introduced by H.~S.~Shapiro in \cite{MR0241614} (see also \cite{MR0270196}).
For operator-valued functions we consider only the case when functions belong to the Schur and Nevanlinna classes on $\D$ and $\De$, respectively.

\bdefnl{D4.42}
Let $\te\in\SFDFG$.
A function $\ka\in\NFDeFG$ is called \emph{a bounded type pseudocontinuation of $\te$ to $\De$} if $\bv\ka=\bv\te$ holds true $\bv\lambda$\nbd-almost everywhere on $\T$.
\edefn

If such a pseudocontinuation exists, then, by the Luzin--Privalov theorem, it is unique.

It follows from \rlem{L4.41} that $\te\in\SFDFG$ possesses a bounded type pseudocontinuation to $\De$ if and only if there exists a scalar inner function $\ups\in\SFDs$ such that $\bv\ups\,\bv\te^\ad\in\LinfpTGF$.

\bdefnnl{\zitaa{SN}{\cch{V}}}{D4.43}
A Schur function $\te\in\SFDFG$ is said to possess \emph{a scalar multiple $\dl\in\SFDs$} if $\dl\not\equiv0$ and there exists a Schur function $\nu\in\SFDGF$ such that
\[
\te\rk{\ze}\nu\rk{\ze}=\dl\rk{\ze}\IG,
\qquad\nu\rk{\ze}\te\rk{\ze}=\dl\rk{\ze}\IF,
\qquad\ze\in\D.
\]
\edefn

\blemnl{\cite{MR0322538}}{L4.44}
A two-sided inner function $\te\in\SFDFG$ has a bounded type pseudocontinuation to $\De$ if and only if it has a scalar multiple.
\elem

\bcorl{C4.45}
Any two-sided inner matrix-valued function $\te\in\SFDFG$ has a bounded type pseudocontinuation to $\De$.
\ecor
\bproof
Denoting $\dl\defeq\det\te\in\SFDs$ the scalar inner function and $\nu\defeq\dl\te^\inv\in\SFDGF$ the matrix-valued inner function, we obtain
\[
\te\nu=\dl\IG,
\qquad\nu\te=\dl\IF.\qedhere
\]
\eproof

On the other hand, in the infinite-dimensional case, there exist two-sided inner functions that have no bounded type pseudocontinuation to $\De$.
If $\dim\cF=\infty$, such a function is $\te\in S\rk{\D;\ek{\cF}}$ defined by the matrix
\[
\te\rk{\ze}\defeq\diag\rk{1,\ze,\ze^2,\dotsc},
\qquad\ze\in\D,
\]
with respect to some orthonormal basis of $\cF$.

The problem of the existence of a bounded type pseudocontinuation of $\te\in\SFDFG$ to $\De$ is closely interrelated with the problem of the Darlington synthesis for $\te$.
The latter was partially considered in \rsubsec{s4.6} for the regular case.

The following fundamental result in this direction is due to D.~Z.~Arov \cite{MR0428098} and, independently, R.~G.~Douglas and J.~W.~Helton \cite{MR0322538}.

\bthml{T4.46}
If a Schur operator function $\te\in\SFDFG$ possesses a bounded type pseudocontinuation to $\De$, then it has a two-sided inner extension.
\ethm

As it is easy to see from the proof of this theorem in \cite{MR0322538}, there are regular extensions among all two-sided inner ones of $\te$.

\bcorl{C4.47}
If $\te\in\SFDFG$ possesses a bounded type pseudocontinuation to $\De$, then its defect function $\vphi\in\SFD{\cF}{\cGo}$ and $\ast$\nbd-defect function $\psi\in\SFD{\cFo}{\cG}$ do as well.
\ecor
\bproof
If $\Xi\in\SFD{\scF\oplus\cF}{\scG\oplus\cG}$ is a two-sided inner extension of form \eqref{E4.41}, then the factorization $\bv{\te_{12}}^\ad\bv{\te_{12}}=\Pi^2$ holds true, where $\Pi\in CM\rk{\T;\ek{\cF}}$ is defined in \eqref{E4.12}.
From \rthmss{T4.6}{T4.8} it follows that the equality $\bv\vphi^\ad\bv\vphi=\IF-\bv\te^\ad\bv\te$ is valid.
If $\ups\in\SFDs$ is an inner function which according to \rlem{L4.41} satisfies the condition $\bv\ups\,\bv\te^\ad\in\LinfpTGF$, then $\bv\ups\,\bv\vphi^\ad\bv\vphi\in L_+^\infty\rk{\T;\ek{\cF}}$.
Since $\vphi$ is outer (see \rcor{C4.9}), we obtain $\bv\ups\,\bv\vphi^\ad\in\LinfpT{\cGo}{\cF}$.
Hence, by \rlem{L4.41}, $\vphi$ possesses a bounded type pseudocontinuation to $\De$.

For $\psi$ the statement follows from the fact that $\te$ and $\te^\sch$ have such a property or not simultaneously.
\eproof

Note that a Schur function $\Xi\in\SFD{\scF\oplus\cF}{\scG\oplus\cG}$ of form \eqref{E4.41} possesses a bounded type pseudocontinuation to $\De$ if and only if each of its entries $\te_{ij}$ ($i,j=1,2$) does.
Therefore, in the general case, the converse to \rthm{T4.46} is false since there exist two-sided inner operator functions which have no bounded type pseudocontinuation to $\De$.

Note that for a matrix-valued function $\te\in\SFDFG$ possessing a bounded type pseudocontinuation to $\De$, among its two-sided inner extensions there are matrix-valued ones (see, \teg, \cite{MR0322538} or \cite{BDFK}).
Therefore, for the matrix case, \rcor{C4.45} implies the following strengthening of \rthm{T4.46}.

\bthmnl{\cite{MR0322538}}{T4.48}
A matrix-valued Schur function $\te\in\SFDFG$ possesses a bounded type pseudocontinuation to $\De$ if and only if it has a two-sided inner matrix-valued extension.
\ethm

In the paper \cite{MR2013544}, D.~Z.~Arov generalized the concept of bounded type pseudocontinuability to $\De$ introducing \emph{the class} $V\tilde P$ ($B\tilde\Pi$ in Russian).
Here we give the definition of it in an equivalent form.

\bdefnl{D4.49}
A Schur function $\te\in\SFDFG$ is said to belong to the class $V\tilde P$ if there exist two-sided inner functions $\om\in\SFD{\cF}{\cN}$ and $\ups\in\SFD{\cM}{\cG}$ such that
\beql{E4.57}
\bv\om\,\bv\te^\ad\in\LinfpT{\cG}{\cN},
\qquad\bv\te^\ad\bv\ups\in\LinfpT{\cM}{\cF}.
\eeq
\edefn

Taking into account \rlem{L4.41}, we see that the class of Schur functions admitting a bounded type pseudocontinuation to $\De$ (called by Arov as the class $VP$) is a subclass of the class $V\tilde P$ ($\om\defeq\gamma\IF$, $\ups\defeq\gamma\IG$, $\gamma\in\SFDs$ is inner).

Conditions \eqref{E4.57} can be interpreted as the existence of two pairs $\set{\mu_\ell,\eta_r}$ and $\set{\eta_\ell,\mu_r}$ of functions with the following properties:
\begin{enuin}
    \item $\mu_\ell\in\SFDe{\cN}{\cG}$, $\mu_r\in\SFDe{\cF}{\cM}$;
    \item $\eta_r\in\SFDe{\cN}{\cF}$, $\eta_\ell\in\SFDe{\cG}{\cM}$ are two-sided inner functions;
    \item their boundary value functions satisfy the equalities
    \beql{E4.58}
    \bv{\mu_\ell}\,\bv{\eta_r}^\inv
    =\bv{\eta_\ell}^\inv\bv{\mu_r}
    =\bv\te.
    \eeq
\end{enuin}
This can be easily checked if we choose $\rho_\ell\in\SFD{\cG}{\cN}$, $\rho_r\in\SFD{\cM}{\cF}$ such that $\bv{\rho_\ell}=\bv\om\,\bv\te^\ad$, $\bv{\rho_r}=\bv\te^\ad\bv\ups$ and set $\mu_\ell\rk{\ze}\defeq\rho_\ell^\ad\rk{1/\ko\ze}$, $\eta_r\defeq\om^\ad\rk{1/\ko\ze}$, $\eta_\ell\rk{\ze}\defeq\ups^\ad\rk{1/\ko\ze}$, $\mu_r\rk{\ze}\defeq\rho_r^\ad\rk{1/\ko\ze}$, $\ze\in\De$.
In view of equalities \eqref{E4.58}, it is natural to call the function $\mu_\ell$ (resp.\ $\mu_r$) as \emph{the left} (resp.\ \emph{right}) \emph{numerator} and the function $\eta_r$ (resp.\ $\eta_\ell$) as \emph{the right} (resp.\ \emph{left}) \emph{denominator of a generalized bounded type pseudocontinuation} $\set{\mu_\ell,\eta_r}$ (resp.\ $\set{\eta_\ell,\mu_r}$) \emph{of the function $\te$ to $\De$}.

\bthmnl{\cite{MR2013544}}{T4.50}
If a Schur function $\te\in\SFDFG$ belongs to the class $V\tilde P$, then it has a two-sided inner extension.
\ethm

As in the case of \rthm{T4.46}, an analysis of the proof of \rthm{T4.50} in \cite{MR2013544} shows that there are regular extensions among all two-sided inner ones of $\te$.

In the general case, the question of the validity of the converse theorem seems to be open.
Thus, the problem of describing the class of Schur operator functions which admit a two-sided inner extension (the Darlington problem) seems to be open as well.

The Darlington realization of Schur functions for the matrix case is studied in \cite{MR0357820}.
A description of two-sided inner extensions in terms of shifts contained in the corresponding fundamental contraction for matrix-valued Schur functions is considered in \cite{BDFK}.
One more approach to a full description of all $\rk{2\times2}$\nbd-matrix-valued inner extensions of a scalar Schur function can be found in \cite{MR2130935}.








\subsection*{Acknowledgment}
The second author would like to express special thanks to Professor Tatjana Eisner for her generous support.
He also thanks the Max Planck Institute for Human Cognitive and Brain Sciences, in particular 
Professor Christian Doeller, Dr Christina Schroeder, and Dr Sebastian Ziegaus.

The second author was also partially supported by the Volkswagen Foundation grant within the frameworks of the international project ``From Modeling and Analysis to Approximation''.

\end{document}